\documentclass[a4paper,10pt]{amsart}
\usepackage{etex}
\usepackage[plainpages=false]{hyperref}
\usepackage{amsfonts,amssymb,amsthm, mathrsfs, lscape}
\usepackage{verbatim}
\usepackage[T1]{fontenc}
\usepackage[utf8]{inputenc}
\usepackage{lmodern}
\usepackage{mathtools}
\usepackage{latexsym}
\usepackage{amssymb}
\usepackage{graphics}
\usepackage{graphicx}
\usepackage[space]{cite}
\usepackage{tikz}

\usetikzlibrary{intersections}
\usetikzlibrary{shapes.callouts}
\tikzset{connect/.style={thick,black} }
\usepackage{latexsym}

\usepackage{amsmath}

\usepackage{rawfonts}
\usepackage[all]{xy}
 \usepackage[usenames,dvipsnames]{pstricks}
 \usepackage{epsfig}
 \usepackage{pst-grad} 
 \usepackage{pst-plot} 
 \usepackage[space]{grffile} 
 \usepackage{etoolbox} 
 \makeatletter 

\renewcommand{\Re}{{\operatorname{Re}\,}}
\renewcommand{\Im}{{\operatorname{Im}\,}}

\renewcommand{\epsilon}{\varepsilon}

\newcommand{\PP}{{\mathbb P}}

\newcommand{\C}{{\mathbb C}}

\newcommand{\CP}{\C\PP}

\newcommand{\dbar}{\bar\partial}
\newcommand{\ddbar}{\partial\bar\partial}

\newcommand{\se}{{\operatorname{Se}}}

\newcommand{\bcal}{\mathcal{B}}
\newcommand{\ccal}{\mathcal{C}}
\newcommand{\dcal}{\mathcal{D}}

\newcommand{\gcal}{\mathcal{G}}
\newcommand{\hcal}{\mathcal{H}}
\newcommand{\ical}{\mathcal{I}}

\newcommand{\kcal}{\mathcal{K}}

\newcommand{\mcal}{\mathcal{M}}
\newcommand{\ncal}{\mathcal{N}}
\newcommand{\ocal}{\mathcal{O}}
\newcommand{\pcal}{\mathcal{P}}
\newcommand{\qcal}{\mathcal{Q}}

\newcommand{\tcal}{\mathcal{T}}
\newcommand{\ucal}{\mathcal{U}}
\newcommand{\vcal}{\mathcal{V}}
\newcommand{\wcal}{\mathcal{W}}
\newcommand{\xcal}{\mathcal{X}}
\newcommand{\zcal}{\mathcal{Z}}

\def \xmd {X\backslash D}

\newtheorem{theorem}{{Theorem}}[section]

\newtheorem{corollary}[theorem]{{Corollary}}
\newtheorem{lem}[theorem]{{Lemma}}
\newtheorem{lemma}[theorem]{{Lemma}}

\newtheorem{proposition}[theorem]{{Proposition}}

\newtheorem{remark}[theorem]{Remark}

\theoremstyle{definition}
\newtheorem{definition}[theorem]{Definition}

\numberwithin{equation}{section}

\def \C {\mathbb C}

\def \ke {\text{KE}}

\def \Ric {\text{Ric}}

\def \Sing {\textbf{Sing}}

\title[projective embedding of stably degenerating family II]{Projective embedding of degenerating family of K\"ahler-Einstein manifolds with normal crossing limit}

\author{Jingzhou Sun}
\thanks{}
\address{Department of Mathematics, Shantou University, Shantou City, Guangdong Province 515063, China}

\email{jzsun@stu.edu.cn}

\begin{document}

\begin{abstract}
    We study the Bergman embeddings of degenerating families of K\"{a}hler-Einstein manifolds of negative curvature. We show that when the central fiber has only simple normal crossing singularities, we can construct orthonormal bases so that the induced Bergman embeddings converge to the Bergman embedding of the limit space together with multi-bubbles.
\end{abstract}

\maketitle

\tableofcontents

\section{Introduction}
This work is a sequel to our previous paper \cite{sun2024projective}. To avoid unnecessary repetition, we refer readers to \cite{sun2024projective} for background and motivation.

Let $D \subset X$ be a simple normal crossing divisor on a complex projective manifold such that $K_X + [D]$ is ample. Cheng-Yau, Kobayashi, Tian-Yau, and Bando (\cite{ChengYau2, Kobayashi, TianYau3, Bando}) have shown that the quasi-projective manifold $\xmd$ admits a unique complete Kähler-Einstein metric $\omega_{\text{KE}}$ (the Cheng-Yau metric) with finite volume and $\Ric(\omega_{\text{KE}}) = -\omega_{\text{KE}}$. This metric induces a Hermitian metric on $K_X$ restricted to $\xmd$.

A natural question arises: given a sequence of compact Kähler-Einstein manifolds degenerating to a variety equipped with the Cheng-Yau metric on its regular part, can one construct Bergman embeddings that converge to the Bergman embedding of the limit variety?

Recall that a degeneration of Kähler-Einstein manifolds is a holomorphic family $\pi: \xcal \to B$, where $B$ is the unit disk, such that:
\begin{itemize}
    \item The fibers $X_t = \pi^{-1}(t)$ are smooth for $t \neq 0$;
    \item Each $X_t$ for $t \neq 0$ admits a Kähler-Einstein metric.
\end{itemize}

In the case of negative Kähler-Einstein metrics, we have the following result:
\begin{theorem}[\cite{Leung1999DegenerationOK}, \cite{RuanWD1}]\label{thm-ruan}
    Let $\pi: \xcal \to B$ be a degeneration of Kähler-Einstein manifolds $\{X_t, g_{E,t}\}$ with $\Ric(g_{E,t}) = -g_{E,t}$. Assume that the total space $\xcal$ is smooth and the central fiber $X_0$ is the union of smooth normal crossing hypersurfaces $\{X_{0,i}\}_{0 \leq i \leq m}$ in $\xcal$ with ample dualizing line bundle $K_{X_0}$. Then the Kähler-Einstein metrics $g_{E,t}$ on $X_t$ converge to a complete Cheng-Yau Kähler-Einstein metric $g_{E,0}$ on $X_0 \setminus \mathrm{Sing}(X_0)$ in the sense of Cheeger-Gromov: there exists an exhaustion of compact subsets $F_\beta \subset\subset X_0 \setminus \mathrm{Sing}(X_0)$ and diffeomorphisms $\phi_{\beta,t}: F_\beta \to X_t$ such that:
    \begin{itemize}
        \item[(1)] $X_t \setminus \bigcup_{\beta=1}^\infty \phi_{\beta,t}(F_\beta)$ is a finite union of submanifolds of real codimension $1$;
        \item[(2)] For each fixed $\beta$, $\phi_{\beta,t}^* g_{E,t}$ converges to $g_{E,0}$ on $F_\beta$ in $C^2$-topology as $t \to 0$.
    \end{itemize}
\end{theorem}

Tian originally proved this theorem in \cite{Tian1993} under the additional assumption that no three of the divisors $X_{0,i}$ have nonempty intersection. Leung-Lu \cite{Leung1999DegenerationOK} generalized Tian's result to a broader setting, and Ruan \cite{RuanWD1} provided an alternative proof of Theorem~\ref{thm-ruan}. For simplicity, we do not state the more general result of Leung-Lu here. Ruan further extended Theorem~\ref{thm-ruan} to the toroidal case in \cite{RuanWD2}.

This article is set in the context of Theorem~\ref{thm-ruan}, and represents a special case in the study of Bergman embeddings for "collapsing" families of Kähler-Einstein manifolds, as initiated in \cite{sun2024apde}.

The determinant $\omega_t^n$ of the Kähler form $\omega_t$ associated to the Kähler-Einstein metric defines a Hermitian metric $h_t$ on $K_{X_t}$. By a slight abuse of notation, we also denote by $h_t$ the induced metric on $kK_{X_t}$.
Let $\hcal_{t,k}$ denote the Bergman space of $L^2$-integrable holomorphic sections $s \in H^0(X_t, kK_{X_t})$, equipped with the Hermitian inner product
\[
\langle s_1, s_2 \rangle = \int_{X_t} (s_1, s_2)_{h_t} \, \omega_t^n,
\]
for $s_1, s_2 \in H^0(X_t, kK_{X_t})$. For simplicity, we will write $K_{X_t}$ as $K_t$.

Since $X_t$ is of general type for $t \neq 0$, Siu's invariance of plurigenera theorem \cite{siu1} implies that $\dim H^0(X_t, kK_t)$ is independent of $t$ for $t \neq 0$. Let $N_k = \dim H^0(X_t, kK_t)$, so for $t \neq 0$, $\dim \hcal_{t,k} = N_k$. For $t \neq 0$, $K_t$ is naturally isomorphic to $K_\xcal|_{X_t}$ via tensoring with $dt$. Similarly, on the regular part of $X_0$, $K_\xcal$ is naturally isomorphic to the canonical bundle.

For convenience, we denote $K_\xcal$ by $L$. Then the pushforward sheaf $\pi_* \ocal(kL)$ is a coherent sheaf on $B$, which is locally free over $B^*$. Moreover, for sufficiently large $k$, $\pi_* \ocal(kL)$ is locally free on all of $B$. For completeness, we include a proof of this claim at the end of the article.

Let $\omega_0$ denote the Kähler form on $X_0' = X_0 \setminus \mathrm{Sing}(X_0)$ corresponding to the complete Cheng-Yau metric. Then $\omega_0^n$ defines a Hermitian metric $h_0$ on $K_{X_0'}$.
Let $\hcal_{0,k}$ denote the Bergman space of $(h_0, \omega_0^n)$-$L^2$-integrable holomorphic sections of $kK_{X_0'}$, and let $\hcal_{0,i,k}$ denote the Bergman space of $L^2$-integrable holomorphic sections of $kK_{X_{0,i}}$ on $X_{0,i} \setminus \mathrm{Sing}(X_0)$.

Since the Cheng-Yau metric is of Poincaré type, the argument at the beginning of Section~\ref{sec-3} shows that $\hcal_{0,i,k}$ can be identified with a subspace of $H^0(X_{0,i}, kL)$, and each $s \in \hcal_{0,i,k} \subset H^0(X_{0,i}, kL)$ must vanish along all divisors at infinity.
Furthermore, we can identify $\hcal_{0,i,k}$ as the subspace of $\hcal_{0,k}$ consisting of sections that vanish on all $X_{0,j}$ except when $j = i$. Therefore,
\[
\hcal_{0,k} = \bigoplus_i \hcal_{0,i,k}.
\]
Clearly, for $i \neq j$, the spaces $\hcal_{0,i,k}$ and $\hcal_{0,j,k}$ are mutually orthogonal.

For each subset $I \subset \{0, \ldots, m\}$, denote $X_{0,I} = \bigcap_{i \in I} X_{0,i}$ and $$D_I = \bigcup_{\{J \mid I \subset J,\, |J \setminus I| = 1\}} X_{0,J}.$$ We also write $|I|$ for the cardinality of $I$.

Let $H^0_0(X_{0,I}, kL)$ denote the subspace of $H^0(X_{0,I}, kL)$ consisting of sections that vanish along $D_I$. Then we have
\begin{equation}
    H^0(X_0, kL) = H^0(\mathrm{Sing}(X_0), kL) \oplus \bigoplus_i H^0_0(X_{0,i}, kL).
\end{equation}
Let $n_k = \dim \hcal_{0,k}$, $n_{i,k} = \dim \hcal_{0,i,k}$, and $n_{I,k}$ the dimension of $H^0_0(X_{0,I}, kL)$. Then
\[
N_k = n_k + \sum_{|I| \geq 2} n_{I,k}.
\]
Thus, we have a decomposition $\C^{N_k} = \C^{n_k} \oplus \bigoplus_{|I| \geq 2} \C^{n_{I,k}}$, which induces natural inclusions of projective spaces:
\[
I_0: \CP^{n_k-1} \hookrightarrow \CP^{N_k-1}, \qquad
I_I: \CP^{n_{I,k}-1} \hookrightarrow \CP^{N_k-1}, \quad |I| \geq 2.
\]
For $k$ sufficiently large, a basis of $\hcal_{0,k}$ induces an embedding
\[
\Phi_{0,k}: \tilde{X}_0 \to \CP^{n_k-1},
\]
where $\tilde{X}_0 = \bigsqcup_{1 \leq i \leq l} X_{0,i}$ is the disjoint union of the components of $X_0$, serving as its normalization. We denote by $\Phi_{0,i,k}$ the restriction of $\Phi_{0,k}$ to $X_{0,i}$.

Let $\ical = \{ I \subset \{1, \ldots, m\} \mid X_{0,I} \neq \emptyset \}$, and set $\kappa = \max_{I \in \ical} |I|$.

\begin{theorem}\label{thm-main}
Under the setting of Theorem~\ref{thm-ruan}, for $k$ sufficiently large, there exist orthonormal bases $\{s_{i,l}^0\}_{1 \leq l \leq n_{i,k}}$ for $\hcal_{0,i,k}$ for each $i$, such that for any sequence of points in $B^*$ converging to $0$, there is a subsequence $\{t_u\}_{u=0}^\infty$ and orthonormal bases $\{s_l^{t_u}\}_{1 \leq l \leq N_k}$ for $\hcal_{t_u,k}$ with the following property: the images of the Bergman embeddings
\[
\Phi_{t_u,k}: X_{t_u} \to \CP^{N_k-1}
\]
induced by $\{s_l^{t_u}\}_{1 \leq l \leq N_k}$ converge, as $u \to \infty$, to a subvariety $Y$ described as follows:
\begin{itemize}
    \item[(a)] The irreducible components of $Y$ can be grouped into $\kappa$ collections of smooth subvarieties, $Y = \bigcup_{1 \leq i \leq \kappa} Y_i$, where $Y_1$ is the image of $I_0 \circ \Phi_{0,k}$, with $\Phi_{0,k}$ the Bergman embedding of $\tilde{X}_0$ induced by $\{s_{i,l}^0\}$.
    \item[(b)] For each $i \geq 2$, $Y_i = \bigcup_{I \in \ical,\, |I| = i} Y_I$, where $Y_I = \bigcup_{j \in I} Y_{I,j}$, and each $Y_{I,j}$ is a $(\CP^1)^{i-1}$-bundle over $X_{0,I}$. For $j_1, j_2 \in I$ with $j_1 \neq j_2$, the intersection $Y_{I,j_1} \cap Y_{I,j_2}$ is a $(\CP^1)^{i-2}$-bundle over $X_{0,I}$.
\end{itemize}
\end{theorem}
\begin{remark}
    \begin{itemize}
        \item The intersection relations among the components $Y_{I,j}$ can be described in terms of the intersection relations between the $X_{0,J}$'s.
        \item Each $Y_{I,j}$ can be realized as the image of a certain completion of the normal bundle of $X_{0,I}$ in $X_{0,J}$ for all $J \subset I$ with $|I| - |J| = 1$, as will be made precise in Theorem~\ref{thm-y-j}.
    \end{itemize}
\end{remark}

In \cite{sun2024projective}, we proved the theorem under the additional assumption (as in \cite{Tian1993}) that no three of the divisors $X_{0,i}$, $0 \leq i \leq m$, have nonempty intersection. Removing this assumption leads to a much more intricate and challenging situation, though the core ideas remain similar.

We also established in \cite{sun2024projective} an estimate for the Bergman kernel function of $(X_t, kL)$, since constructing orthonormal bases (as required in the proof of our main theorem) is closely related to the analysis of the Bergman kernel. Here, we state a generalization of that result:

\begin{theorem}\label{thm-2}
    Let $\rho_{t,k}$ denote the Bergman kernel of $\hcal_{t,k}$. Define
    \[
    \lambda_u(t,k) = \max_{x \in X_t} \rho_{t,k}(x), \qquad
    \lambda_l(t,k) = \min_{x \in X_t} \rho_{t,k}(x).
    \]
    Then for $k$ sufficiently large, there exist constants $c_k > 0$, $c_k' > 0$, and $c_k'' > 0$ such that
    \[
    c_k < \frac{\lambda_u(t,k)}{|\log |t||^{\kappa-1}} < c_k', \qquad
    \lambda_l(t,k) < c_k'' \left(|\log |t||^{-2k+1} a_t^{2k}\right)^{\kappa-1},
    \]
    where $a_t = \log |\log |t||$.
\end{theorem}

As will be seen in the proof, the maximum is attained at points "closest" to the strata $X_{0,I}$ with $|I| = \kappa$. We also remark that the estimate for $\lambda_l(t,k)$ is not optimal.

This can be compared with Zhou's result in \cite{ZHOU2024109514}, where he showed that under the "non-collapsing" condition and a lower bound on Ricci curvature, if a sequence of pointed complete polarized Kähler manifolds converges in the Gromov-Hausdorff sense, then a subsequence of the polarizations also converges and the Bergman kernels converge to the Bergman kernel of the polarization on the limit space. In our setting, the convergence of the Bergman kernels in the non-collapsed part is clear; the interesting phenomenon occurs on the collapsing part. We hope this result can be generalized to more general degenerating families of Kähler-Einstein manifolds. It would also be very interesting to study the Bergman embeddings in other "collapsing" cases beyond the negative curvature case; see, for example, \cite{MR4322391}, \cite{songzhang2024Acta}, \cite{Biquard2022}, and \cite{songzhang2024csdegcollapsing}.

\

We briefly comment on the proofs. Theorem \ref{thm-ruan} concerning the convergence of the Kähler-Einstein metrics addresses the "non-collapsing part," while the main difficulties lie in the "collapsing part." We rely on Zhang's results in \cite{2015Collapsing} for the analysis of the collapsing region. Fortunately, the "neck" region between the non-collapsing and collapsing parts does not present substantial difficulties for our arguments. It would be interesting to see further developments in overcoming the complications caused by the neck region. In our setting, the collapsing part is concentrated near the strata $X_{0,I}$ with $|I| \geq 2$. When $|I| > 2$, the presence of corners of varying dimensions around $X_{0,I}$ makes the analysis more intricate. A natural strategy is to proceed inductively from lower to higher codimension, first handling the case when $D_I$ is empty. In our proofs of the intermediate theorems, especially Theorems \ref{thm-inner-sections-prop} and \ref{thm-almost-orthon}, we often employ this approach. Furthermore, in \cite{sun2024projective}, to describe the $Y_i$ for $i = 2$, it sufficed to show that, in the limit, linear $\CP^1$'s—referred to as bubbles—connecting different images of $X_{0,I}$ appear. In contrast, in the more general setting of this article, we must describe higher-dimensional bubbles, namely $(\CP^1)^l$'s, more explicitly, since these bubbles are not determined by $2^l$ points alone.

It will become clear that Theorem \ref{thm-2} is a byproduct of the proof of the main theorem.

\
The structure of this article is as follows. In Section~\ref{sec-2}, we recall general estimates for the Kähler-Einstein metrics on the general fibers and review Zhang's results on the collapsing part of $X_t$. With these estimates in hand, we carry out some preparatory calculations. Section~\ref{sec-3} is devoted to the construction of global sections on the general fibers and to the proof of Theorem~\ref{thm-main}, omitting for the moment the proofs of the intermediate results (Theorems~\ref{thm-inner-sections-prop}, \ref{thm-almost-orthon}, and \ref{thm-almost-Bergman-Embedding}). In Section~\ref{sec-4}, we prove Theorem~\ref{thm-inner-sections-prop}. Section~\ref{sec-5} contains the proof of Theorem~\ref{thm-almost-orthon}. In Section~\ref{sec-almost-bergman}, we prove Theorem~\ref{thm-almost-Bergman-Embedding}, thereby completing the proof of our main theorem, and then establish Theorem~\ref{thm-2}.

\medskip

\textbf{Acknowledgements.} The author would like to thank Professor Song Sun for many helpful discussions.

\medskip

\textbf{Declaration.} The author did not use ChatGPT or any other AI tools.

\section{Setting-up}\label{sec-2}
Although we will not explicitly use it, the following construction helps to visualize the intersection relations among the $X_{0,I}$'s via a cell complex.

We begin with $m+1$ vertices, labelled $0$ through $m$. For each pair $(v_i, v_j)$ such that $X_{0,i} \cap X_{0,j} \neq \emptyset$, we attach an edge $e(i,j)$ connecting $v_i$ and $v_j$. For each triple $(v_a, v_b, v_c)$ with $X_{0,a} \cap X_{0,b} \cap X_{0,c} \neq \emptyset$, we attach a 2-cell to the triangle $T(a,b,c)$ bounded by $e(a,b)$, $e(a,c)$, and $e(b,c)$. Proceeding inductively, for every nonempty intersection of $k$ components, we attach a $(k-1)$-cell corresponding to the $k$-tuple of vertices. Continuing this process up to the minimal-dimensional strata, we obtain a cell complex $\ccal$ encoding the intersection pattern of the components $X_{0,i}$, $0 \leq i \leq m$. Note that $\kappa = \dim \ccal + 1$.

\medskip

\textbf{Terminology.} When we focus on a subset $I \subset \{0, \ldots, m\}$, we will often write $l = |I| - 1$.

\subsection*{Zhang's results on the collapsing part}
For each $i$, fix a section \[S_i \in H^0(\xcal, [X_{0,i}])\] vanishing along $X_{0,i}$, and require that $\prod_{i=0}^{m} S_i = t$. Let $\|\cdot\|_i$ be a smooth Hermitian metric on $[X_{0,i}] \to \xcal$. By rescaling if necessary, we may assume $\|S_i\|_i \leq \epsilon \ll 1$.

The assumption in Theorem~\ref{thm-ruan} implies that $K_\xcal$ is ample on $\xcal$. Fix a smooth Hermitian metric $h$ on $K_\xcal$ with positive curvature, and let $\omega = \frac{\sqrt{-1}}{2\pi} \Theta_h$ be the associated Kähler form. Let $V$ be a smooth volume form on $\xcal$. For $t \neq 0$, define the volume form
\[
V_t = V \otimes (\sqrt{-1} dt \wedge d\bar{t})^{-1}.
\]
Set
\[
\alpha_i = \log \|S_i\|_i^2, \qquad \chi_t = (\log |t|^2)^2 \prod_{i=0}^{m} \alpha_i^{-2},
\]
and
\[
\tilde{\omega}_t = \sqrt{-1} \partial \bar{\partial} \log (\chi_t V_t).
\]

Zhang proved the following key estimate:

\begin{proposition}[Proposition 3.1 in \cite{2015Collapsing}]\label{prop-ke-t}
Let $\varphi_t$ be the unique solution to the Monge-Ampère equation
\begin{equation}\label{eqn-ma-t}
    (\tilde{\omega}_t + \sqrt{-1} \partial \bar{\partial} \varphi_t)^n = e^{\varphi_t} \chi_t V_t,
\end{equation}
and set $\omega_t = \tilde{\omega}_t + \sqrt{-1} \partial \bar{\partial} \varphi_t$. Then there exist constants $C_1, C_2 > 0$ independent of $t$ such that
\[
|\varphi_t| \leq C_1, \qquad C_2^{-1} \tilde{\omega}_t \leq \omega_t \leq C_2 \tilde{\omega}_t.
\]
\end{proposition}

Let $\omega_t^o = \sqrt{-1}\ddbar \log V_t$. Then,
\begin{eqnarray*}
    \tilde{\omega}_t &=& \sqrt{-1}\ddbar \log (\chi_t V_t) \\
    &=& \omega_t^o + \sqrt{-1}\ddbar \log \chi_t \\
    &=& \omega_t^o + 2\sum_{i=0}^{m} \left( \frac{\Ric(\|\cdot\|_i)}{\alpha_i} + \sqrt{-1} \frac{\partial \alpha_i \wedge \bar{\partial} \alpha_i}{\alpha_i^2} \right)\Big|_{X_t}
\end{eqnarray*}
We also have a complete Kähler metric on $X_{0,I}^o$:
\[
\tilde{\omega}_{0,I} = \omega^o|_{X_{0,I}^o} + 2\sum_{i\notin I} \left( \frac{\Ric(\|\cdot\|_i)}{\alpha_i} + \sqrt{-1} \frac{\partial \alpha_i \wedge \bar{\partial} \alpha_i}{\alpha_i^2} \right)\Big|_{X_{0,I}^o}.
\]

Define
\[
B_I = \left\{ (x_1, \ldots, x_l) \in \mathbb{R}^l \;\middle|\; \sum_{j=1}^l x_j < 1,\; x_j > 0 \text{ for } 1 \leq j \leq l \right\},
\]
and
\[
B_{t,l} = \left\{ (x_1, \ldots, x_l) \in \mathbb{R}^l \;\middle|\; x_j > \frac{\log \epsilon}{\log|t|},\; 1 - \sum_{j=1}^l x_j > \frac{\log \epsilon}{\log|t|} \right\}.
\]
For any subset $J \subset B_I$, denote by $G_J \subset \mathbb{C}^l$ the strip $J \times \sqrt{-1}\mathbb{R}^l$. Let $w = (w_1, \ldots, w_l)$ be the complex coordinates on $G_J$, and set $x_j = \Re w_j$.

Let $U$ be a neighborhood of $p \in X_{0,I}^o$ with local coordinates $z = (z_0, z_1, \ldots, z_n)$ such that $t = z_0 z_1 \cdots z_l$, where $l = |I| - 1$. Then $(z_{l+1}, \ldots, z_n)$ are local coordinates for $U \cap X_{0,I}$. For any compact subset $K \subset B_I$, we define a covering map
\[
P_{t,l} : G_K \times (U \cap X_{0,I}) \to U \cap X_t
\]
by
\[
P_{t,l}(w_1, \ldots, w_l, z_{l+1}, \ldots, z_n) = (e^{w_1 \log|t|}, \ldots, e^{w_l \log|t|}, z_{l+1}, \ldots, z_n).
\]

In \cite{2015Collapsing}, Zhang established the following results.

\begin{lem}[Lemma 3.2 in \cite{2015Collapsing}]\label{lem-zhang-3.2}
    Let $K \subset B_I$ be a compact subset such that $K \subset B_t$ for $|t| \ll 1$. On $G_K \times (U \cap D_\alpha)$, as $t \to 0$, we have:
    \begin{itemize}
        \item 
        \[
        P_t^*\chi_t V_t \to V_0' = \frac{1}{4(1-\sum_{j=1}^{l} x_j)^2} \prod_{j=1}^{l} \frac{dw_j \wedge d\bar{w}_j}{4x_j^2} \wedge V_I,
        \]
        in the $C^\infty$ sense, where $V_I$ is a smooth volume form on $U \cap X_{0,I}$.
        \item 
        \[
        P_t^*\tilde{\omega}_t \to \omega_{U,I}^o + \frac{\sqrt{-1}}{2} \left( \sum_{j=1}^{l} \frac{dw_j \wedge d\bar{w}_j}{x_j^2} + \frac{\sum_{i,j=1}^{l} dw_i \wedge d\bar{w}_j}{(1-\sum_{j=1}^{l} x_j)^2} \right),
        \]
        in the $C^\infty$ sense, where $\omega_{U,I}^o$ is the pullback of the complete Kähler metric $\tilde{\omega}_{0,I}$ on $U \cap X_{0,I}$.
    \end{itemize}
\end{lem}

When comparing different coordinate patches $\{U_i\}$ for $X_{0,I}^o$, it is more precise to denote the volume form in the lemma by $\vcal_{I,i}$ for $U_i$.

\subsubsection*{Appropriate Coordinate Patches}
Let $X^o_{0,I} = X_{0,I} \setminus D_I$. Around any point $p \in X^o_{0,I}$, we can choose a neighborhood $U$ with local coordinates $z = (z_0, z_1, \ldots, z_n)$ such that $t = z_0 z_1 \cdots z_l$, where $l = |I| - 1$. Thus, $(z_{l+1}, \ldots, z_n)$ serve as local coordinates for $X_{0,I}$, and the sets $\{z_i = 0\}$ for $i \leq l$ correspond to the divisors $X_{0,a}$ in $I$, ordered so that if $X_{0,a} \cap U = \{z_i = 0\}$ and $X_{0,b} \cap U = \{z_j = 0\}$ with $a < b$, then $i < j$.

We can further require that
\[
\|dz_0(p)\| = \|dz_1(p)\| = \cdots = \|dz_l(p)\|
\]
with respect to the norm defined by $\omega$, and that
\[
\left\langle \frac{\partial}{\partial z_i}(p), \frac{\partial}{\partial z_j}(p) \right\rangle = \delta^i_j, \quad \text{for } i > l,\, j > l.
\]
By shrinking $U$ if necessary, we may also assume
\[
\frac{1}{2} < \|dz_i(p)\| < 2, \quad \text{for } i > l,
\]
and
\[
\frac{\max_{i \leq l} \|dz_i(p)\|}{\min_{i \leq l} \|dz_i(p)\|} < 2.
\]
Moreover, we can arrange that $U$ is a polydisc in these coordinates:
\[
U = \{ z \mid |z_i| < R_i,\, 0 \leq i \leq n \}.
\]
We will refer to such a neighborhood, together with these coordinates, as an \emph{appropriate coordinate patch}.

\

For any $\delta > 0$, let $X_{0,I}^\delta$ denote the subset of $X_{0,I}^o$ consisting of points whose distance to $D_I$ is at least $\delta$. 
For fixed $\delta > 0$, we can cover $X_{0,I}^\delta$ by finitely many appropriate coordinate patches $\{U_i\}$. By the construction of appropriate coordinate patches, on the overlap of two such patches $(U_i, (z_0, \ldots, z_n))$ and $(U_j, (z'_0, \ldots, z'_n))$, we have
\[
\left| \log \frac{|z_a'|}{|z_a|} \right| < C, \quad a \leq l,
\]
for some constant $C > 0$. Thus, the corresponding real coordinates on $K$ satisfy $|x_a' - x_a| < \frac{C}{|\log |t||}$, so as $t \to 0$, $x_a' - x_a \to 0$. The same holds for the imaginary parts. Therefore, on overlaps, $\vcal_{I,j} = \vcal_{I,i}$, and the local volume forms glue together to define a global volume form $\vcal_I$ on $X_{0,I}^\delta$. Letting $\delta \to 0$, we obtain a global volume form, still denoted by $\vcal_I$, on $X_{0,I}^o$.

\begin{definition}
    Let $X$ be a compact complex manifold of dimension $n$, and let $D \subset X$ be a simple normal crossing divisor. A volume form $V$ on $X \setminus D$ is said to be of \emph{Poincar\'e type} near $D$ if, for any open neighborhood $U$ of a point $p \in D$ with coordinates $(z_1, \ldots, z_n)$ centered at $p$ such that $U \cap D = \cap_{1 \leq i \leq u} \{z_i = 0\}$, there exists a constant $C > 0$ such that on $U \setminus D$,
    \[
    C^{-1} < \frac{V}{\mu_p} < C,
    \]
    where $\mu_p = \frac{\bigwedge_{i=1}^n (dz_i \wedge d\bar{z}_i)}{\prod_{i=1}^u \left[ |z_i|^2 (\log |z_i|^2)^2 \right]}$.
\end{definition}

\begin{proposition}
    The volume form $\vcal_I$ is of Poincar\'e type near $D_I$.
\end{proposition}
\begin{proof}
    To illustrate the idea, we first verify the proposition near regular points of $D_I$.

    Let $J \subset \{1, \ldots, m\}$ with $J = I \cup \{\lambda\}$ for some $\lambda \notin I$ and $X_{0,J} \neq \emptyset$. Around a point $p \in X_{0,J}^o$, let $(U, (z_0, \ldots, z_n))$ be an appropriate coordinate patch for $X_{0,J}^o$ such that $\prod_{i=0}^{l+1} z_i = t$ and $U \cap X_{0,\lambda} = \{z_{l+1} = 0\}$. Using $(z_1, \ldots, z_n)$ as coordinates for $X_t \cap U$, we have
    \[
    V_t = \frac{(\sqrt{-1})^{n} \eta}{\left| \prod_{1 \leq j \leq l+1} z_j \right|^2} \bigwedge_{1 \leq j \leq n} (dz_j \wedge d\bar{z}_j),
    \]
    for some smooth bounded function $\eta$ on $U$.

    Similarly, $\chi_t$ can be written as
    \[
    \chi_t = \eta_{l+2} (\log |t|^2)^2 \prod_{j=0}^{l+1} (\log |z_j|^2 + \eta_j)^{-2},
    \]
    where each $\eta_j$ is a smooth bounded function on $U$.

    Thus,
    \[
    \chi_t V_t = \eta' \frac{(\log |t|^2)^2}{(\log |z_0|^2 + \eta_0)^2} \bigwedge_{j=1}^{l+1} \frac{dz_j \wedge d\bar{z}_j}{(\log |z_j|^2 + \eta_j)^2} \wedge \bigwedge_{j=l+2}^n (dz_j \wedge d\bar{z}_j),
    \]
    where $\eta'$ is another smooth bounded function on $U$.

    For any $q \in X_{0,I}^o \cap U$, choose a small neighborhood $U' \Subset U \setminus X_{0,\lambda}$. Let $z_0' = z_0 z_{l+1}$ and $z_i' = z_i$ for $i > 0$. Then $(z_0', \ldots, z_n')$ are coordinates on $U'$ with $\prod_{i=0}^l z_i' = t$. Note that $U'$ need not be an appropriate coordinate patch, but it suffices to define $P_{t,I}$.

    We then have
    \begin{align*}
        P_{t,I}^* (\chi_t V_t) &= \frac{\eta' \circ P_{t,I}}{(2x_0 + \frac{\eta_0}{\log |t|^2})^2} \bigwedge_{j=1}^l \frac{dw_j \wedge d\bar{w}_j}{(2x_j + \frac{\eta_j}{\log |t|^2})^2} \\
        &\quad \wedge \frac{dz_{l+1} \wedge d\bar{z}_{l+1}}{|z_{l+1}|^2 (\log |z_{l+1}|^2 + \eta_{l+1})^2} \wedge \bigwedge_{j=l+2}^n (dz_j \wedge d\bar{z}_j),
    \end{align*}
    where $x_0 = 1 - \sum_{j=1}^l x_j$. Letting $t \to 0$, it is clear that
    \[
    \vcal_I = \eta'' \frac{dz_{l+1} \wedge d\bar{z}_{l+1}}{|z_{l+1}|^2 (\log |z_{l+1}|^2 + \eta_{l+1})^2} \wedge \bigwedge_{j=l+2}^n (dz_j \wedge d\bar{z}_j),
    \]
    for some smooth function $\eta''$ on $U'$, bounded in terms of the bounds of $\eta'$. Thus, $\vcal_I$ is of Poincaré type near $X_{0,J}^o$.

    The same calculation applies to higher codimensional intersections, so the proposition follows.
\end{proof}

\begin{theorem}[Lemma 3.3 in \cite{2015Collapsing}]\label{the-zhang-3.3}
    Let $\varphi_t$ be the unique solution of \eqref{eqn-ma-t}, and set $\omega_t = \tilde{\omega}_t + \sqrt{-1}\partial\bar{\partial} \varphi_t$. For any sequence $t_k \to 0$, there exists a subsequence such that $\varphi_{t_k} \circ P_{t_k}$ converges in the $C^\infty$ sense on $K \times \sqrt{-1}\mathbb{R}^l \times (U \cap X_{0,I})$ to a function $\varphi_0$ satisfying the complex Monge-Ampère equation:
    \begin{equation}\label{eqn-ma-0}
        (\tilde{\omega}_0 + \sqrt{-1}\partial\bar{\partial} \varphi_0)^n = e^{\varphi_0} V_0'
    \end{equation}
    with $|\varphi_0| \leq C_3$, and $C_4^{-1} \tilde{\omega}_0 \leq \tilde{\omega}_0 + \sqrt{-1}\partial\bar{\partial} \varphi_0 \leq C_4 \tilde{\omega}_0$.
    
    Furthermore, $\varphi_0$ is independent of $\Im(w)$, i.e.,
    \[
    \varphi_0 = \varphi_0(x, z_{l+1}, \ldots, z_n).
    \]
\end{theorem}

We denote by $\varphi_{I,U}$ the function $\varphi_0$ in the theorem to emphasize its dependence on $U$ and $X_{0,I}$. Note that $\varphi_{I,U}$ also depends on the chosen good sequence $\{t_u\}$, but for simplicity, we suppress this in the notation.

By Lemma~\ref{lem-zhang-3.2} and Theorem~\ref{the-zhang-3.3}, we have
\[
e^{\varphi_{I,U}} V_0' = e^{\varphi_{I,U}} \frac{\vcal_I}{4(1-\sum_{j=1}^{l} x_j)^2} \wedge \bigwedge_{j=1}^{l} \frac{dw_j \wedge d\bar{w}_j}{4x_j^2}
\]
on $U_i \cap X_{0,I}$.

\begin{definition}
    A sequence $\{t_u\}$ converging to $0$ is called a \emph{good sequence} if, for every $I$, every $U$, and every compact $K \subset B_I$, the pullbacks $P_{t_u}^* \omega_{t_u}^n$ converge on $G_K \times (U \cap X_{0,I})$.
\end{definition}
Let $t_u \to 0$ be a good sequence. Consider the form
\[
e^{\varphi_{I,U}} \frac{\vcal_I}{4(1-\sum_{j=1}^{l} x_j)^2} \prod_{j=1}^{l} \frac{1}{4x_j^2}
\]
on $G_K \times (U \cap X_{0,I})$. If we change the coordinates $(z_{l+1},\ldots,z_n)$ on $U \cap X_{0,I}$ to $(z_{l+1}',\ldots,z_n')$, then the representation function of $\omega_t^n$ is multiplied by $|\det J|^2$, where $J$ is the Jacobian matrix $\left\{\frac{\partial z_a}{\partial z_b'}\right\}_{a,b \geq l+1}$. Since the image of $G_K \times (U \cap X_{0,I})$ under $P_t$ converges to $U \cap X_{0,I}$ as $t \to 0$, the limit of the representation function of $P_t^*\omega_t^n$ is also multiplied by $|\det J|_{U \cap X_{0,I}}|^2$. Therefore, the expression
\[
e^{\varphi_{I,U}} \frac{\vcal_I}{4(1-\sum_{j=1}^{l} x_j)^2} \prod_{j=1}^{l} \frac{1}{4x_j^2}
\]
glues together to define a smooth section of $\Pi^{-1}\Omega_{X_{0,I}}^{n-l,n-l}$ on $G_K \times X_{0,I}^\delta$, where $\Pi: G_K \times X_{0,I}^\delta \to X_{0,I}^\delta$ is the natural projection. Letting $\delta \to 0$, we obtain a smooth section on $G_K \times X_{0,I}^o$. By taking an increasing sequence of compact subsets $K_j \subset B_I$ with $\bigcup K_j = B_I$, we further obtain a smooth section on $G_{B_I} \times X_{0,I}^o$. In particular, the functions $\varphi_{I,U}$ glue together to define a smooth bounded function $\varphi_I$ on $B_I \times X_{0,I}^o$.

\subsection*{Identifications of line bundles}
Let $A_I$ denote the line bundle $\sum_{j \notin I} [X_{0,j}]$. By the adjunction formula,
\[
K_{\xcal}|_{X_{0,i}} = K_{X_{0,i}} - [X_{0,i}] = K_{X_{0,i}} + A_{\{i\}}.
\]
By induction on $|I|$, we have
\[
L|_{X_{0,I}} = K_{X_{0,I}} + A_I.
\]
Let $K_I$ denote the canonical bundle $K_{X_{0,I}}$.

If $J \subset I$ with $|J| = |I| - 1$, the normal bundle $N_{I,J}$ of $X_{0,I}$ in $X_{0,J}$ is simply the restriction of $[X_{0,I \setminus J}]$ to $X_{0,I}$.

Around an appropriate coordinate patch $(U,(z_0,\cdots,z_n))$ with $t = z_0\cdots z_l$, let $(z'_0,\cdots,z'_n)$ be another such coordinate system. Then $z_i' = F_i(z)z_i$ for $i \leq l$, where each $F_i$ is a nowhere-vanishing holomorphic function on $U$. On $X_{0,I} \cap U$, we have $dz_i' = F_i(z)dz_i$, so
\begin{align*}
    dz_0'\wedge\cdots\wedge dz_l' &= \prod_{i=0}^{l} F_i(z) \, dz_0\wedge\cdots\wedge dz_l \\
    &= dz_0\wedge\cdots\wedge dz_l,
\end{align*}
since the product $\prod_{i=0}^l F_i(z)$ restricts to $1$ on $X_{0,I}$. Thus, we obtain a well-defined explicit isomorphism from $K_I$ to $K_\xcal$, locally given by
\[
dz_{l+1}\wedge\cdots\wedge dz_n \mapsto dz_{0}\wedge\cdots\wedge dz_n
\]
on $X_{0,I}^o$. In particular, when $D_I$ is empty, this isomorphism is defined over all of $X_{0,I}$. When $D_I$ is nonempty, let $(U, (z_0,\cdots,z_n))$ be local coordinates around a point $p \in D_I$ such that $t = z_0\cdots z_{l'}$ for some $l' > l$ and $U \cap X_{0,I} = \{z_i = 0 \mid i \leq l\}$. Set $z_0' = z_0 \prod_{i=l+1}^{l'} z_i$ and $z_i' = z_i$ for $i > 0$; then $(z_0',\cdots,z_n')$ are local coordinates on $U \setminus D_I$ with $\prod_{i=0}^l z_i' = t$. The morphism
\begin{equation}\label{e-iso-mero}
    f \wedge_{i>l} dz_i' \mapsto f \wedge_{i\geq 0} dz'_i
\end{equation}
gives the isomorphism on $X_{0,I}^o$. In the original coordinates $(z_0,\cdots,z_n)$, this is written as
\[
f \wedge_{i>l} dz_i \mapsto f \left( \prod_{i=l+1}^{l'} z_i \right) \wedge_{i\geq 0} dz_i.
\]
Identifying $K_{X_{0,I}} + A_I$ as the sheaf of meromorphic sections of $K_{X_{0,I}}$ with pole order at most $1$ along $D_I$, the morphism in \eqref{e-iso-mero} defines an isomorphism from $K_{X_{0,I}} + A_I$ to $\kcal$. In the following, we will tacitly use this isomorphism.

On the other hand, let
\[
S_{\hat{I}} = \prod_{j\notin I} S_j,
\]
and fix an isomorphism from $K_I$ on $X_{0,I}^o$ to $K_I + A_I$ defined by $F_I: s \mapsto s \otimes S_{\hat{I}}$ for any local section $s$.

\subsection{General bounds on volume forms}

\subsubsection*{A fixed open cover}

We use the notation
\[
\frac{1}{2}U = \left\{ z \in U \mid |z_i| < \frac{1}{2} R_i,\, 0 \leq i \leq n \right\}.
\]
We now construct a collection of open sets $\ucal = \{ (U_\beta, z) \}_{\beta \in \Lambda}$ as follows. Begin with those sets $I$ such that $D_I$ is empty, i.e., the $I$-cell in $\ccal$ does not bound any higher-dimensional cell. Set $\ucal = \emptyset$. For each such $I$, since $X_{0,I}^o = X_{0,I}$ is compact, we can find finitely many open sets $\{ U_i^I \}$ so that $\{ \frac{1}{2} U_i^I \}$ covers $X_{0,I}$. We add these open sets $\{ \frac{1}{2} U_i^I \}$ to $\ucal$. If necessary, shrink these open sets so that if $I \neq J$, then $U_i^I \cap U_j^J = \emptyset$ for any $i$ and $j$. Let $\wcal$ denote the collection of such $I$-cells.

Next, for each $J$ such that $|J| = |I| - 1$ and the $J$-cell bounds a cell in $\wcal$, since $X_{0,J}^o \setminus \bigcup_{U \in \ucal} U$ is compact, we can find finitely many open sets $\{ U_i^J \}$ so that $\{ \frac{1}{2} U_i^J \}$ covers $X_{0,J}^o \setminus \bigcup_{U \in \ucal} U$. Add these open sets $\{ \frac{1}{2} U_i^J \}$ to $\ucal$ and update $\wcal$ to include these $J$-cells. Repeat this process inductively: at each step, for each cell $K$ of codimension one less than the previous, cover the remaining part of $X_{0,K}^o$ not already covered, and add the corresponding open sets to $\ucal$.

By continuing this process, we obtain a collection of open sets $\ucal = \{ (U_\beta, z) \}_{\beta \in \Lambda}$ as claimed. Clearly, $\bigcup_{U \in \ucal} U$ covers $\Sing(X_0)$. We also denote by $\ucal^I \subset \ucal$ the subset consisting of those $U$ centered at a point in $X_{0,I}^o$.

\medskip

Then for each $I$ with $|I| \geq 2$, there exists a constant $C$ such that for any $U_i, U_j \in \ucal^I$ with coordinates $(z_0, z_1, \ldots, z_n)$ and $(z'_0, z'_1, \ldots, z'_n)$ respectively, and $U_i \cap U_j \neq \emptyset$, we have
\[
C^{-1} |z'_i| < |z_i| < C |z'_i|, \quad \text{for } i \leq l.
\]
There also exists a constant $C'$ such that
\[
\frac{1}{C'} < \frac{V}{(\sqrt{-1})^{n+1} dz_0 \wedge d\bar{z}_0 \wedge \cdots \wedge dz_n \wedge d\bar{z}_n} < C'
\]
on every $U_j$. Let $\mu_I = \omega^{n-l}$ denote the volume form on $X_{0,I}$. Then there exists a constant $C''$ such that
\[
\frac{1}{C''} < \frac{\mu_I}{(\sqrt{-1})^{n-l} dz_{l+1} \wedge d\bar{z}_{l+1} \wedge \cdots \wedge dz_n \wedge d\bar{z}_n} < C''
\]
on every $U_j \cap X_{0,I}$ with $U_j \in \ucal^I$.

\medskip

Let $dz = \bigwedge_{j \leq n} dz_j$ and $dz_{\hat{i}} = (-1)^i \bigwedge_{j \leq n,\, j \neq i} dz_j$. Since
\[
dt = \sum_{i=0}^l \left( \prod_{j \leq l,\, j \neq i} z_j \right) dz_i,
\]
we have $dt \wedge dz_{\hat{i}} = \left( \prod_{j \leq l,\, j \neq i} z_j \right) dz$ for $i \leq l$. Thus, for $t \neq 0$,
\[
\frac{(\sqrt{-1})^{(n+1)^2} dz \wedge d\bar{z}}{\sqrt{-1} dt \wedge d\bar{t}} = \frac{(\sqrt{-1})^n}{\left| \prod_{j \leq l,\, j \neq i} z_j \right|^2} \bigwedge_{j \leq n,\, j \neq i} (dz_j \wedge d\bar{z}_j).
\]

Therefore, for each $I$, there exists a constant $C$ such that
\[
\frac{1}{C} < \frac{|\prod_{j \leq l,\, j \neq i} z_j|^2 V_t}{(\sqrt{-1})^n \bigwedge_{j \leq n,\, j \neq i} (dz_j \wedge d\bar{z}_j)} < C
\]
for each $i \leq l$, valid on every $U \in \ucal^I$.

Now, set
\[
\tilde{\chi}_t = \frac{(\log |t|^2)^2}{\prod_{i \in I} \alpha_i^2}.
\]
Then, for each $I$, there exists a constant $C'$ such that
\[
C'^{-1} < \frac{\chi_t}{\tilde{\chi}_t} < C'
\]
on each $U \in \ucal^I$.

For simplicity, we may assume that $I = \{0, 1, \ldots, l\}$ and $U \cap X_{0,i} = \{z_i = 0\}$ for $i \leq l$. Set $\sigma_i = -\log |z_i|^2$ for $0 \leq i \leq l$, and $y_t = |\log |t|^2|$, so that $\sum_{i=0}^l \sigma_i = y_t$. Then, for $i \leq l$, we have $\alpha_i = \psi_i - \sigma_i$ for some bounded smooth functions $\psi_i$, and
\[
\tilde{\chi}_t^{-1} = \left(\frac{\alpha_0}{\log |t|^2}\right)^2 \prod_{i=1}^l \sigma_i^2 \left(1 - \frac{\psi_i}{\sigma_i}\right)^2,
\]
for each $U \in \ucal^I$.

We thus obtain the following proposition.

\begin{proposition}
    There exist positive constants $t_0$, $M$, and $C_3$ such that for $0 < |t| < t_0$, on the region
    \[
    U^0(t) = \left\{ z \in U \cap X_t \;\middle|\; \sum_{i=1}^l \sigma_i < \frac{l}{l+1} y_t,\; \sigma_i > M \text{ for } i \leq l \right\},
    \]
    we have
    \[
    C_3^{-1} < \frac{\left(\prod_{i=1}^l |z_i|^2 \sigma_i^2\right) \chi_t V_t}{(\sqrt{-1})^n dz_1 \wedge d\bar{z}_1 \wedge \cdots \wedge dz_n \wedge d\bar{z}_n} < C_3.
    \]
\end{proposition}

Similarly, for $j \leq l$, define
\[
U^j(t) = \left\{ z \in U \cap X_t \;\middle|\; \sum_{i=0,\, i \neq j}^l \sigma_i < \frac{l}{l+1} y_t,\; \sigma_i > M \text{ for } i \leq l \right\}.
\]
Then we also have
\[
C_3'^{-1} < \frac{\left(\prod_{i=0,\, i \neq j}^l |z_i|^2 \sigma_i^2\right) \chi_t V_t}{(\sqrt{-1})^n \bigwedge_{i=0,\, i \neq j}^n dz_i \wedge d\bar{z}_i} < C_3'.
\]

Combining this with Proposition~\ref{prop-ke-t}, we obtain the following corollary.

\begin{corollary}\label{cor-wt-bound}
    There exist positive constants $t_0$, $M$, and $C_4$ such that for $0 < |t| < t_0$, on the region $U^j(t)$, we have
    \[
    C_4^{-1} < \frac{\left(\prod_{i=0,\, i \neq j}^l |z_i|^2 \sigma_i^2\right) \omega_t^n}{(\sqrt{-1})^n \bigwedge_{i=0,\, i \neq j}^n dz_i \wedge d\bar{z}_i} < C_4,
    \]
    where $\omega_t$ is the Kähler-Einstein metric on $X_t$.
\end{corollary}

For each $U \in \ucal$, let $R_U = \min_{i \leq l} R_i$, where $R_i$ are the radii of the coordinate variables. We may choose $M$ so that $M > -\log R_U + 1$ for every $U \in \ucal$.

On $U \in \ucal^I$, for $0 \leq j \leq l$, we can write $\alpha_j = \log |z_j|^2 + \zeta_j$.
Then,
\[
\tilde{\omega}_t = \tilde{\omega}_t' + 2\sum_{i=0}^l \left( \sqrt{-1} \frac{\partial \alpha_i \wedge \bar{\partial} \alpha_i}{\alpha_i^2} \right),
\]
where
\[
\tilde{\omega}_t' = 2\sum_{i=0}^m \left( \frac{\Ric(\|\cdot\|_i)}{\alpha_i} \right) + 2\sum_{i=l+1}^m \left( \sqrt{-1} \frac{\partial \alpha_i \wedge \bar{\partial} \alpha_i}{\alpha_i^2} \right)
\]
is a smooth form in a neighborhood of $X_{0,I}^o$.

On $U^0(t)$, we have $\partial \alpha_j = \frac{dz_j}{(\log |z_j|^2 + \zeta_j) z_j}$ for $1 \leq j \leq l$. Therefore,
\begin{align*}
2\sqrt{-1} \sum_{i=0}^l \frac{\partial \alpha_i \wedge \bar{\partial} \alpha_i}{\alpha_i^2}
&\geq 2\sqrt{-1} \sum_{i=1}^l \frac{\partial \alpha_i \wedge \bar{\partial} \alpha_i}{\alpha_i^2} \\
&= 2\sqrt{-1} \sum_{i=1}^l \frac{dz_i \wedge d\bar{z}_i}{|z_i|^2 (\log |z_i|^2 + \zeta_i)^2} \\
&\geq (1+\epsilon) \sqrt{-1} \sum_{i=1}^l \frac{dz_i \wedge d\bar{z}_i}{|z_i|^2 (\log |z_i|^2)^2}
\end{align*}
for some $\epsilon > 0$ (possibly after shrinking $U$).

Thus, by the second item of Lemma~\ref{lem-zhang-3.2}, we obtain
\[
\tilde{\omega}_t \geq \sqrt{-1} \sum_{i=1}^l \frac{dz_i \wedge d\bar{z}_i}{|z_i|^2 (\log |z_i|^2)^2} + \frac{1}{2} \omega_{U,I}^o.
\]
In particular,
\[
\|v_j\|_{\tilde{\omega}_t}^2 \geq \frac{1}{|z_j|^2 (\log |z_j|^2)^2}
\]
for $v_j = \frac{\partial}{\partial z_j}$ or $v_j = \frac{\partial}{\partial \bar{z}_j}$.

Similar estimates hold on $U^i(t)$ for $i \leq l$.

\medskip

Let $d_I(q)$ denote the distance from a point $q$ to $X_{0,I}$ with respect to $\omega$, and set $\tau_I(q) = -\log d_I^2(q)$. There exists $\epsilon > 0$ such that, in any $U \in \ucal^I$, for $|z_0|^2 + \cdots + |z_l|^2 < \epsilon$, we have
\[
d_I^2(q) = (1 + O((\sum_{i=0}^l |z_i|^2)^{1/2})) \cdot (z_0, \ldots, z_l) Q (\bar{z}_0, \ldots, \bar{z}_l)^t,
\]
where $Q$ is a positive definite Hermitian matrix depending smoothly on $(z_{l+1}, \ldots, z_n)$. Thus, there exists $C > 0$ (independent of $U$) such that, for $|z_0|^2 + \cdots + |z_l|^2 < \epsilon$,
\[
\frac{1}{C} (|z_0|^2 + \cdots + |z_l|^2) < d_I^2(q) < C (|z_0|^2 + \cdots + |z_l|^2).
\]
Therefore, there exists $C_1 > 0$ such that
\[
-C_1 < \tau_I - \log (|z_0|^2 + \cdots + |z_l|^2) < C_1.
\]

For each $J$ with $I \subsetneq J$ and each $U \in \ucal^J$, we may assume $J = \{0, \ldots, l'\}$, $I = \{0, \ldots, l\}$, and $X_{0,i} \cap U = \{z_i = 0\}$ for $i \leq l'$. Then, similarly, there exists $C_2 > 0$ such that
\[
-C_2 < \tau_I - \log (|z_0|^2 + \cdots + |z_l|^2) < C_2.
\]

\

Let $\mcal_U$ denote the region $\{p\in U\mid M<\sigma_i<3M \text{ for some } 0\leq i\leq l\}\subset U$. Let $\ncal_I$ denote the region $\{q\in\xcal \mid \frac{3}{2}M<\tau_I(q)<\frac{5}{2}M\}$. For $M$ sufficiently large, we have $\ncal_I\cap U\subset \mcal_U$ for every $U\in \bigcup_{I\subset J}\ucal^J$.

When $U\in \ucal^I$, note that on $U^0(t)$, $\sum_{j=1}^{l}\sigma_j\leq \frac{l}{l+1}y_t$. By our estimates for $\tilde{\omega}_t$ on $U^0(t)$, there exists $C_3>0$ such that, for $M$ large enough,
\[
|\dbar \tau_I|^2\leq |d \tau_I|^2<C_3 y_t^2
\]
on $\mcal_U\cap U^0(t)$, and similarly on $U^i(t)\cap \mcal_U$. For those $U\in \ucal^J$ with $I\subsetneq J$, analogous estimates hold. In summary, on $\ncal_I\cap X_t$ we have
\begin{equation}\label{e-sigma-dbar}
    |\dbar \tau_I|^2 <C_4 y_t^2,
\end{equation}
for some constant $C_4>0$.

\section{Outline of the proof of Theorem~\ref{thm-main}}\label{sec-3}

Let $\{t_u\}$ be a good sequence, which we fix throughout this section.

When $|I|\geq 2$, we define a volume form on $X_{0,I}^o$ by
\[
\nu_{\{t_u\},I,k+1} \triangleq \vcal_I (4\pi)^l \int_{B_I} \left(4\left(1-\sum_{j=1}^{l}x_j\right)^2 \prod_{j=1}^{l} 4x_j^2\right)^k e^{-k\varphi_I} dx_1\cdots dx_l.
\]
Since $\varphi_I$ is bounded, $\nu_{\{t_u\},I,k+1}$ is also of Poincar\'e type near $D_I$.

Let $\hcal_{\{t_u\},I,k+1}$ denote the space of holomorphic $L^2$-integrable sections of $(k+1)K_I$ on $X_{0,I}^o$, with norm squared given by
\[
\int_{X_{0,I}^o} \| \cdot \|^2_{\vcal_I} \, \nu_{\{t_u\},I,k+1}.
\]
\textbf{Terminology.} Since $\{t_u\}$ is fixed, we will write $\hcal_{I,k+1}$ for $\hcal_{\{t_u\},I,k+1}$ for simplicity, and use $t$ in place of $t_u$ for an element of this good sequence.

When $|I|=1$, i.e., $I=\{i\}$ for some $0\leq i\leq m$, we set $\hcal_{I,k+1} := \hcal_{0,i,k+1}$.

Under this identification, we obtain a space of holomorphic sections of $(k+1)(K_I+A_I)$ on $X_{0,I}^o$. Let $V\subset X_{0,I}$ be a neighborhood of a point $p\in D_I$ with coordinates $(z_{l+1},\ldots,z_n)$ such that $D_I$ is defined by $\prod_{i=l+1}^{l'}z_i=0$ for some $l'\geq l+1$. We may choose a local frame of $A_I$ so that the local holomorphic function representing $S_{\hat{I}}$ is $\prod_{i=l+1}^{l'}z_i$. Let $s\in \hcal_{I,k+1}$; then locally $s=f\,dz_{l+1}\wedge\cdots\wedge dz_n$ for some holomorphic function $f$ on $V\setminus D_I$. Then
\[
\|s\|^2_{\vcal_I} = |f|^2 e^\eta \prod_{i=l+1}^{l'} |z_i|^{2(k+1)} \prod_{i=l+1}^{l'} (\log |z_i|^2)^{2(k+1)},
\]
for some bounded function $\eta$. Since $\nu_{\{t_u\},I,k+1}$ is also of Poincar\'e type near $D_I$, the $L^2$-integrability of $s$ implies that $|f|^2 \prod_{i=l+1}^{l'} |z_i|^{2(k+1)}$ must vanish along $D_I$, so $f \prod_{i=l+1}^{l'} z_i^{k+1}$ extends to a holomorphic function vanishing along $D_I$. Therefore,
\[
s \otimes S_{\hat{I}}^{\otimes (k+1)}
\]
extends to a holomorphic section in $H^0(X_{0,I}, (k+1)(K_I+A_I))$ vanishing along $D_I$. By this explicit identification, we view $\hcal_{I,k+1}$ as a subspace of $H^0(X_{0,I}, (k+1)(K_I+A_I))$, and hence as a subspace of $H^0(X_{0,I}, (k+1)K_{\xcal})$.

\

To apply these estimates to global sections, we use an extension theorem due to Finski. Let $X$ be a complex manifold of dimension $n$ equipped with a positive line bundle $(E, h^E)$, and let $(F, h^F)$ be an arbitrary Hermitian vector bundle over $X$. Let $Y$ be a complex submanifold of $X$ of dimension $m$, with embedding $\iota: Y \hookrightarrow X$. Fix volume forms $dv_X$ and $dv_Y$ on $X$ and $Y$, respectively, so that both $H^0(X, E^k \otimes F)$ and $H^0(Y, \iota^*(E^k \otimes F))$ are endowed with natural Hermitian inner products. Denote by $H^{0,\bot}(X, E^k \otimes F)$ the orthogonal complement of the subspace of sections vanishing along $Y$. Consider the restriction operator
\[
\mathrm{Res}_k: H^{0,\bot}(X, E^k \otimes F) \to H^0(Y, \iota^*(E^k \otimes F)).
\]

Under suitable bounded geometry assumptions, Finski proved the following result.

\begin{theorem}[Theorem 4.1 in \cite{Finski}]\label{thm-fin-total}
    There exist constants $c, C > 0$ and an integer $k_1 > 0$ such that for all $k \geq k_1$,
    \[
    c\, k^{\frac{n-m}{2}} \leq \| \mathrm{Res}_k \| \leq C\, k^{\frac{n-m}{2}}.
    \]
\end{theorem}

We refer the reader to \cite{Finski}, Definitions 2.3 and 2.4, for the precise definition of bounded geometry. The lower bound in the theorem can be made more explicit:

\begin{theorem}[Theorem 4.4 in \cite{Finski}]
    There exist constants $C > 0$ and $k_1 > 0$ such that for any $k \geq k_1$ and any $s \in H^0(Y, \iota^*(E^k \otimes F))$, there exists $f \in H^0(X, E^k \otimes F)$ with $f|_Y = s$ and
    \[
    \| f \|_{L^2(X)} \leq C\, k^{(m-n)/2} \| s \|_{L^2(Y)}.
    \]
\end{theorem}

In our setting, we take $Y = X_{0,I}$, $X = \xcal$ with Kähler form $\omega$, and $E = K_\xcal$ equipped with the Hermitian metric $h$. By possibly shrinking the base disk $B$, our situation (with $F$ trivial) satisfies the bounded geometry assumptions of \cite{Finski}. For any submanifold $Y' \subset \xcal$, we denote by $\| \cdot \|_{h, Y'}$ the $L^2$-norm on $H^0(Y', kL)$ defined by $h$ and $\omega$. Then there exist constants $C_1 > 0$ and $k_1 > 0$ such that for any section $s \in H^0(X_{0,I}, kL)$ of unit norm and $k \geq k_1$, there exists an extension $\tilde{s} \in H^0(\xcal, kL)$ satisfying
\begin{equation}\label{e-c1-x0I}
    \| \tilde{s} \|^2_{h, \xcal} \leq \frac{C_1}{k^{l+1}}.
    \end{equation}
To fix our choice, we let $\tilde{s}$ be the extension of minimal $L^2$-norm.

When $D_I$ is not empty, let $\hcal_{I,k+1,2} \subset \hcal_{I,k+1}$ denote the subspace of sections vanishing to order at least $2$ along $D_I$, and let $\gcal_{I,k+1}$ be its orthogonal complement.

We now apply Finski's extension theorem to extend sections from $X_{0,I}$ to $\xcal$. There are three cases:
\begin{itemize}
    \item[1.] $D_I$ is empty. For each $s \in H^0(X_{0,I}, (k+1)L)$, let $\tilde{s} \in H^0(\xcal, (k+1)L)$ be its minimal $L^2$-norm extension.
    \item[2.] $D_I$ is nonempty and $s \in \gcal_{I,k+1}$. Then $s = S_{\hat{I}} \otimes s'$ for some $s' \in H^0(X_{0,I}, (k+1)L - \sum_{j \notin I} [X_{0,j}])$. Let $h_{\hat{I},k+1}$ denote the Hermitian metric on $(k+1)L - \sum_{j \notin I} [X_{0,j}]$ induced by $h$ and the $\|\cdot\|_j$. Let $\tilde{s}' \in H^0(\xcal, (k+1)L - \sum_{j \notin I} [X_{0,j}])$ be the minimal extension of $s'$, and set $\tilde{s} = S_{\hat{I}} \otimes \tilde{s}'$.
    \item[3.] $D_I$ is nonempty and $s \in \hcal_{I,k+1,2}$. Then $s = S_{\hat{I}}^{\otimes 2} \otimes s'$ for some $s' \in H^0(X_{0,I}, (k+1)L - 2\sum_{j \notin I} [X_{0,j}])$. Let $\tilde{s}' \in H^0(\xcal, (k+1)L - 2\sum_{j \notin I} [X_{0,j}])$ be the minimal extension of $s'$, and set $\tilde{s} = S_{\hat{I}}^{\otimes 2} \otimes \tilde{s}'$.
\end{itemize}

\

If $s$ is from any of these three scenarios, we obtain a section, which can be locally written as
\[
s_t \triangleq \frac{\tilde{s}}{(dt)^{\otimes (k+1)}} = \frac{f}{\prod_{i=1}^{l} z_i^{k+1}} (dz_1 \wedge \cdots \wedge dz_n)^{\otimes (k+1)}
\]
in $H^0(X_t, (k+1)K_t)$.

\textbf{Terminology.} For simplicity, we refer to $s_t$ as the restriction of $\tilde{s}$ to $X_t$.

To control the behavior of $s_t$ away from $X_{0,I}$, we further modify $s_t$ using H\"ormander's $L^2$ estimates. The following lemma is standard; see, for example, \cite{Tian1990On}.

\begin{lem}
    Suppose $(M,g)$ is a complete Kähler manifold of complex dimension $n$, and $\mathcal{L}$ is a line bundle on $M$ with Hermitian metric $h$. If
    \[
    \langle -2\pi i\, \Theta_h + \mathrm{Ric}(g), v \wedge \bar{v} \rangle_g \geq C |v|_g^2
    \]
    for any tangent vector $v$ of type $(1,0)$ at any point of $M$, where $C > 0$ is a constant and $\Theta_h$ is the curvature form of $h$, then for any smooth $\mathcal{L}$-valued $(0,1)$-form $\alpha$ on $M$ with $\bar{\partial}\alpha = 0$ and $\int_M |\alpha|^2 dV_g < \infty$, there exists a smooth $\mathcal{L}$-valued function $\beta$ on $M$ such that $\bar{\partial}\beta = \alpha$ and
    \[
    \int_M |\beta|^2 dV_g \leq \frac{1}{C} \int_M |\alpha|^2 dV_g,
    \]
    where $dV_g$ is the volume form of $g$ and the norms are induced by $h$ and $g$.
\end{lem}

In our setting, $(M,g) = (X_t, \omega_t)$ and the line bundle is $kL$, so for $k$ sufficiently large, the assumptions of the lemma are satisfied.

We now introduce a cut-off function $\varpi(\eta)$ of one variable, satisfying:
\begin{itemize}
    \item $\varpi(\eta) = 1$ for $\eta \geq \frac{5}{2}M$;
    \item $\varpi(\eta) = 0$ for $\eta \leq \frac{3}{2}M$;
    \item $0 \leq \varpi'(\eta) < \frac{1}{2}$.
\end{itemize}
We seek to solve the equation
\[
\bar{\partial} v = \bar{\partial} \varpi(\tau) \otimes s_t
\]
on $X_t$.

By formula~\ref{e-sigma-dbar}, we have on $\ncal_I \cap X_t$ that
\[
|\bar{\partial} \varpi(\tau_I)|^2 < c_1 y_t^2,
\]
for some constant $c_1 > 0$.

We will need the following theorem.

\begin{theorem}\label{thm-inner-sections-prop}
    If $0 \neq s \in \hcal_{I,k+1}$ is from any of these three scenarios, then:
    \begin{itemize}
        \item
        \[
        \int_{X_t} \| s_t \|^2_{\mathrm{KE}} \, \omega_t^n > c_2 y_t^{l(2k+1)},
        \]
        for some constant $c_2 > 0$ depending on $s$ but independent of $t$.
        \item
        \[
        \int_{X_t \cap \ncal_I} \| s_t \|^2_{\mathrm{KE}} \, \omega_t^n < c_3 y_t^{(l-1)(2k+1)},
        \]
        for some constant $c_3 > 0$ depending on $s$ but independent of $t$.
    \end{itemize}
\end{theorem}

Due to the length of the arguments, we postpone the proof of the theorem to maintain the coherence of the exposition.

We now state the following proposition, which will be useful in the construction:

\begin{proposition}
    If $s \neq 0$ is as in any of the three cases above, then
    \[
    \int_{X_t} \| \bar{\partial} \varpi(\tau) \otimes s_t \|^2_{\mathrm{KE}} \, \omega_t^n < c_4 y_t^{(l-1)(2k+1)+2},
    \]
    for some constant $c_4$ independent of $t$.
\end{proposition}

Therefore, by Hörmander's $L^2$-estimate, we can find a solution $v_t \in C^\infty(X_t, kL)$ to
\[
\bar{\partial} v_t = \bar{\partial} \varpi(\tau) \otimes s_t
\]
such that
\[
\int_{X_t} \| v_t \|^2_{\mathrm{KE}} \, \omega_t^n < c_5 y_t^{(l-1)(2k+1)+2},
\]
for some constant $c_5$ independent of $t$. We choose $v_t$ to be the solution of minimal $L^2$-norm.

We then define a new holomorphic pluricanonical section by
\[
s_{t,\mathrm{rn}} \triangleq y_t^{-\frac{l}{2}(2k+1)} \left( \varpi(\tau) s_t - v_t \right).
\]

\medskip

When $D_I$ is empty, given an orthonormal basis $\mathcal{B}_{I,0} = \{ s^i \}$ of $\mathcal{H}_{I,k+1}$, we obtain a set of sections $\mathcal{B}_{I,t} = \{ s^i_{t,\mathrm{rn}} \}$, and we fix an order on $\mathcal{B}_{I,t}$.

When $D_I$ is nonempty, we construct $\mathcal{B}_{I,t}$ in two steps:
\begin{itemize}
    \item Given an orthonormal basis $\mathcal{B}_{I,0,1} = \{ s^i \}$ of $\mathcal{G}_{I,k+1}$, we obtain $\mathcal{B}_{I,t,1} = \{ s^i_{t,\mathrm{rn}} \}$ and fix an order.
    \item Given an orthonormal basis $\mathcal{B}_{I,0,2} = \{ s^i \}$ of $\mathcal{H}_{I,k+1,2}$, we obtain $\mathcal{B}_{I,t,2} = \{ s^i_{t,\mathrm{rn}} \}$ and fix an order.
\end{itemize}
We then set $\mathcal{B}_{I,t} = \mathcal{B}_{I,t,1} \cup \mathcal{B}_{I,t,2}$, ordered so that elements of $\mathcal{B}_{I,t,1}$ precede those of $\mathcal{B}_{I,t,2}$. The corresponding basis of $\mathcal{H}_{I,k+1}$ is denoted by $\mathcal{B}_{I,0}$.

We define the global basis
\[
\mathcal{B}_t = \bigcup_{I,\, X_{0,I} \neq \emptyset} \mathcal{B}_{I,t},
\]
with the following order: first by increasing $|I|$, then lexicographically within subsets of the same $|I|$, and finally by the order within each $\mathcal{B}_{I,t}$.

Let $\mathcal{Q}_t$ be the Gram matrix of inner products of the elements of $\mathcal{B}_t$ with respect to the Kähler-Einstein metric on $X_t$:
\[
\mathcal{Q}_t = \left( \langle s_i, s_j \rangle_{\mathrm{KE}, X_t} \right)_{i,j}.
\]

We have the following theorem:

\begin{theorem}\label{thm-almost-orthon}
    The sections in $\mathcal{B}_t$ are almost orthonormal in the sense that
    \[
    \lim_{u \to \infty} \mathcal{Q}_{t_u} = \mathrm{Id}.
    \]
\end{theorem}

We will prove Theorem~\ref{thm-almost-orthon} after establishing Theorem~\ref{thm-inner-sections-prop}.

By the formula for $\dim H^0(X_t, (k+1)L)$, it follows that $\mathcal{B}_t$ forms a basis for $\mathcal{H}_{t,k+1}$ that is almost orthonormal. This basis $\mathcal{B}_t$ defines a Kodaira embedding
\[
\Phi'_{t,k+1}: X_t \to \mathbb{CP}^{N_k-1}.
\]
We will prove the following result in Section~\ref{sec-almost-bergman}:

\begin{theorem}\label{thm-almost-Bergman-Embedding}
    Let $\{t_u\}$ be a good sequence. Then the images of $\Phi'_{t_u,k+1}$ converge, as $u \to \infty$, to a subvariety $Y$ as described in Theorem~\ref{thm-main}.
\end{theorem}

Applying the Gram-Schmidt process to the bases $\mathcal{B}_{I,t_u}$ yields orthonormal bases $\overline{\mathcal{B}}_{I,t_u}$, which in turn define Kodaira embeddings $\Phi_{t_u,k+1}$. It is clear that, as $u \to \infty$, the images of $\Phi_{t_u,k+1}$ converge to the same limit as those of $\Phi'_{t_u,k+1}$. This completes the proof of Theorem~\ref{thm-main}.

Note that $\mathcal{B}_{I,0}$ also induces an embedding
\[
\Phi_{I,k+1}: X_{0,I} \to \mathbb{CP}^{n_{I,k+1}-1}.
\]
The inclusion $\mathcal{B}_{I,t} \subset \mathcal{B}_t$ induces an inclusion $\iota_I: \mathbb{CP}^{n_{I,k+1}-1} \hookrightarrow \mathbb{CP}^{N_{k+1}-1}$. For simplicity, we will continue to denote by $\Phi_{I,k+1}$ its composition with $\iota_I$.

\section{Proof of Theorem~\ref{thm-inner-sections-prop}}\label{sec-4}

\subsection{Case: $D_I$ is empty}\label{subsec-case-0}

Let $s \in H^0(X_{0,I}, (k+1)L)$ be a section of unit $L^2$-norm with respect to $(h, \omega)$. On $U \in \mathcal{U}^I$, we can write $\tilde{s} = f (dz_0 \wedge \cdots \wedge dz_n)^{\otimes (k+1)}$. Since $U$ is a polydisc, let $R_i$ denote the radius of the $z_i$ variable for $i \leq n$.

On $U^0(t)$, we can write
\[
\omega_t^n = e^{-\phi_t} \frac{(\sqrt{-1})^{n}}{\prod_{i=1}^{l} |z_i|^2 \sigma_i^2} dz_1 \wedge d\bar{z}_1 \wedge \cdots \wedge dz_n \wedge d\bar{z}_n,
\]
where $|\phi_t| < \log C_4$ with $C_4$ as in Corollary~\ref{cor-wt-bound}.

Thus, the pointwise norm of $s_t$ is
\[
\| s_t \|_{\mathrm{KE}}^2 = e^{(k+1)\phi_t} |f|^2 \left( \prod_{i=1}^{l} \sigma_i^2 \right)^{k+1}.
\]
The local $L^2$-norm of $s_t$ on $U^0(t)$ is then
\begin{align*}
   & (\sqrt{-1})^{n} \int_{U^0(t)} e^{k\phi_t} |f|^2 \left( \prod_{i=1}^{l} \sigma_i^2 \right)^{k} \frac{dz_1 \wedge d\bar{z}_1 \wedge \cdots dz_n \wedge d\bar{z}_n}{\prod_{i=1}^{l} |z_i|^2} \\
   &= (\sqrt{-1})^{n} \int \left( \int e^{k\phi_t} |f|^2 \left( \prod_{i=1}^{l} \sigma_i^2 \right)^{k} \frac{dz_1 \wedge d\bar{z}_1 \wedge \cdots dz_l \wedge d\bar{z}_l}{\prod_{i=1}^{l} |z_i|^2} \right) \bigwedge_{j>l} dz_j \wedge d\bar{z}_j.
\end{align*}

To estimate the local $L^2$-norm of $s_t$, by Corollary~\ref{cor-wt-bound}, it suffices to estimate integrals of the form
\[
\int |f|^2 \left( \prod_{i=1}^l \sigma_i^2 \right)^k \frac{dz_1 \wedge d\bar{z}_1 \wedge \cdots \wedge dz_l \wedge d\bar{z}_l}{\prod_{i=1}^l |z_i|^2}.
\]
We can expand
\[
f = \sum_{\beta} b_{\beta}(z_{l+1}, \ldots, z_n) \prod_{i=0}^l z_i^{\beta_i},
\]
where $\beta = (\beta_0, \ldots, \beta_l)$ runs over all multi-indices with $\beta_i \geq 0$ for $i \leq l$, and the coefficients $b_{\beta}$ are holomorphic functions. For simplicity, we write $z^\beta = \prod_{i=0}^l z_i^{\beta_i}$, with the convention that $\beta_j = 0$ for $j > l$.

On $U \cap X_t$, using $(z_1, \ldots, z_n)$ as holomorphic coordinates, we can substitute $z_0 = t / \prod_{j=1}^l z_j$ into each term $\prod_{i=0}^l z_i^{\beta_i}$, yielding
\[
f = \sum_{\alpha} a_{\alpha}(z_{l+1}, \ldots, z_n) \prod_{i=1}^l z_i^{\alpha_i},
\]
where $\alpha = (\alpha_1, \ldots, \alpha_l)$ runs over all multi-indices (possibly with some $\alpha_i < 0$), and the coefficients $a_{\alpha}$ are holomorphic functions (depending on $t$). For notational simplicity, we write $z^\alpha = \prod_{i=1}^l z_i^{\alpha_i}$.

If we define $A_{\alpha} = \{ \beta \mid \beta_i - \beta_0 = \alpha_i,\, 1 \leq i \leq l \}$, then
\[
a_\alpha = \sum_{\beta \in A_{\alpha}} b_\beta\, t^{\beta_0}.
\]

Therefore,
\begin{align*}
    &\int_{U^0(t)} |f|^2 \left( \prod_{i=1}^l \sigma_i^2 \right)^k \frac{(\sqrt{-1})^n dz_1 \wedge d\bar{z}_1 \wedge \cdots dz_n \wedge d\bar{z}_n}{\prod_{i=1}^l |z_i|^2} \\
    &= \sum_{\alpha} B_{\alpha, t} \int_{|z_j| < R_j} |a_\alpha|^2\, (\sqrt{-1})^{n-l} dz_{l+1} \wedge d\bar{z}_{l+1} \wedge \cdots dz_n \wedge d\bar{z}_n,
\end{align*}
where
\[
B_{\alpha, t} = (2\pi)^l \int_{V(t)} |z^\alpha|^2 \left( \prod_{i=1}^l \sigma_i^2 \right)^k d\sigma_1 \cdots d\sigma_l,
\]
and $V(t)$ is defined as
\[
V(t) = \left\{ (\sigma_1, \ldots, \sigma_l) \in \mathbb{R}^l \;\middle|\; \sum_{i=1}^l \sigma_i < \frac{l}{l+1} y_t,\;\; \sigma_i > M \text{ for } 1 \leq i \leq l \right\}.
\]

It remains to estimate the integral
\[
\int_{V(t)} |z^\alpha|^2 \left( \prod_{i=1}^l \sigma_i^2 \right)^k d\sigma_1 \cdots d\sigma_l.
\]

\subsubsection*{Case 0}
When $\alpha = 0 = (0, \ldots, 0)$, we have
\begin{align*}
    \int_{V(t)} \left( \prod_{i=1}^{l} \sigma_i^2 \right)^k d\sigma_1 \cdots d\sigma_l
    &> \left( \int_{M}^{\frac{1}{l+1} y_t} \sigma^{2k} d\sigma \right)^l \\
    &= \left[ \frac{\left( \frac{1}{l+1} y_t \right)^{2k+1} - M^{2k+1}}{2k+1} \right]^l.
\end{align*}
Similarly,
\begin{align*}
    \int_{V(t)} \left( \prod_{i=1}^{l} \sigma_i^2 \right)^k d\sigma_1 \cdots d\sigma_l
    &< \left( \int_{M}^{\frac{l}{l+1} y_t} \sigma^{2k} d\sigma \right)^l \\
    &= \left[ \frac{\left( \frac{l}{l+1} y_t \right)^{2k+1} - M^{2k+1}}{2k+1} \right]^l.
\end{align*}
If we set
\[
s_{t,0} = \frac{a_{(0,\ldots,0)}}{\prod_{i=1}^{l} z_i^{k+1}} (dz_1 \wedge \cdots \wedge dz_n)^{\otimes (k+1)},
\]
then for $|t|$ sufficiently small,
\begin{equation}\label{e-a-0-int}
    \int_{U \cap X_t} \| s_{t,0} \|^2_{\mathrm{KE}} \, \omega_t^n > C_0^k y_t^{l(2k+1)} \int_{|z_j| < R_j} |a_0|^2 (\sqrt{-1})^{n-l} \bigwedge_{j=l+1}^n (dz_j \wedge d\bar{z}_j),
\end{equation}
for some constant $C_0$ independent of $t$ and $U$.

Let $W_i(t) \subset V(t)$, $1 \leq i \leq l$, be defined by
\[
W_i(t) = \left\{ (\sigma_1, \ldots, \sigma_l) \in V(t) \mid \sigma_i < \sqrt{y_t} \right\}.
\]
Then,
\begin{align*}
    \int_{W_j(t)} \left( \prod_{i=1}^{l} \sigma_i^2 \right)^k d\sigma_1 \cdots d\sigma_l
    &< \left( \int_{M}^{\frac{l}{l+1} y_t} \sigma^{2k} d\sigma \right)^{l-1} \int_{M}^{\sqrt{y_t}} \sigma^{2k} d\sigma \\
    &= \left[ \frac{\left( \frac{l}{l+1} y_t \right)^{2k+1} - M^{2k+1}}{2k+1} \right]^{l-1} \frac{\sqrt{y_t}^{2k+1} - M^{2k+1}}{2k+1}.
\end{align*}
It is then clear that, for $|t|$ sufficiently small,
\begin{equation}\label{e-geq-sqrt-yt}
    \frac{ \int_{W_j(t)} \left( \prod_{i=1}^{l} \sigma_i^2 \right)^k d\sigma_1 \cdots d\sigma_l }{ \int_{V(t)} \left( \prod_{i=1}^{l} \sigma_i^2 \right)^k d\sigma_1 \cdots d\sigma_l }
    < c\, y_t^{-\frac{2k+1}{2}},
\end{equation}
for some constant $c$ depending only on $k$ and $l$.

Similarly, let $T_i(t) \subset V(t)$, $1 \leq i \leq l$, be defined by
\[
T_i(t) = \left\{ (\sigma_1, \ldots, \sigma_l) \in V(t) \mid \sigma_i < 4M \right\}.
\]
Then,
\begin{equation}\label{e-geq-4m}
    \frac{ \int_{T_j(t)} \left( \prod_{i=1}^{l} \sigma_i^2 \right)^k d\sigma_1 \cdots d\sigma_l }{ \int_{V(t)} \left( \prod_{i=1}^{l} \sigma_i^2 \right)^k d\sigma_1 \cdots d\sigma_l }
    < c' y_t^{-(2k+1)},
\end{equation}
for some constant $c'$ depending only on $k$, $M$, and $l$.

\medskip

Denote by $\nu_l=(\sqrt{-1})^{n-l}\bigwedge_{j=l+1}^n (dz_{j}\wedge d\bar{z}_{j})$. 
We need to show that the integral $\int_{|z_j|<R_j}|a_{0}|^2\nu_l$ in formula \ref{e-a-0-int} converges to $\int_{|z_j|<R_j}|b_{0}|^2\nu_l$ as $|t|\to 0$.

Let $\omega^{n+1}=(\sqrt{-1})^{n+1}e^{-\eta}dz_0\wedge d\bar{z}_0\wedge\cdots dz_n\wedge d\bar{z}_n$, so that 
$$\|\tilde{s}\|_h^2=|f|^2e^{(k+1)\eta}.$$
It is easy to see that there exists $\epsilon_1>0$ such that on each $U\in\ucal^I$, we have $e^{k\eta}>\epsilon_1$.
Therefore, by formula \ref{e-c1-x0I},
\begin{align*}
    \frac{C_1}{k^{l+1}} &> \int_{U}|f|^2 e^{k\eta}(\sqrt{-1})^{n+1}dz_0\wedge d\bar{z}_0\wedge\cdots dz_n\wedge d\bar{z}_n \\ 
    &> \epsilon_1\int_{U}|f|^2(\sqrt{-1})^{n+1}dz_0\wedge d\bar{z}_0\wedge\cdots dz_n\wedge d\bar{z}_n \\
    &= \epsilon_1(2\pi)^{l+1} \sum_{\beta}\int_{|z_j|<R_j}\prod_{i=0}^{l}\frac{R_i^{\beta_i+1}}{\beta_i+1}|b_\beta|^2\nu_l,
\end{align*}
where $c_\beta=\prod_{i=0}^{l}\frac{R_i^{\beta_i+1}}{\beta_i+1}$. By the Cauchy-Schwarz inequality and the fact that $c_\beta>(\min_i R_i)^{l+\sum \beta_i}$, we have
\begin{align*}
    \left|\sum_{\beta\in A_0,\, \beta_0>0}b_\beta t^{\beta_0}\right|^2
    &\leq \left(\sum_{\beta\in A_0,\, \beta_0>0}|b_\beta|^2c_\beta\right)\left(\sum_{\beta\in A_0,\, \beta_0>0}\frac{|t|^{2\beta_0}}{c_\beta}\right) \\
    &\leq C|t| \sum_{\beta\in A_0,\, \beta_0>0}|b_\beta|^2c_\beta,
\end{align*}
for some constant $C>0$ independent of $U$. 
Therefore,
$$
\int_{|z_j|<R_j}\left|\sum_{\beta\in A_0,\, \beta_0>0}b_\beta t^{\beta_0}\right|^2\nu_l < \frac{CC_1}{k^{l+1}\epsilon_1 (2\pi)^{l+1}}|t|.
$$
Thus, $\int_{|z_j|<R_j}|a_{0}|^2\nu_l$ converges to $\int_{|z_j|<R_j}|b_{0}|^2\nu_l$ as $t\to 0$.

\medskip

Similarly, for $\alpha\neq 0$, we have
\begin{align}
    \left|\sum_{\beta\in A_\alpha,\, \beta_0>0}b_\beta t^{\beta_0}\right|^2
    &\leq \left(\sum_{\beta\in A_\alpha,\, \beta_0>0}|b_\beta|^2c_\beta\right)\left(\sum_{j=1}^{\infty}\frac{|t|^{2j}}{c_{(j,\alpha_1+j,\cdots,\alpha_l+j)}}\right) \notag\\
    &\leq C_2|t| \sum_{\beta\in A_\alpha,\, \beta_0>0}|b_\beta|^2c_\beta, \label{e-alpha-not0}
\end{align}
for some constant $C_2>0$ independent of $\alpha$ and $U$, when $|t|$ is small enough.

\medskip

For any subset $J'\subset J=\{1,\cdots,l\}$, define
$$
Q_{J'} = \left\{\alpha \;\middle|\; \alpha_i>0 \text{ for all } i\in J',\; \alpha_i=0 \text{ for } i\in J\setminus J'\right\}.
$$
When $\alpha\neq 0$, there are three cases:
\begin{enumerate}
    \item[(1)] $\alpha\in Q_J$;
    \item[(2)] $\alpha\in Q_{J'}$ for some proper subset $J'\subset J$;
    \item[(3)] $\alpha_i<0$ for some $i$.
\end{enumerate}
The calculations for cases (1) and (2) are similar, but for clarity, we treat them separately below.

\subsubsection*{Case 1.} When $\alpha \in Q_J$, we have 
\begin{align*}
    B_{\alpha,t} &= (2\pi)^l \int_{V(t)} |z^\alpha|^2 \left( \prod_{i=1}^{l} \sigma_i^{2k} \right) d\sigma_1 \cdots d\sigma_l \\
    &\leq (2\pi)^l \prod_{i=1}^{l} \int_M^\infty |z_i|^{2\alpha_i} \sigma_i^{2k} d\sigma_i.
\end{align*}

Since $\tilde{s}$ is also an extension of $\tilde{s}|_{X_{0,i}}$ for each $i \in I$, by Theorem~\ref{thm-fin-total}, we have
\[
\| \tilde{s} \|^2_{h, \xcal} \geq \frac{c'}{k} \| \tilde{s} \|^2_{h, X_{0,i}},
\]
for some $c' > 0$. Thus,
\[
\| \tilde{s} \|^2_{h, X_{0,i}} \leq \frac{C'}{k^l} \| \tilde{s} \|^2_{h, X_{0,I}},
\]
for some $C' > 0$. 

Let $\omega^{n}|_{X_{0,0}} = (\sqrt{-1})^{n} e^{-\eta_1} dz_1 \wedge d\bar{z}_1 \wedge \cdots \wedge dz_n \wedge d\bar{z}_n$. Then there exists $\epsilon_2 > 0$ such that on $U \cap X_{0,0}$, $e^{(k+1)\eta - \eta_1} > \epsilon_2$. Therefore,
\begin{align*}
    \frac{C'}{k^{l}} &> \int_{U \cap X_{0,0}} |f|^2 e^{(k+1)\eta - \eta_1} (\sqrt{-1})^{n} dz_1 \wedge d\bar{z}_1 \wedge \cdots dz_n \wedge d\bar{z}_n \\
    &> \epsilon_2 \int_{U \cap X_{0,0}} |f|^2 (\sqrt{-1})^{n} dz_1 \wedge d\bar{z}_1 \wedge \cdots dz_n \wedge d\bar{z}_n \\
    &= \epsilon_2 (2\pi)^{l} \sum_{\beta,\, \beta_0=0} \int_{|z_j|<R_j} \prod_{i=1}^{l} \frac{R_i^{\beta_i+1}}{\beta_i+1} |b_\beta|^2 \nu_l.
\end{align*}
Thus,
\[
\sum_{\beta,\, \beta_0=0} \int_{|z_j|<R_j} \prod_{i=1}^{l} \frac{R_i^{\beta_i+1}}{\beta_i+1} |b_\beta|^2 \nu_l < \frac{C'}{k^{l} \epsilon_2 (2\pi)^{l}}.
\]

Let $R = \min_i R_i$. We have the following proposition.

\begin{proposition}
    For $k$ sufficiently large,
    \[
    B_{\alpha,t} < \left( (2k)! \frac{4\pi}{R^2} \right)^l \prod_{i=1}^{l} \frac{R_i^{\alpha_i+1}}{\alpha_i+1}.
    \]
\end{proposition}
\begin{proof}
    The idea is to compare the functions $g_1(\lambda) = \frac{1}{(2k)!} \int_M^\infty e^{-\lambda \sigma} \sigma^{2k} d\sigma$ and $g_2(\lambda) = \frac{1}{\lambda+1} R^{\lambda+1}$. Since
    \[
    g_1(\lambda) < \frac{1}{(2k)!} \int_0^\infty e^{-\lambda \sigma} \sigma^{2k} d\sigma = \frac{1}{\lambda^{2k+1}},
    \]
    we can choose $k$ large enough so that $\frac{1}{2^{2k+1}} < \frac{1}{3} R^3$. Now,
    \[
    \frac{d}{d\lambda} \log g_1 = -\frac{\int_M^\infty \sigma e^{-\lambda \sigma} \sigma^{2k} d\sigma}{\int_M^\infty e^{-\lambda \sigma} \sigma^{2k} d\sigma} \leq -M,
    \]
    while $\frac{d}{d\lambda} \log g_2 = \log R - \frac{1}{\lambda+1}$. By our choice of $M$, we have $\frac{d}{d\lambda} \log g_1 < \frac{d}{d\lambda} \log g_2$. Thus, $g_1(\lambda) < g_2(\lambda)$ for $\lambda \geq 2$. When $\lambda = 1$, clearly $g_1(1) < 1$. The conclusion follows.
\end{proof}

By formula \ref{e-alpha-not0}, we have 
\begin{eqnarray}
    \left|\sum_{\beta\in A_\alpha,\, \beta_0>0} b_\beta t^{\beta_0}\right|^2 \leq C'|t| \sum_{\beta\in A_\alpha,\, \beta_0>0} |b_\beta|^2 c_\beta.
\end{eqnarray}
Therefore,
\begin{eqnarray*}
   &&\frac{R^2}{((2k)!4\pi)^l} \sum_{\alpha\in Q_{J}} B_{\alpha,t} \int_{|z_j|<R_j} |a_{\alpha}|^2 \nu_l \\
   &\leq& \sum_{\alpha\in Q_{J}} \left( \prod_{i=1}^{l} \frac{R_i^{\alpha_i+1}}{\alpha_i+1} \right) \int_{|z_j|<R_j} |a_{\alpha}|^2 \nu_l \\
   &\leq& 2 \sum_{\alpha\in Q_{J}} \left( \prod_{i=1}^{l} \frac{R_i^{\alpha_i+1}}{\alpha_i+1} \right) \int_{|z_j|<R_j} \left[ |b_{\overline{0\alpha}}|^2 + \left| \sum_{\beta\in A_\alpha,\, \beta_0>0} b_\beta t^{\beta_0} \right|^2 \right] \nu_l \\
   &\leq& \frac{2C'}{k^{l}\epsilon_2(2\pi)^{l}} + 2 \sum_{\alpha\in Q_{J}} \int_{|z_j|<R_j} C_2|t| \sum_{\beta\in A_\alpha,\, \beta_0>0} |b_\beta|^2 c_\beta \nu_l \\
   &\leq& \frac{2C'}{k^{l}\epsilon_2(2\pi)^{l}} + \frac{2C_2C_1}{k^{l+1}\epsilon_1 (2\pi)^{l+1}} |t| ,
\end{eqnarray*}
where we denote by $\overline{0\alpha}$ the multi-index $\beta\in A_\alpha$ with $\beta_0=0$, and $C_2$ is the constant appearing in formula \ref{e-alpha-not0}. Thus, for $|t|$ sufficiently small,
\[
\sum_{\alpha\in Q_{J}} B_{\alpha,t} \int_{|z_j|<R_j} |a_{\alpha}|^2 \nu_l < \frac{3C'((2k)!4\pi)^l}{R^2 k^{l} \epsilon_2 (2\pi)^{l}}.
\]

\medskip

For any $i\neq 0$ with $i\leq l$, since $\sigma_i \geq \frac{1}{l+1} y_t$ on $U^i(t)$, we have $|z_i|^2 \leq |t|^{\frac{2}{l+1}}$. Thus,
\[
|z^\alpha|^2 \leq |t|^{\frac{2}{l+1}\alpha_i} \prod_{j=1,\, j\neq i}^{l} |z_j|^2.
\]
Let $\vartheta_i = \prod_{j=0,\, j\neq i}^{l} \sigma_j^{2k}$ and define
\[
\mu_i = \vartheta_i \left( \bigwedge_{j=0,\, j\neq i}^{l} \frac{dz_j \wedge d\bar{z}_j}{|z_j|^2} \right) \wedge \bigwedge_{j=l+1}^{n} dz_j \wedge d\bar{z}_j.
\]
Then,
\begin{eqnarray*}
    \int_{U^i(t)} |a_{\alpha}|^2 \prod_{j=1}^{l} |z_j^{\alpha_j}|^2 \mu_i
    &<& |t|^{\frac{2}{l+1}\alpha_i} \left( \prod_{j=1,\, j\neq i}^{l} \int_{M}^{\frac{l}{l+1} y_t} |z_j|^2 \sigma_j^{2k} d\sigma_j \right) \int_{|z_j|<R_j} |a_{\alpha}|^2 \nu_l \\
    &<& ((2k)!4\pi)^l |t|^{\frac{2}{l+1}\alpha_i} \frac{\alpha_i+1}{R_i^{\alpha_i+1}} \left( \prod_{j=1}^{l} \frac{R_j^{\alpha_j+1}}{\alpha_j+1} \right) \int_{|z_j|<R_j} |a_{\alpha}|^2 \nu_l.
\end{eqnarray*}
Let $\mathcal{S}_t(\alpha) = |t|^{\frac{2}{l+1}\alpha_i} \frac{\alpha_i+1}{R_i^{\alpha_i+1}} \prod_{j=1}^{l} \frac{R_j^{\alpha_j+1}}{\alpha_j+1}$.
Therefore,
\begin{eqnarray*}
    && \frac{1}{((2k)!4\pi)^l} \sum_{\alpha\in Q_J} \int_{U^i(t)} |a_{\alpha}|^2 \prod_{j=1}^{l} |z_j^{\alpha_j}|^2 \mu_i \\
    &\leq& \sum_{\alpha\in Q_J} \mathcal{S}_t(\alpha) \int_{|z_j|<R_j} |a_{\alpha}|^2 \nu_l \\
    &\leq& \sum_{\alpha\in Q_J} \mathcal{S}_t(\alpha) 2 \int_{|z_j|<R_j} \left[ |b_{\overline{0\alpha}}|^2 + \left| \sum_{\beta\in A_\alpha,\, \beta_0>0} b_\beta t^{\beta_0} \right|^2 \right] \nu_l \\
    &\leq& |t|^{\frac{2}{l+1}} \frac{4}{R^{2}} \left[ \frac{2C'}{k^{l}\epsilon_2(2\pi)^{l}} + \sum_{\alpha\in Q_J} \int_{|z_j|<R_j} C_2|t| \sum_{\beta\in A_\alpha,\, \beta_0>0} |b_\beta|^2 c_\beta \nu_l \right] \\
    &\leq& |t|^{\frac{2}{l+1}} \frac{4}{R^{2}} \left[ \frac{2C'}{k^{l}\epsilon_2(2\pi)^{l}} + \frac{2C_2C_1}{k^{l+1}\epsilon_1 (2\pi)^{l+1}} |t| \right] \\
    &\leq& |t|^{\frac{2}{l+1}} \frac{4}{R^{2}} \left[ \frac{3C'}{k^{l}\epsilon_2(2\pi)^{l}} \right],
\end{eqnarray*}
for $|t|$ sufficiently small.

Then, if we denote
\[
s_{t,J} = \frac{\sum_{\alpha \in Q_J} a_\alpha z^\alpha}{\prod_{i=1}^{l} z_i^{k+1}} (dz_1 \wedge \cdots \wedge dz_n)^{\otimes (k+1)},
\]
we have, for $|t|$ sufficiently small,
\[
\int_{U \cap X_t} \| s_{t,J} \|^2_{\mathrm{KE}} \, \omega_t^n < C_3,
\]
for some constant $C_3 > 0$ depending on $k$, but independent of $t$ and $U$.

\subsubsection*{Case 2.} When $\alpha \in Q_{J'}$ for some non-empty $J' \subsetneq J$, for simplicity, we may assume $J' = \{l', \ldots, l\}$ for some $1 < l' \leq l$. Let $I' = \{0, \ldots, l'-1\}$. Then $\tilde{s}$ is also an extension of $\tilde{s}|_{X_{0,I'}}$. By Theorem~\ref{thm-fin-total}, we have
\[
\| \tilde{s} \|^2_{h, \xcal} \geq \frac{c''}{k^{l'}} \| \tilde{s} \|^2_{h, X_{0,I'}},
\]
for some $c'' > 0$. Thus,
\[
\| \tilde{s} \|^2_{h, X_{0,I'}} \leq \frac{C''}{k^{l - l' + 1}} \| \tilde{s} \|^2_{h, X_{0,I}},
\]
for some $C'' > 0$.

Since
\[
\prod_{i=1}^{l'-1} \int_{M}^{\frac{l}{l+1} y_t} \sigma_i^{2k} d\sigma_i < c_1 y_t^{(l'-1)(2k+1)},
\]
for some $c_1 > 0$, we have
\[
B_{\alpha, t} < c_1 y_t^{(l'-1)(2k+1)} \prod_{i=l'}^{l} \int_{M}^{\frac{l}{l+1} y_t} |z_i|^{2\alpha_i} \sigma_i^{2k} d\sigma_i.
\]

On $U^0(t)$, the remainder of the argument is similar to Case 0, and we obtain
\[
\sum_{\alpha \in Q_{J'}} B_{\alpha, t} \int_{|z_j| < R_j} |a_\alpha|^2 \nu_l < c_2 y_t^{(l'-1)(2k+1)},
\]
for some $c_2$ independent of $U$ and $t$.

\medskip

On $U^i(t)$ for any $i \neq 0$ with $i \leq l$, there are two possibilities. If $i \in J'$, then, similar to Case 1,
\[
\int_{U^i(t)} \sum_{\alpha \in Q_{J'}} |a_\alpha|^2 \prod_{j=1}^{l} |z_j|^{2\alpha_j} \mu_i < c_3 |t|^{\frac{2}{l+1}} y_t^{l'(2k+1)},
\]
for some $c_3$ independent of $U$ and $t$. If $i \notin J'$, we have the same bound as in $U^0(t)$:
\[
\int_{U^i(t)} \sum_{\alpha \in Q_{J'}} |a_\alpha|^2 \prod_{j=1}^{l} |z_j|^{2\alpha_j} \mu_i < c_2 y_t^{(l'-1)(2k+1)}.
\]

Therefore, if we denote
\[
s_{t,J'} = \frac{\sum_{\alpha \in Q_{J'}} a_\alpha z^\alpha}{\prod_{i=1}^{l} z_i^{k+1}} (dz_1 \wedge \cdots \wedge dz_n)^{\otimes (k+1)},
\]
then for $|t|$ sufficiently small,
\begin{equation}\label{in-st-star}
    \int_{U \cap X_t} \| s_{t,J'} \|^2_{\mathrm{KE}} \, \omega_t^n < c_4 y_t^{(l'-1)(2k+1)},
\end{equation}
for some constant $c_4$ independent of $t$ and $U$.

\subsubsection*{Case 3.} Assume $0 > \alpha_{i_0} = \min\{\alpha_i \mid i \leq l\}$. Then on $U^{i_0}(t)$, using
\[(z_0, \ldots, z_{i_0-1}, z_{i_0+1}, \ldots, z_n)\]  as coordinates, the term $z^\alpha$ can be rewritten as
\[
z^\alpha = t^{\alpha_{i_0}} \left( \prod_{i=0,\, i \neq i_0}^{l} z_i \right)^{-\alpha_{i_0}} \prod_{i=1,\, i \neq i_0}^{l} z_i^{\alpha_i} = t^{\alpha_{i_0}} z_0^{-\alpha_{i_0}} \prod_{i=1,\, i \neq i_0}^{l} z_i^{\alpha_i - \alpha_{i_0}},
\]
which reduces to either Case 1 or Case 2 on $U^{i_0}(t)$, as previously discussed.

\medskip

Combining the estimates from all three cases, we obtain
\begin{equation}\label{e-st-star}
    \int_{U \cap X_t} \| s_{t,*} \|^2_{\mathrm{KE}} \, \omega_t^n < c_5 y_t^{(l-1)(2k+1)},
\end{equation}
for some constant $c_5$ independent of $t$ and $U$, where
\[
s_{t,*} = \frac{\sum_{\alpha \neq 0} a_\alpha z^\alpha}{\prod_{i=1}^{l} z_i^{k+1}} (dz_1 \wedge \cdots \wedge dz_n)^{\otimes (k+1)}.
\]

\medskip

Let $\tcal$ denote the region where $\tau_I > \frac{3}{2}M$. Then, by inequalities~\ref{e-a-0-int} and~\ref{e-st-star}, we have
\begin{equation}\label{e-tcal-u}
    \| s_t \|^2_{\mathrm{KE},\, \tcal \cap U} > c_6 y_t^{l(2k+1)} \int_{|z_j| < R_j} |b_{(0,\ldots,0)}|^2 (\sqrt{-1})^{n-l} \bigwedge_{j=l+1}^n (dz_j \wedge d\bar{z}_j),
\end{equation}
for some constant $c_6$ independent of $t$ and $U$.

Therefore, the $L^2$-norm of $s_t$ over $\tcal$ satisfies
\begin{equation}\label{e-st-tcal}
    \| s_t \|^2_{\mathrm{KE},\, \tcal} > c_7 y_t^{l(2k+1)} \| s \|^2_{h,\, \omega,\, X_{0,I}},
\end{equation}
for some constant $c_7$ independent of $t$.

Moreover, by inequalities~\ref{e-geq-4m} and~\ref{in-st-star}, we obtain
\begin{equation}\label{e-l2-mu}
    \int_{\mcal_U} \| s_t \|^2_{\mathrm{KE}} \, \omega_t^n < c_8 y_t^{(l-1)(2k+1)},
\end{equation}
for some $c_8 > 0$.

Define $s_{tn} = \frac{1}{\| s_t \|_{\mathrm{KE},\, \tcal}} s_t$. Then the total $L^2$-norm of $s_{tn}$ over the region $\mcal_U$ is bounded by $c_9 y_t^{-(2k+1)/2}$ for some $c_9 > 0$.

Thus, we have the following proposition:
\begin{proposition}\label{prop-stn-ncal}
    There exists $c_{10} > 0$ such that, for $|t|$ sufficiently small,
    \[
    \int_{\ncal_I} \| s_{tn} \|^2_{\mathrm{KE}} \, \omega_t^n < c_{10} y_t^{-(2k+1)}.
    \]
\end{proposition}

\subsection{When $D_I$ is not empty}
Let $s \in \gcal_{I,k+1} \subset H^0(X_{0,I}, (k+1)L)$ be such that $s = S_{\hat{I}} \otimes s'$ for some $s' \in H^0(X_{0,I}, (k+1)L - \sum_{j \notin I} [X_{0,j}])$ with unit $L^2$-norm with respect to $h_{\hat{I},k+1}$ and $\omega$. Since we have chosen $\|S_i\|_i$ to be sufficiently small for each $i$, it follows that
\[
\| \tilde{s} \|^2_{h, \omega, \xcal} < \frac{c}{k^{l+1}},
\]
for some constant $c > 0$.

When $(U, z) \in \ucal^I$, the proof proceeds as in the case $D_I$ is empty, yielding the same estimates as in formulas~\ref{e-tcal-u} and~\ref{e-l2-mu}.

Now consider $(U, z) \in \ucal^J$ for some $J$ with $I \subsetneq J$. Without loss of generality, assume $I = \{0, \ldots, l\}$ and $J = \{0, \ldots, l_1\}$. Then $\tilde{s} = f (dz_0 \wedge \cdots \wedge dz_n)^{\otimes (k+1)}$ for some holomorphic function $f$ on $U$. In these coordinates,
\[
f = \sum_{\beta \in \Upsilon} b_\beta \prod_{i=0}^{l_1} z_i^{\beta_i},
\]
where $\beta = (\beta_0, \ldots, \beta_{l_1})$ runs over multi-indices and $b_\beta$ are holomorphic functions of $(z_{l_1+1}, \ldots, z_n)$. Since $\tilde{s}$ vanishes to order at least $1$ along $X_{0,j}$ for $j \in J \setminus I$, the sum is over $\beta$ with $\beta_j \geq 1$ for all $j \in J \setminus I$.

On $U^0(t)$, using $(z_1, \ldots, z_n)$ as coordinates, we can write $f = \sum_{\alpha} a_\alpha z^\alpha$, where $a_\alpha$ are holomorphic functions of $(z_{l_1+1}, \ldots, z_n)$ and $z^\alpha = \prod_{j=1}^{l_1} z_j^{\alpha_j}$. For each $\alpha$,
\[
a_\alpha = \sum_{\tau = 0}^{\infty} \sum_{\substack{\beta \in A_\alpha \\ \beta_0 = \tau}} b_\beta t^\tau.
\]
If $\alpha_j < 0$ for some $l+1 \leq j \leq l_1$, then $\{\beta \in A_\alpha \mid \beta_0 = 0\} = \emptyset$. Thus, by formula~\ref{e-alpha-not0},
\begin{equation}\label{e-alpha-not-empty}
    |a_\alpha|^2 \leq C_5 |t| \sum_{\substack{\beta \in A_\alpha \\ \beta_0 > 0}} |b_\beta|^2 c_\beta,
\end{equation}
for some constant $C_5 > 0$ independent of $\alpha$ and $U$, provided $|t|$ is sufficiently small. Therefore, it suffices to focus on those $\alpha$ with $\alpha_j \geq 1$ for all $l+1 \leq j \leq l_1$.

Let $Q_+$ denote the set of $\alpha$ with $\alpha_j > 0$ for all $j$. Then, as in case 1 for $D_I$ empty, we have for $|t|$ small enough,
\[
\int_{U \cap X_t} \| s_{t,+} \|^2_{\ke} \omega_t^n < C_6,
\]
for some constant $C_6 > 0$ independent of $t$ and $U$, where
\[
s_{t,+} = \frac{\sum_{\alpha \in Q_+} a_\alpha z^\alpha}{\prod_{i=1}^{l_1} z_i^{k+1}} (dz_1 \wedge \cdots \wedge dz_n)^{\otimes (k+1)}.
\]

Let $Q_{\geq 0}$ denote the set of $\alpha$ with $\alpha_j \geq 0$ for all $j$. As in case 2 for $D_I$ empty, since there are at most $l$ indices $j$ with $\alpha_j = 0$, we obtain
\[
\sum_{\alpha \in Q_{\geq 0}} B_{\alpha, t} \int_{|z_j| < R_j} |a_\alpha|^2 \nu_{l_1} < C_7 y_t^{l(2k+1)},
\]
for some $C_7$ independent of $U$ and $t$. Moreover, if we set
\[
s_{t, \geq 0} = \frac{\sum_{\alpha \in Q_{\geq 0}} a_\alpha z^\alpha}{\prod_{i=1}^{l_1} z_i^{k+1}} (dz_1 \wedge \cdots \wedge dz_n)^{\otimes (k+1)},
\]
then
\[
\int_{U \cap X_t} \| s_{t, \geq 0} \|^2_{\ke} \omega_t^n < C_8 y_t^{l(2k+1)},
\]
for some constant $C_8 > 0$ independent of $t$ and $U$. In particular, when $\alpha_j = 0$ for $j \leq l$, by formula~\ref{e-geq-4m},
\[
\int_{\mcal_U} \| s_{t, \geq 0} \|^2_{\ke} \omega_t^n < C_9 y_t^{(l-1)(2k+1)},
\]
for some constant $C_9 > 0$ independent of $t$ and $U$.
For those $\alpha$ with at least one negative entry, the arguments are similar to those in case 3 when $D_I$ is empty, namely, they reduce to cases 1 or 2.

The case when $s \in \hcal_{I,k+1,2} \subset H^0(X_{0,I}, (k+1)L)$ is analogous.

In summary, as in the case when $D_I$ is empty, we obtain the following conclusions. When $s \in \gcal_{I,k+1}$ or $s \in \hcal_{I,k+1,2}$, we have:
\begin{itemize}
    \item There exists a constant $C_{10}$ independent of $t$ such that
    \begin{equation}\label{e-st-di-nempty}
        \| s_t \|^2_{\mathrm{KE},\, \tcal} > C_{10} y_t^{l(2k+1)} \| s \|^2_{h,\, \omega,\, X_{0,I} \setminus \tilde{D}_I},
    \end{equation}
    where $\tilde{D}_I$ is the union of those $U \in \ucal$ centered at points in $D_I$.
    \item There exists $C_{11} > 0$ such that, for $|t|$ sufficiently small,
    \begin{equation}\label{e-stn-di-nempty}
        \int_{\ncal_I} \| s_{tn} \|^2_{\mathrm{KE}}\, \omega_t^n < C_{11} y_t^{-(2k+1)}.
    \end{equation}
\end{itemize}

\section{Proof of Theorem~\ref{thm-almost-orthon}}\label{sec-5}

\begin{theorem}[Theorem~\ref{thm-almost-orthon}]
    The sections in $\bcal_{t}$ are almost orthonormal in the sense that
    \[
    \lim_{u \to \infty} \qcal_{t_u} = \mathrm{Id}.
    \]
\end{theorem}

To prove this theorem, we divide the argument into two parts:
\begin{itemize}
    \item The sections in each $\bcal_{I,t}$ are almost orthonormal, i.e., the inner product matrix $\qcal_{I,t}$ satisfies
    \[
    \lim_{u \to \infty} \qcal_{I,t_u} = \mathrm{Id}.
    \]
    \item For $I \neq J$, the sections $\varsigma_1 \in \bcal_{I,t}$ and $\varsigma_2 \in \bcal_{J,t}$ are almost orthogonal to each other.
\end{itemize}

Let $\{t_u\}$ be a good sequence, which we fix throughout this section. For simplicity of notation, we continue to write $t$ for an arbitrary element in this sequence.
\subsection{Part 1}
\subsubsection*{When $D_I$ is empty}

When $|I| \geq 2$, we define a volume form on $X_{0,I}^o$ by
\[
\nu_{\{t_u\},I,k+1} \triangleq \vcal_I (4\pi)^l \int_{B_I} \left( 4\left(1-\sum_{j=1}^{l} x_j\right)^2 \prod_{j=1}^{l} 4x_j^2 \right)^k e^{-k\varphi_I} dx_1 \cdots dx_l.
\]
Since $\varphi_I$ is bounded, $\nu_{\{t_u\},I,k+1}$ is also of Poincar\'e type near $D_I$.

Let $\hcal_{\{t_u\},I,k+1}$ denote the space of holomorphic $L^2$-integrable sections of $(k+1)K_I$ on $X_{0,I}^o$, with norm squared given by
\[
\int_{X_{0,I}^o} \| \cdot \|^2_{\vcal_I} \, \nu_{\{t_u\},I,k+1}.
\]

For any $U \in \ucal^I$ with coordinates $(z_0, \ldots, z_n)$, let $f$ be the holomorphic function representing $\tilde{s}$. As in Case 0 of subsection~\ref{subsec-case-0}, we may ignore all terms except $b_0 = b_{(0,\ldots,0)}$, i.e., we may assume $b_\beta = 0$ for $\beta \neq 0$.

We can write
\[
\omega_t^n = e^{-\psi_t} \frac{(\sqrt{-1})^n}{\prod_{j=1}^l |z_j|^2} \bigwedge_{i=1}^n dz_i \wedge d\bar{z}_i.
\]
Since $P_t^* dz_i \wedge d\bar{z}_i = (\log |t|)^2 |z_i|^2 dw_i \wedge d\bar{w}_i$ for $1 \leq i \leq l$, it follows that
\[
e^{\psi_t \circ P_t} P_t^* \omega_t^n = (\log |t|)^{2l} (\sqrt{-1})^n \bigwedge_{i=1}^l (dw_i \wedge d\bar{w}_i) \wedge \bigwedge_{j=l+1}^n (dz_j \wedge d\bar{z}_j).
\]

Thus, we need to estimate the integral
\[
\int_{U \cap X_t} |b_0|^2 e^{k\psi_t} \frac{(\sqrt{-1})^n}{\prod_{j=1}^l |z_j|^2} \bigwedge_{i=1}^n dz_i \wedge d\bar{z}_i.
\]
By inequality~\ref{e-geq-4m}, it suffices to estimate the integral over $^M U \cap X_t$, where $^M U = \{ q \in U \mid \sigma_i > 4M,\, 0 \leq i \leq l \}$.

For simplicity, let $d\mu = (\sqrt{-1})^n \bigwedge_{i=1}^l (dw_i \wedge d\bar{w}_i) \wedge \bigwedge_{j=l+1}^n (dz_j \wedge d\bar{z}_j)$. If we write $\vcal_I = e^{\psi_{I,U}} (\sqrt{-1})^{n-l} dz_{l+1} \wedge d\bar{z}_{l+1} \wedge \cdots \wedge dz_n \wedge d\bar{z}_n$, then
\[
\lim_{u \to \infty} (\log |t_u|)^{2l} e^{-\phi_{t_u} \circ P_{t_u}} = \frac{e^{\psi_{I,U} + \varphi_I}}{4(1-\sum_{j=1}^l x_j)^2} \prod_{j=1}^l \frac{1}{4x_j^2}.
\]

For any compact subset $K \subset B_I$, define
\[
\nu_{\{t_u\},I,K,k+1} \triangleq \vcal_I (4\pi)^l \int_K \left( 4(1-\sum_{j=1}^l x_j)^2 \prod_{j=1}^l 4x_j^2 \right)^k e^{-k\varphi_I} dx_1 \cdots dx_l.
\]
Then,
\begin{align*}
& (\log |t|)^{-l(2k+1)} \int_{P_t(G_K)} |b_0|^2 e^{k\psi_t} \frac{(\sqrt{-1})^n}{\prod_{j=1}^l |z_j|^2} \bigwedge_{i=1}^n dz_i \wedge d\bar{z}_i \\
&= \int_{K \times (\sqrt{-1}[0, \frac{2\pi}{-\log|t|}])^l \times (U \cap X_{0,I})} |b_0|^2 \frac{e^{k\psi_t}}{(\log|t|)^{2kl}} (\log|t|)^l d\mu \\
&\to (4\pi)^l \int_{K \times (U \cap X_{0,I})} |b_0|^2 e^{-k(\psi_{I,U} + \varphi_I)} \left[ 4(1-\sum_{j=1}^l x_j)^2 \prod_{j=1}^l 4x_j^2 \right]^k d\mu \\
&= \int_{U \cap X_{0,I}} |b_0|^2 e^{-k\psi_{I,U}} \nu_{\{t_u\},I,K,k+1},
\end{align*}
as $t = t_u \to 0$. Let
\[
K_\delta = \left\{ (x_1, \ldots, x_l) \in \mathbb{R}^l \;\middle|\; \sum_{j=1}^l x_j \leq 1-\delta,\; x_j \geq \delta,\; 1 \leq j \leq l \right\}.
\]
Since $\varphi_I$ is bounded, it is clear that as $\delta \to 0$, $\nu_{\{t_u\},I,K_\delta,k+1}$ converges uniformly to $\nu_{\{t_u\},I,k+1}$. Therefore,
\[
\int_{U \cap X_{0,I}} |b_0|^2 e^{-k\psi_{I,U}} \nu_{\{t_u\},I,K_\delta,k+1} \to \int_{U \cap X_{0,I}} |b_0|^2 e^{-k\psi_{I,U}} \nu_{\{t_u\},I,k+1}
\]
as $\delta \to 0$.

On the other hand, we can estimate
\[
(\log |t|)^{-l(2k+1)} \int_{U^i(t)\setminus P_t(G_{K_\delta})} |b_0|^2 e^{k\psi_t} \frac{(\sqrt{-1})^n}{\prod_{j=1}^{l}|z_j|^2} \bigwedge_{j=1}^{n} dz_j \wedge d\bar{z}_j,
\]
for each $0 \leq i \leq n$. For example, when $i = 0$, since
\[
U^0(t) \setminus P_t(G_{K_\delta}) = \left\{ q \in U^0(t) \;\middle|\; \sigma_j < \delta y_t \text{ for some } 1 \leq j \leq l \right\},
\]
we have
\begin{align*}
    &\int_{U^0(t)\setminus P_t(G_{K_\delta})} |b_0|^2 e^{k\psi_t} \frac{(\sqrt{-1})^n}{\prod_{j=1}^{l}|z_j|^2} \bigwedge_{j=1}^{n} dz_j \wedge d\bar{z}_j \\
    &\leq C^k \int_{U^0(t)\setminus P_t(G_{K_\delta})} |b_0|^2 \left( \prod_{j=1}^{l} |\sigma_j| \right)^{2k} \frac{(\sqrt{-1})^n}{\prod_{j=1}^{l}|z_j|^2} \bigwedge_{j=1}^{n} dz_j \wedge d\bar{z}_j \\
    &\leq C^k l \left( \int_M^{y_t} x^{2k} dx \right)^{l-1} \int_{M}^{\delta y_t} x^{2k} dx \int_{|z_j|<R_j,\, j\geq l+1} |b_0|^2 \nu_l,
\end{align*}
for some constant $C$ independent of $U$ and $t$. It is then straightforward to see that
\[
\lim_{\delta \to 0} (\log |t|)^{-l(2k+1)} \int_{U^0(t)\setminus P_t(G_{K_\delta})} |b_0|^2 e^{k\psi_t} \frac{(\sqrt{-1})^n}{\prod_{j=1}^{l}|z_j|^2} \bigwedge_{j=1}^{n} dz_j \wedge d\bar{z}_j = 0.
\]
The arguments for the other $U^i(t)$ are analogous. Therefore, we have shown that
\[
\lim_{u \to \infty} \| s_{t_u,\mathrm{rn}} \|_{\ke,\, U \cap X_{t_u}} = \int_{U \cap X_{0,I}} \| s \|_{\vcal_I} \, \nu_{\{t_u\},I,k+1}.
\]
For any open subset $W' \subset \{ (z_{l+1},\ldots,z_n) \mid |z_j| < R_j \}$, let $U' = \{ z \in U \mid (z_{l+1},\ldots,z_n) \in W' \}$. Then the same argument applies to $U'$, namely,
\[
\lim_{u \to \infty} \| s_{t_u,\mathrm{rn}} \|_{\ke,\, U' \cap X_{t_u}} = \int_{U' \cap X_{0,I}} \| s \|_{\vcal_I} \, \nu_{\{t_u\},I,k+1}.
\]
This implies that
\[
\lim_{u \to \infty} \| s_{t_u,\mathrm{rn}} \|_{\ke,\, X_{t_u}} = \int_{X_{0,I}} \| s \|_{\vcal_I} \, \nu_{\{t_u\},I,k+1} = 1.
\]
Therefore,
\[
\lim_{u \to \infty} \qcal_{I,t_u} = \mathrm{Id}.
\]

\subsubsection*{When $D_I$ is not empty}
We first assume that $|I| > 1$.

Let $s \in \gcal_{I,k+1}$. On each $U \in \ucal^I$, we can apply the same arguments as in the case $D_I$ is empty to obtain
\[
\lim_{u \to \infty} \| s_{t_u,\mathrm{rn}} \|^2_{\ke,\, U' \cap X_{t_u}} = \int_{U' \cap X_{0,I}} \| s \|^2_{\vcal_I} \, \nu_{\{t_u\},I,k+1},
\]
for any sub-polydisc $U'$ as defined previously.
Consequently,
\[
\lim_{u \to \infty} \| s_{t_u,\mathrm{rn}} \|^2_{\ke,\, X_{t_u} \cap (\bigcup_{U \in \ucal^I} U)} = \int_{X_{0,I} \cap (\bigcup_{U \in \ucal^I} U)} \| s \|^2_{\vcal_I} \, \nu_{\{t_u\},I,k+1}.
\]
For a fixed $M_1 > 0$, denote $^{M_1}\ucal^{I} = \{\,^{M_1}U \mid U \in \ucal^I \}$, i.e., we replace each $U \in \ucal^I$ by $^{M_1}U$. Then, by formula~\ref{e-geq-sqrt-yt}, we also have
\[
\lim_{u \to \infty} \| s_{t_u,\mathrm{rn}} \|^2_{\ke,\, X_{t_u} \cap (\bigcup_{U \in\, ^{M_1}\ucal^{I}} U)} = \int_{X_{0,I} \cap (\bigcup_{U \in\, ^{M_1}\ucal^{I}} U)} \| s \|^2_{\vcal_I} \, \nu_{\{t_u\},I,k+1}.
\]

For each $U \in \ucal$ centered at a point $p \in D_I$, let $(z_0, \ldots, z_n)$ be coordinates such that
\[
X_{0,I} \cap U = \{ z \mid z_i = 0,\, 0 \leq i \leq l \}.
\]
Define $^{M_1}U = \{ z \in U \mid \sigma_i > M_1,\, 0 \leq i \leq l \}$, and let $^{M_1}\ucal^{I,D}$ denote the collection of such $^{M_1}U$. We call an appropriate coordinate patch centered at $p \in X_{0,I}^o$ \emph{$M_1$-admissible} if $^{M_1}U$ has compact closure in $U$.

Given any $\delta_1 > 0$, we can choose $M_1 > M$ so that, by adding finitely many $M_1$-admissible appropriate coordinate patches to $^{M_1}\ucal^{I}$, we obtain a cover of $X_{0,I}^{\delta_1}$. Denote this enlarged collection by $^{M_1}\ucal^{I}_{\delta_1}$. Clearly, we also have
\[
\lim_{u \to \infty} \| s_{t_u,\mathrm{rn}} \|^2_{\ke,\, X_{t_u} \cap (\bigcup_{U \in\, ^{M_1}\ucal^{I}_{\delta_1}} U)} = \int_{X_{0,I} \cap (\bigcup_{U \in\, ^{M_1}\ucal^{I}_{\delta_1}} U)} \| s \|^2_{\vcal_I} \nu_{\{t_u\},I,k+1}.
\]

Since $\lim_{M_2 \to \infty} \int_{M_2}^\infty e^{-a x} x^{2k} dx = 0$, for any $\epsilon_1 > 0$, we can find $\delta_1$ such that
\[
\| s_{t_u,\mathrm{rn}} \|^2_{\ke,\, X_{t_u} \cap N_{D_I}^{\delta_1}} < \epsilon_1
\]
for $u$ sufficiently large, where $N_{D_I}^{\delta_1}$ is the $\delta_1$-neighborhood of $D_I$ in $\xcal$, and
\[
\int_{X_{0,I} \setminus X_{0,I}^{\delta_1}} \| s \|^2_{\vcal_I} \nu_{\{t_u\},I,k+1} < \epsilon_1.
\]
We can then choose $M_1 > M$ so that the $M_1$-admissible appropriate coordinate patches cover $X_{0,I}^{\delta_1}$. Thus,
\[
\lim_{u \to \infty} \| s_{t_u,\mathrm{rn}} \|^2_{\ke,\, X_{t_u} \setminus W_{I,M_1}} = 0,
\]
where $W_{I,M_1} = \bigcup_{U \in\, ^{M_1}\ucal^{I}_{\delta_1} \cap\, ^{M_1}\ucal^{I,D}} U$.
Therefore, in conclusion,
\[
\lim_{u \to \infty} \| s_{t_u,\mathrm{rn}} \|^2_{\ke,\, X_{t_u}} = \int_{X_{0,I}^o} \| s \|^2_{\vcal_I} \nu_{\{t_u\},I,k+1} = 1.
\]

The same argument applies when $s \in \hcal_{I,k+1,2}$, so we also have
\[
\lim_{u \to \infty} \| s_{t_u,\mathrm{rn}} \|^2_{\ke,\, X_{t_u}} = \int_{X_{0,I}^o} \| s \|^2_{\vcal_I} \nu_{\{t_u\},I,k+1} = 1.
\]

For any $s \in \hcal_{I,k+1}$ of unit norm, write $s = \lambda_1 s_1 + \lambda_2 s_2$ with $s_1 \in \gcal_{I,k+1}$, $s_2 \in \hcal_{I,k+1,2}$, and $|\lambda_1|^2 + |\lambda_2|^2 = 1$. Let $\tilde{s} = \lambda_1 \tilde{s}_1 + \lambda_2 \tilde{s}_2$, and define $s_t$ and $s_{t_u,\mathrm{rn}}$ accordingly. Then
\[
\| \tilde{s} \|_{h,\omega,\xcal} \leq \| \tilde{s}_1 \|_{h,\omega,\xcal} + \| \tilde{s}_2 \|_{h,\omega,\xcal}.
\]
Repeating the arguments above for $s_1$ and $s_2$ shows that
\[
\lim_{u \to \infty} \| s_{t_u,\mathrm{rn}} \|^2_{\ke,\, X_{t_u}} = 1.
\]

Therefore,
\[
\lim_{u \to \infty} \qcal_{I,t_u} = \mathrm{Id}.
\]
\
When $|I|=1$, i.e., $I = \{i\}$ for some $i \leq m$, if we replace both $\vcal_I$ and $\nu_{\{t_u\},I,k+1}$ by $\omega_0^n$, then by Theorem~\ref{thm-ruan}, the arguments for the case $|I| > 1$ apply verbatim, yielding
\[
\lim_{u \to \infty} \qcal_{I,t_u} = \mathrm{Id}.
\]
This completes the proof of the first part.

\subsection{Part 2}
If $X_{0,I} \cap X_{0,J} = \emptyset$, then it is clear from the construction of $s_{t,\mathrm{rn}}$ that
\[
\lim_{u \to \infty} \left\langle s^I_{t_u,\mathrm{rn}}, s^J_{t_u,\mathrm{rn}} \right\rangle_{\ke, X_t} = 0,
\]
for any $s^I \in \hcal_{I,k+1}$ and $s^J \in \hcal_{J,k+1}$.

If $X_{0,I}$ and $X_{0,J}$ are not disjoint, we have $X_{0,I} \cap X_{0,J} \subset D_J$. By inequality~\ref{e-geq-sqrt-yt} and repeating the arguments in the proof of Theorem~\ref{thm-inner-sections-prop}, one sees that $s^I_{t_u,\mathrm{rn}}$ has negligible mass outside a $\sqrt{y_t}$-neighborhood of $X_{0,I}$. Similarly, $s^J_{t_u,\mathrm{rn}}$ has negligible mass inside a $\sqrt{y_t}$-neighborhood of $D_J$. To see this, recall the following basic lemma from \cite{sun2024ChengYau}:

\begin{lem}\label{lem-concave}
    Let $f(x)$ be a concave function. Suppose $f'(x_0) < 0$, then
    \[
    \int_{x_0}^\infty e^{f(x)}\,dx \leq \frac{e^{f(x_0)}}{-f'(x_0)}.
    \]
\end{lem}

Since the exponent $-b\sigma + 2k\log \sigma$ is a concave function of $\sigma$, we have
\[
\int_{\sqrt{y_t}}^{\frac{l}{l+1}y_t} e^{-b\sigma + 2k\log \sigma}\,d\sigma < (b - \tfrac{1}{2})^{-1} e^{-b\sqrt{y_t} + k\log y_t}.
\]
Also, clearly,
\[
\int_M^{\sqrt{y_t}} e^{-b\sigma + 2k\log \sigma}\,d\sigma > e^{-b(M+1) + 2k\log(M+1)}.
\]
Therefore, when $b \geq 1$ and $|t|$ is sufficiently small,
\[
\frac{\int_{\sqrt{y_t}}^{\frac{l}{l+1}y_t} e^{-b\sigma + 2k\log \sigma}\,d\sigma}{\int_M^{\sqrt{y_t}} e^{-b\sigma + 2k\log \sigma}\,d\sigma} < e^{-\frac{1}{2}\sqrt{y_t}}.
\]

Thus, the off-diagonal inner products decay rapidly, and we have established the almost orthogonality. This completes the proof of Theorem~\ref{thm-almost-orthon}.

\section{Almost Bergman embeddings}\label{sec-almost-bergman}
\subsection{Embedding of projective bundles}\label{subsec-pro-bundle}

Let $p_i : (\CP^1)^l \to \CP^1$ denote the projection onto the $i$-th factor. Set $\ocal(1,\ldots,1) = \bigotimes_{i=1}^l p_i^* \ocal(1)$. A choice of basis for $H^0((\CP^1)^l, \ocal(1,\ldots,1))$ induces the Segre embedding
\[
\se_l : (\CP^1)^l \to \CP^{2^l - 1}.
\]
A standard form in homogeneous coordinates is
\[
[a_{1,1}, a_{1,2}] \times \cdots \times [a_{l,1}, a_{l,2}] \mapsto [b_1 \prod_{i=1}^l a_{i,1},\; b_2 a_{1,2} \prod_{i=2}^l a_{i,1},\; \ldots,\; b_{2^l} \prod_{i=1}^l a_{i,2}],
\]
where $b_j \neq 0$ for each $1 \leq j \leq 2^l$. We say the Segre embedding is of type $\{b_j\}_{1 \leq j \leq 2^l}$.

\
Following the convention, we denote by $0$ and $\infty$ the points $[1,0]$ and $[0,1]$ in $\CP^1$, respectively. 
Thus, the tuples $(x_1,\ldots,x_l) \in (\CP^1)^l$ with $x_i = 0$ or $\infty$ are mapped under the Segre embedding to points $\{\gamma_i\}_{1 \leq i \leq 2^l}$, where each $\gamma_i$ has homogeneous coordinates $[z_1, \ldots, z_{2^l}]$ with $z_i = 1$ and $z_j = 0$ for $j \neq i$. We refer to the image $\se_l((\CP^1)^l)$ as a parallelogram with vertices $\{\gamma_i\}_{1 \leq i \leq 2^l}$. 
A different choice of the coefficients $\{b_j\}_{1 \leq j \leq 2^l}$ yields a different parallelogram; to distinguish them, we say the parallelogram is of type $\{b_j\}_{1 \leq j \leq 2^l}$. Here, both $\{\gamma_i\}$ and $\{b_j\}$ are considered as ordered sets.

Via the map $z \mapsto [1, z]$, we identify $\C$ with $\CP^1 \setminus \{\infty\}$. Thus, we have an open subset $\C^l \subset (\CP^1)^l$, and the standard Segre embedding restricted to $\C^l$ can be written as
\[
(z_1, \ldots, z_l) \mapsto [b_1, b_2 z_1, \ldots, b_{2^l} \prod_{i=1}^l z_i].
\]

More generally, given $2^l$ linearly independent points $\{\beta_i\}_{1 \leq i \leq 2^l}$ in $\CP^N$ with $N \geq 2^l - 1$, let $V$ be a subvariety containing all these points and biholomorphic to $(\CP^1)^l$. Suppose that, after a unitary transformation $H$ of $\CP^N$, each $\beta_i$ is mapped to $[e_i]$ (i.e., the point with only the $i$-th coordinate nonzero), and $HV$ is the image of the composition of the natural inclusion $\CP^{2^l-1} \subset \CP^N$ with a Segre embedding of type $\{b_j\}_{1 \leq j \leq 2^l}$. In this case, we still call $V$ the parallelogram of type $\{b_j\}_{1 \leq j \leq 2^l}$ with vertices $\{\beta_i\}_{1 \leq i \leq 2^l}$.

Let $\pi_i: A_i \to M$, $1 \leq i \leq l$, be line bundles over a connected compact complex manifold $M$. Let $\hat{A}_i = A_i \cup M_i$ denote the projective completion of the total space $A_i$. Let $L$ be a positive line bundle over $M$, and $F$ an arbitrary line bundle over $M$. 
Let $\pi_{i,\infty}: \hat{A}_i \to M_i$ be the projection, and let $L$ on $\hat{A}_i$ denote the pullback $\pi_i^* L$. We use $M \subset \hat{A}_i$ to denote the zero section of $A_i$.

Let $P \to M$ be the principal $(\mathbb{C}^*)^l$-bundle associated to the vector bundle $\bigoplus_{i=1}^l A_i \to M$. We fix a $\mathbb{C}^*$-action on $\CP^1$ given by $\lambda [a_1, a_2] = [\lambda a_1, a_2]$. This induces a $(\C^*)^l$-action $\varrho$ on $(\CP^1)^l$. We then define the $(\CP^1)^l$-bundle $\pi: \Theta_{\{A_1, \ldots, A_l\}} \to M$ by
\[
\Theta_{\{A_1, \ldots, A_l\}} \coloneqq P \times_{\varrho} (\CP^1)^l.
\]
For simplicity, we refer to $\Theta_{\{A_1, \ldots, A_l\}}$ as $X$ when there is no confusion, and we also use the notation $\Theta_{\vartheta}$, where $\vartheta = \{A_1, \ldots, A_l\}$ is a set of line bundles.

Let $\varrho_i$ denote the $(\C^*)^l$-action on $\CP^1$ given by projecting $(\C^*)^l$ onto the $i$-th factor and composing with the above $\mathbb{C}^*$-action on $\CP^1$. Thus,
\[
\hat{A}_i = P \times_{\varrho_i} \CP^1.
\]
Since the projection $(\CP^1)^l \to \CP^1$ onto the $i$-th factor is $(\C^*)^l$-invariant, we obtain projections
\[
p_i: X \to \hat{A}_i.
\]
Define the line bundle $\varDelta \to X$ by $\varDelta = \bigotimes_{i=1}^l p_i^*([M_i])$. The restriction of $\varDelta$ to each fiber $\pi^{-1}(q)$ is $\ocal(1, \ldots, 1)$. For simplicity, we still denote by $L$ the line bundle $\pi^* L$ and by $F$ the line bundle $\pi^* F$.

\begin{proposition}\label{prop-x-m}
We have
\begin{equation}\label{e-prop-x}
    H^0(X, kL + F + \varDelta) \simeq \bigoplus_{(c_1, \ldots, c_l),\, c_i = 0,1} H^0\left(M, kL + F - \sum_{i=1}^l c_i A_i\right).
\end{equation}
\end{proposition}
\begin{proof}
We proceed by induction on $l$. Let $\pi': X' \to M$ denote the $(\CP^1)^{l-1}$-bundle associated to the line bundles $A_i$ for $i \leq l-1$, and let $\varDelta'$ be the corresponding line bundle on $X'$. Clearly, $\bar{\pi}: X \to X'$ is a $\CP^1$-bundle, and we have the following commutative diagram:
\begin{equation*}
    \xymatrix{
        X\ar[r]^{p_l}\ar[d]_{\bar{\pi}}& \hat{A}_l\ar[d]^{\pi_l}\\
        X'\ar[r]^{\pi'}&M
    }
\end{equation*}
There is a natural inclusion $X' \subset X$ as the zero section of $\pi'^* A_l$, and we denote $X'_\infty = p_l^{-1}(M_l)$. Then $\varDelta = \bar{\pi}^* \varDelta' + [X'_\infty]$. Fix a section $s' \in H^0(X, [X'])$ vanishing along $X'$. This gives rise to the short exact sequence:
\[
0 \to H^0(X, kL + F + \varDelta - [X']) \xrightarrow{\otimes s'} H^0(X, kL + F + \varDelta) \to H^0(X', kL + F + \varDelta) \to 0.
\]
We claim that $H^1(X, kL + F + \varDelta - [X']) = 0$ for $k$ sufficiently large. Indeed, since $kL + F + \varDelta - [X']$ is trivial along the fibers of $\bar{\pi}$, and its restriction to $X'$ is $kL + F + \varDelta' - \pi'^* A_l$, we have
\[
kL + F + \varDelta - [X'] = \bar{\pi}^*(kL + F + \varDelta' - \pi'^* A_l).
\]
Because $H^i(\CP^1, \mathcal{O}) = 0$ for $i > 0$, it follows that $R^i \bar{\pi}_* (kL + F + \varDelta - [X']) = 0$ for $i > 0$. Therefore,
\[
H^1(X, kL + F + \varDelta - [X']) \cong H^1(X', kL + F + \varDelta' - \pi'^* A_l).
\]
Similarly, since $H^i((\CP^1)^{l-1}, \mathcal{O}(1, \ldots, 1)) = 0$ for $i > 0$, we have
\[
R^i \pi'_* (kL + F + \varDelta' - \pi'^* A_l) = 0 \quad \text{for } i > 0.
\]
Thus,
\[
H^1(X', kL + F + \varDelta' - \pi'^* A_l) \cong H^1(M, kL + F - A_l + E),
\]
where $E = \pi'_* \varDelta'$ is a vector bundle on $M$. Since $L$ is ample, the Kodaira vanishing theorem implies $H^1(M, kL + F - A_l + E) = 0$ for $k$ sufficiently large. Therefore, $H^1(X, kL + F + \varDelta - [X']) = 0$ as claimed.

It follows that the restriction map $H^0(X, kL + F + \varDelta) \to H^0(X', kL + F + \varDelta)$ is surjective. By similar arguments, the restriction map
\[
H^0(X, kL + F + \varDelta - [X']) \to H^0(X'_\infty, kL + F + \varDelta - [X'])
\]
is an isomorphism. Noting that $H^0(X'_\infty, kL + F + \varDelta - [X']) \cong H^0(X', kL + F + \varDelta' - \pi'^* A_l)$ and $H^0(X', kL + F + \varDelta) \cong H^0(X', kL + F + \varDelta')$, we obtain
\begin{equation}\label{e-h0-x}
H^0(X, kL + F + \varDelta) \cong H^0(X', kL + F + \varDelta') \oplus H^0(X', kL + F + \varDelta' - \pi'^* A_l).
\end{equation}
The proposition then follows by induction on $l$.
\end{proof}
Notice that in the proof of the proposition, the isomorphism in formula~\ref{e-h0-x} can be made explicit by identifying $H^0(X', kL + F + \varDelta')$ with a subspace of $H^0(X, kL + F + \varDelta)$ via the map $s \mapsto \bar{\pi}^* s \otimes s_\infty$, where $s_\infty$ is a fixed section of $[X'_\infty]$ on $X$ vanishing along $X'_\infty$. Consequently, following the induction process, by fixing a section $s_{i,\infty} \in H^0(\hat{A}_i, [M_i])$ vanishing along $M_i$, we obtain an explicit isomorphism in formula~\ref{e-prop-x}. Since $X \setminus X'$ can be identified with the total space of the line bundle $\pi'_\infty: [X'_\infty] \to X'_\infty$, for each point $p \in X'_\infty$ and any choice of local frame $e$ of $[X'_\infty]$, on the fiber of $[X'_\infty]$, $s_\infty$ can be written as $a z (\pi'_\infty)^*(e)$ for some $a \neq 0$, where $z$ is a linear coordinate on the fiber $(\pi'_\infty)^{-1}(p)$.

For each component $H^0(M, kL + F - \sum_{i=1}^l c_i A_i)$ of the direct sum, choose a basis $\bcal_{(c_1, \ldots, c_l)}$. Collecting all such bases yields a basis $\bcal$ for $H^0(X, kL + F + \varDelta)$. We can impose a lexicographic (dictionary) order on the set $\Upsilon = \{ c = (c_1, \ldots, c_l) \mid c_i = 0,1 \}$ by interpreting each tuple as a word $c_1 c_2 \cdots c_l$. With this ordering, we use the index $j$, $1 \leq j \leq 2^l$, to denote the $j$-th element of $\Upsilon$.

By ordering each $\bcal_{(c_1, \ldots, c_l)}$, the full basis $\bcal$ becomes an ordered set. For $k$ sufficiently large, $\bcal$ induces a Kodaira embedding $\Phi: X \to \CP^N$, where $N = \dim H^0(X, kL + F + \varDelta) - 1$. For each tuple $\upsilon = (c_1, \ldots, c_l)$, we have a Kodaira embedding induced by $\bcal_{(c_1, \ldots, c_l)}$:
\[
\Phi_{\upsilon}: M \to \CP^{n_\upsilon},
\]
where $n_\upsilon = \# \bcal_{(c_1, \ldots, c_l)} - 1$. The inclusions $\bcal_{(c_1, \ldots, c_l)} \subset \bcal$ as ordered sets induce inclusions $I_{\upsilon}: \CP^{n_\upsilon} \hookrightarrow \CP^N$. By abuse of notation, we still denote by $\Phi_{\upsilon}$ the composition $I_{\upsilon} \circ \Phi_{\upsilon}$. We can now describe the image of $\Phi$.

For each $i$, we have a pair of divisors $M$ and $M_i$ on $\hat{A}_i$, yielding $l$ pairs of divisors $(p_i^{-1}(M), p_i^{-1}(M_i))$, $1 \leq i \leq l$. Picking one divisor from each pair gives $l$ divisors whose intersection is a copy of $M$. This can be seen by considering the fibers $\pi^q$, $q \in M$, where on each fiber the intersection is exactly one point. Therefore, for each $q \in M$, we obtain a set $\varSigma(q)$ consisting of $2^l$ points $\{ \beta_j(q) \}_{1 \leq j \leq 2^l }$. For each $i$, associate $0$ to $p_i^{-1}(M_i)$ and $1$ to $p_i^{-1}(M)$. Thus, each choice of $l$ divisors as above corresponds to a word of $l$ letters, each being $0$ or $1$. We order these choices lexicographically, and accordingly order the set $\varSigma(q)$. We may permute the indices of $\beta_j(q)$ so that $\beta_j(q)$ denotes the $j$-th element of $\varSigma(q)$ in this order.

We denote by $M_{sp} = \bigcap_{1 \leq i \leq l} p_i^{-1}(M_i)$ the distinguished (special) copy of $M$ in $X$. Let $\pi_\infty: X \to M_{sp}$ be the natural projection, so that $X$ is a $(\CP^1)^l$-bundle over $M_{sp}$. The complement $X \setminus \bigcup_{1 \leq i \leq l} p_i^{-1}(M_i)$ is the total space of a rank $l$ vector bundle $E$ over $M_{sp}$.

Fix a point $p \in M_{sp}$, and choose local frames $e_L$ for $L$, $e_F$ for $F$, and $e_i$ for $A_i$ ($1 \leq i \leq l$). We can also choose linear coordinates $(z_1, \ldots, z_l)$ on the fiber $(\pi_\infty)^{-1}(p) \cap E$ so that the identification is explicit.

A section $s \in H^0(M, kL + F - \sum_{i=1}^l c_i A_i)$ can be written locally as $f_c\, e_L^{\otimes k} \otimes e_F \otimes \bigotimes_{i,\, c_i=1} e_i^*$. When viewed as a section of $H^0(X, kL)$, its restriction to the fiber $(\pi_\infty)^{-1}(p) \cap E$ is
\[
s = f_c(p) \left( \prod_{i=1,\, c_i=1}^l a_i z_i \right) e_L^{\otimes k} \otimes e_F,
\]
where $a_i$ are the transition constants for the frames.

For each tuple $(c_1, \ldots, c_l)$, we can apply a unitary $(n_\upsilon+1) \times (n_\upsilon+1)$ matrix $\mathcal{A}_{(c_1, \ldots, c_l)} = (c_{ij})$ to the basis $\mathcal{B}_{(c_1, \ldots, c_l)} = \{ s_j \mid 1 \leq j \leq n_\upsilon+1 \}$ so that $s_1' = \sum_j c_{1j} s_j$ satisfies $s_1'(p) \neq 0$, while $s_i'(p) = 0$ for $i > 1$. After a unitary transformation of $\CP^N$, we may assume that $\mathcal{B}_{(c_1, \ldots, c_l)} = \{ s_i' \mid 1 \leq i \leq n_\upsilon+1 \}$.

By rearranging the order of the homogeneous coordinates, the restriction of $\Phi$ to $(\pi_\infty)^{-1}(p) \cap E$ can be written as
\[
\left[ \left( f_c(p) \prod_{i=1,\, c_i=1}^l a_i z_i \right)_{c},\, 0, \ldots, 0 \right],
\]
where the tuple $(f_c(p) \prod_{i=1,\, c_i=1}^l a_i z_i)_c$ is ordered according to the lexicographic order of $c$.

Define a linear map $H: \C^l \to (\pi_\infty)^{-1}(p) \cap E$ by $(w_1, \ldots, w_l) \mapsto (a_1^{-1} w_1, \ldots, a_l^{-1} w_l)$. Then the composition $\Phi \circ H$ is given by
\[
(w_1, \ldots, w_l) \mapsto \left[ \left( f_c(p) \prod_{i=1,\, c_i=1}^l w_i \right)_c,\, 0, \ldots, 0 \right].
\]
Therefore, the image $\Phi((\pi_\infty)^{-1}(p))$ is a parallelogram of type $\{ f_c(p) \}_c$ with vertices $\Phi(\varSigma(p))$.

By the inductive construction in the proof of formula~\ref{e-prop-x}, we have the following:

\begin{proposition}
Let $\upsilon = (c_1, \ldots, c_l)$ be the $j$-th element of $\Upsilon$. Then for each $q \in M$, we have
\[
\Phi_{\upsilon}(q) = \Phi(\beta_j(q)).
\]
\end{proposition}

\begin{proposition}
For each $q \in M$, let $\mathcal{N}_q$ be the net (parallelogram) spanned by $\{ \Phi(\beta_j(q)) \}_{1 \leq j \leq 2^l }$. Then the image $\Phi(X)$ is the union of all such $\mathcal{N}_q$ as $q$ varies over $M$.
\end{proposition}
\begin{proof}
For a fixed $q \in M$, we can apply unitary transformations to each $\CP^{n_\upsilon}$ so that $\Phi_{\upsilon}(q) = [1, 0, \ldots, 0] \in \CP^{n_\upsilon}$. Composing with the inclusions $I_{\upsilon}$, each $\Phi(\beta_j(q))$ has homogeneous coordinates with a single $1$ and the rest zero. By the explicit decomposition in~\ref{e-prop-x}, each fiber $\pi^{-1}(q)$ is mapped to $\mathcal{N}_q$.
\end{proof}

\subsection{Proof of Theorem~\ref{thm-almost-Bergman-Embedding}}
\begin{theorem}[Theorem~\ref{thm-almost-Bergman-Embedding}]
    Let $\{t_u\}$ be a good sequence. Then the images of $\Phi'_{t_u,k+1}$ converge, as $u \to \infty$, to a subvariety $Y$ as described in Theorem~\ref{thm-main}.
\end{theorem}

Let $I \subset \{0, \ldots, m\}$ with $X_{0,I} \neq \emptyset$ and $|I| - 1 = l \geq 1$. Fix $p \in X_{0,I}^o$ and let $(U, (z_0, \ldots, z_n))$ be an appropriate coordinate patch centered at $p$. Without loss of generality, assume $I = \{0, \ldots, l\}$. Let $R_i$ denote the radius of the $i$-th coordinate variable. Choose $M_1 > 3M$ and shrink $U$ so that $-\log R_i^2 = M_1$ for $i \leq l$. Redefine
\[
U^j(t) = \left\{ z \in U \cap X_t \;\middle|\; \sum_{i=0,\, i \neq j}^{l} \sigma_i < \frac{l}{l+1} y_t,\; \sigma_i > M_1 \text{ for } i \leq l \right\},
\]
where $\sigma_i = -\log |z_i|^2$ and $y_t = |\log |t|^2|$.

We have
\begin{align*}
    & (\sqrt{-1})^{n} \int_{U^0(t)} e^{k\phi_t} |f|^2 \left( \prod_{i=1}^{l} \sigma_i^2 \right)^{k} \frac{dz_1 \wedge d\bar{z}_1 \wedge \cdots dz_n \wedge d\bar{z}_n}{\prod_{i=1}^{l} |z_i|^2} \\
    &= (\sqrt{-1})^{n} \int \left( \int e^{k\phi_t} |f|^2 \left( \prod_{i=1}^{l} \sigma_i^2 \right)^{k} \frac{dz_1 \wedge d\bar{z}_1 \wedge \cdots dz_l \wedge d\bar{z}_l}{\prod_{i=1}^{l} |z_i|^2} \right) \bigwedge_{j>l} dz_j \wedge d\bar{z}_j.
\end{align*}

\subsubsection{Case: $D_I$ is empty}\label{subsubsec-6.1}

For any non-empty subset $J \subset I$, for $k$ sufficiently large, the restriction map
\[
\hcal_{J,k+1} \to H^0(D_J, (k+1)L - \sum_{j \notin J} [X_{0,j}])
\]
given by $s \mapsto \frac{s}{S_{\hat{J}}}$ is surjective. Thus, we may apply a unitary transformation $\mathcal{A}$ to the basis of $\hcal_{J,k+1,1}$ (corresponding to $\bcal_{J,t,1}$) so that, in the new basis, all but one section vanish at $p$ when restricted to $D_J$, i.e., $\frac{s}{S_{\hat{J}}}(p) = 0$. Denote the unique section not vanishing at $p$ by $s^{J,p}$. Let $f^{J,p}$ be the local holomorphic function on $U$ representing $\tilde{s}^{J,p}$.

\textbf{Terminology:} We refer to $s^{J,p}$ as the \emph{chosen peak section at $p$ of degree $J$}.

Then as before, we can expand
\[
f^{J,p} = \sum_{\beta} b_{\beta}^{J,p} z^\beta,
\]
where the sum is over all multi-indices $\beta \geq 0$ with $\beta_j \geq 1$ for $j \notin J$.
The condition on $s^{J,p}$ ensures that
\[
b^{J,p}_{\beta(J)} \neq 0,
\]
where $\beta(J)$ is the multi-index with $\beta_j = 0$ for $j \in J$ and $\beta_j = 1$ for $j \in I \setminus J$.

On $U \cap X_t$, using $(z_1, \ldots, z_n)$ as holomorphic coordinates, we have
\[
f^{J,p} = \sum_{\alpha} a^{J,p}_{\alpha}(z_{l+1}, \ldots, z_n) \prod_{i=1}^{l} z_i^{\alpha_i},
\]
where
\[
a^{J,p}_\alpha = \sum_{\beta \in A_{\alpha}} b^{J,p}_\beta t^{\beta_0}.
\]

Recall that
\[
\sum_{\substack{\beta \in A_\alpha \\ \beta_0 > 0}} \int_{|z_j| < R_j} |b^{J,p}_\beta|^2 c_\beta \nu_l < \frac{C_1}{k^{l+1} \epsilon_1 (2\pi)^{l+1}},
\]
where $C_1$ depends only on the $L^2$-norm of $\frac{s^{J,p}}{S_{\hat{J}}}$ on $X_{0,J}^o$, and
\[
\left| \sum_{\substack{\beta \in A_\alpha \\ \beta_0 > 0}} b^{J,p}_\beta t^{\beta_0} \right|^2 < C_2 |t| \sum_{\substack{\beta \in A_\alpha \\ \beta_0 > 0}} |b^{J,p}_\beta|^2 c_\beta,
\]
for some constant $C_2 > 0$ independent of $\alpha$ and $U$, provided $|t|$ is sufficiently small.

Since $|b^{J,p}_\beta(p)|^2 (2\pi)^{n-l} \prod_{j=l+1}^{n} R_j^2 \leq \int_{|z_i| < R_i} |b^{J,p}_\beta|^2 \vcal_l$, we obtain
\[
\left| \sum_{\substack{\beta \in A_\alpha \\ \beta_0 > 0}} b^{J,p}_\beta(p) t^{\beta_0} \right|^2 < \frac{C_1 C_2 |t|}{k^{l+1} \epsilon_1 (2\pi)^{n+1} \prod_{j=l+1}^{n} R_j^2}.
\]
Therefore,
\[
|a^{J,p}_\alpha(0) - b^{J,p}_{\overline{0\alpha}}(0)| < C_{12} |t|,
\]
for some constant $C_{12} > 0$ independent of $\alpha$ and $t$.

Let $\bcal_{J,t,1}' = \mathcal{A} \bcal_{J,t,1}$ denote the transformed basis. For each $i \in I$, we obtain a set of $2^l$ sections:
\[
\dcal_{t,p,i} = \left\{ s^{J,p}_{t,\mathrm{rn}} \mid i \in J \subset I \right\}.
\]
We denote $z_{\hat{J}} = \prod_{j \in I \setminus J} z_j$. Then, on $U^0(t)$, letting $\epsilon_3 = \min_{0 \in J \subset I} |b^{J,p}_{\beta(J)}(0)|^2$, we have
\begin{equation}\label{e-dpt}
    \sum_{0 \in J \subset I} y_t^{-(|J|-1)(2k+1)} |b^{J,p}_{\beta(J)}(0)|^2 |z_{\hat{J}}|^2 \geq \epsilon_3 \prod_{j=1}^{l} \left( y_t^{-(2k+1)} + |z_j|^2 \right).
\end{equation}
For later use, we set
\[
d_{p,t} = \sum_{0 \in J \subset I} y_t^{-(|J|-1)(2k+1)} |b^{J,p}_{\beta(J)}(0)|^2 |z_{\hat{J}}|^2.
\]

Recall that in the construction of $s_{t,\mathrm{rn}}$ for $X_{0,J}$ with $|J| = l'+1$, we needed to solve the $\bar{\partial}$-equation
\[
\bar{\partial} v = \bar{\partial} \varpi(\tau) \otimes s_t,
\]
and obtain a minimal solution $v_t$ satisfying
\[
\int_{X_t} \| v_t \|^2_{\mathrm{KE}} \, \omega_t^n < c_5 y_t^{(l'-1)(2k+1)+2},
\]
for some constant $c_5$ independent of $t$. Since $\bar{\partial} \varpi(\tau) = 0$ on $U$, $v_t$ is holomorphic on $U$.

Let $f^{v_t}$ denote the local holomorphic function on $U$ representing $y_t^{-\frac{l'}{2}(2k+1)} v_t$. We can expand
\[
f^{v_t} = \sum_{\beta} b_{\beta}^{v_t} z^\beta,
\]
and on $X_t \cap U$ with coordinates $(z_1, \ldots, z_n)$,
\[
f^{v_t} = \sum_{\alpha} a^{v_t}_{\alpha}(z_{l'+1}, \ldots, z_n) \prod_{i=1}^{l} z_i^{\alpha_i}.
\]
Since
\[
\int_{X_t} \| y_t^{-\frac{l'}{2}(2k+1)} v_t \|^2_{\mathrm{KE}} \, \omega_t^n < c_5 y_t^{-2k+1},
\]
we have
\[
\sum_{\alpha \geq 0} \left[ \int_{|z_i| < R_i} |a^{v_t}_\alpha|^2 \bigwedge_{j \geq l+1} (dz_j \wedge d\bar{z}_j) \right] \left[ \prod_{j=1}^{l} \int_{M_1}^{\frac{1}{l+1} y_t} |z_j|^{2\alpha_j} \sigma_j^{2k} d\sigma_j \right] < C_{13} y_t^{-2k+1}.
\]

Let $H_\lambda = \int_{M_1}^{\frac{1}{l+1} y_t} e^{-\lambda x} x^{2k} dx$ and $H_\alpha = \prod_{j=1}^{l} H_{\alpha_j}$. Then
\[
\sum |a^{v_t}_\alpha(0)|^2 H_\alpha \leq \frac{C_{13}}{(2\pi)^l \prod_{j=l+1}^{n} R_j^2} y_t^{-2k+1} = c y_t^{-2k+1}.
\]

Therefore,
\begin{align*}
    \left| \sum_{\alpha \geq 0} a^{v_t}_\alpha(0) z^\alpha \right|^2
    &\leq \left( \sum |a^{v_t}_\alpha(0)|^2 H_\alpha \right) \left( \sum H_\alpha^{-1} |z^\alpha|^2 \right) \\
    &\leq c' y_t^{-2k+1} \prod_{j=1}^{l} \left( \sum_{\lambda=0}^\infty H_\lambda^{-1} |z_j|^{2\lambda} \right).
\end{align*}

We have $H_0 = \frac{(\frac{1}{l+1} y_t)^{2k+1} - M_1^{2k+1}}{2k+1}$, and for $\lambda \geq 1$,
\[
H_\lambda = \lambda^{-2k-1} \int_{\lambda M_1}^{\lambda \frac{1}{l+1} y_t} e^{-x} x^{2k} dx.
\]

Let $g(\lambda) = \log \int_{\lambda M_1}^{\lambda \frac{1}{l+1} y_t} e^{-x} x^{2k} dx$. Then
\[
g'(\lambda) = \frac{ -M_1 e^{-\lambda M_1} (\lambda M_1)^{2k} + \frac{1}{l+1} y_t e^{-\lambda \frac{1}{l+1} y_t} (\lambda \frac{1}{l+1} y_t)^{2k} }{ \int_{\lambda M_1}^{\lambda \frac{1}{l+1} y_t} e^{-x} x^{2k} dx }.
\]
For $|t|$ sufficiently small, $g'(\lambda) < 0$, and
\begin{align*}
    |g'(\lambda)| &< M_1 \frac{ e^{-\lambda M_1} (\lambda M_1)^{2k} }{ \int_{\lambda M_1}^{\lambda \frac{1}{l+1} y_t} e^{-x} x^{2k} dx } \\
    &< M_1 \frac{ e^{-\lambda M_1} (\lambda M_1)^{2k} }{ e^{-(\lambda M_1 + 1)} (\lambda M_1 + 1)^{2k} } \\
    &< e M_1.
\end{align*}

Therefore, we have the following proposition.

\begin{proposition}\label{prop-delta-2}
    There exists $\delta_2 > 0$ such that, when $|z_j| < \delta_2$,
    \[
    \sum_{\lambda=1}^{\infty} |z_j|^{2\lambda} H_\lambda^{-1} < 2 |z_j|^2 H_1^{-1}.
    \]
\end{proposition}

We denote by $\mathfrak{N}_p$ the vertical polydisc at $p$, i.e.,
\[
\mathfrak{N}_p = \{ q \in U \mid z_j(q) = z_j(p),\ j > l \},
\]
and define
\[
W_t = \left\{ q \in \mathfrak{N}_p \cap X_t \mid |z_j(q)| < y_t^{-k+3/2},\ 0 \leq j \leq l \right\}.
\]
On $U^0(t) \cap W_t$, we have
\[
\left| \sum_{\alpha \geq 0} a^{v_t}_\alpha(0) z^\alpha \right|^2 \leq c' y_t^{-2k+1} \prod_{j=1}^l \left( H_0^{-1} + 2 H_1^{-1} |z_j|^2 \right).
\]
Since for $|t|$ sufficiently small, $H_1$ is close to $\int_{M_1}^{\infty} e^{-x} x^{2k} dx$, there exists $\epsilon_4 > 0$ such that
\[
\prod_{j=1}^l \left( H_0^{-1} + 2 H_1^{-1} |z_j|^2 \right) < \epsilon_4 \prod_{j=1}^l \left( y_t^{-(2k+1)} + |z_j|^2 \right).
\]
Therefore,
\begin{equation}\label{e-a-vt-u0}
    \left| \sum_{\alpha \geq 0} a^{v_t}_\alpha(0) z^\alpha \right|^2 \leq c' y_t^{-2k+1} \frac{\epsilon_4}{\epsilon_3} d_{p,t},
\end{equation}
on $U^0(t) \cap W_t$.

Next, we consider terms with some $\alpha_i < 0$. The idea is similar to Case 3 in subsection~\ref{subsec-case-0}. Let
\[
\Gamma_i = \left\{ \alpha \geq 0 \mid \alpha_i > 0,\ \alpha_j \leq \alpha_i\ \text{for}\ j \leq l \right\}.
\]
On $U_i(t)$, using coordinates $(z_0, \ldots, z_{i-1}, z_{i+1}, \ldots, z_n)$, the terms in the expansion with negative degrees arise from $a_\alpha z^\alpha$ for $\alpha \in \Gamma_i$.

On $U_i(t)$ with $\alpha \in \Gamma_i$, we have
\[
H_\alpha^{-1} |z^\alpha|^2 < H_{\alpha_i}^{-1} e^{-\frac{\alpha_i}{l+1} y_t} \prod_{\substack{j=1 \\ j \neq i}}^l H_{\alpha_j}^{-1} |z_j^{\alpha_j}|^2.
\]
Therefore,
\begin{align*}
    \left| \sum_{\alpha \in \Gamma_i} a^{v_t}_\alpha(0) z^\alpha \right|^2
    &\leq c' y_t^{-2k+1} \left( \sum_{\alpha \in \Gamma_i} H_\alpha^{-1} |z^\alpha|^2 \right) \\
    &\leq c' y_t^{-2k+1} e^{-\frac{1}{l+2} y_t} \prod_{\substack{j=1 \\ j \neq i}}^l \left( \sum_{\lambda=0}^\infty H_\lambda^{-1} |z_j^\lambda|^2 \right) \\
    &\leq c' y_t^{-2k+1} e^{-\frac{1}{l+3} y_t} \prod_{j=0,\, j \neq i}^l \left( \sum_{\lambda=0}^\infty H_\lambda^{-1} |z_j^\lambda|^2 \right),
\end{align*}
for $|t|$ sufficiently small. Thus, by symmetry between $U^i(t)$ and $U^0(t)$, we obtain
\begin{equation*}
    \left| \sum_{\alpha,\ \alpha_j < 0\ \text{for some}\ j} a^{v_t}_\alpha(0) z^\alpha \right|^2 \leq c_6 e^{-\frac{1}{l+3} y_t} d_{p,t},
\end{equation*}
for some $c_6 > 0$.

Combining the above, we have
\begin{equation}\label{e-a-vt-0}
    \left| \sum_{\alpha} a^{v_t}_\alpha(0) z^\alpha \right|^2 \leq c_7 y_t^{-2k+1} d_{p,t},
\end{equation}
on $W_t \cap U^0(t)$ for some $c_7 > 0$ independent of $t$. In other words, on $W_t \cap U^0(t)$, the contribution from $v_t$ can be neglected.

Up to this point, we have not fully utilized the definition of $W_t$; specifically, we have only required $|z_j| < \delta_2$ for $j \in J$. Now, by combining the previous arguments with the additional condition $|z_j| < y_t^{-k+3/2}$ for $j \in J$, we see that, for the sum $\sum_{\beta \neq \beta(J)} b_{\beta}^{J,p}(0) z^\beta$, its contribution can be neglected on $W_t \cap U^0(t)$. Thus, to understand the image of $W_t \cap U^0(t)$ under $\Phi_{t,k+1}$, for the sections $s^{J,p}_{t,\mathrm{rn}}$, it suffices to consider only the leading term, i.e., we may represent $s^{J,p}_{t,\mathrm{rn}}$ by
\[
y_t^{-\frac{|J|-1}{2}(2k+1)} b_{\beta(J)}^{J,p}(0) z^{\beta(J)}.
\]
For the other sections $s_{t,\mathrm{rn}} \in \bcal_{J,t,1}'$, a similar argument shows that only the term $b_{\beta(J)}(0) z^{\beta(J)}$ matters, but since $b_{\beta(J)}(0) = 0$ for these sections, their contribution can also be ignored on $W_t \cap U^0(t)$. The same reasoning applies to all sections in $\bcal_{J,t,2}$ for $J \subset I$, and to all sections in $\bcal_{J,t}$ for $J \nsubseteq I$, since their $L^2$-norms on $U$ are $O(y_t^{-2k+1})$ and thus negligible.

In summary, to describe the image of $W_t \cap U^0(t)$ under $\Phi_{t,k+1}$, it suffices to consider the map to $\CP^{2^l-1}$ defined by
\[
\left\{ y_t^{-\frac{|J|-1}{2}(2k+1)} b_{\beta(J)}^{J,p}(0) z^{\beta(J)} \right\}_{0 \in J \subset I},
\]
or, equivalently,
\[
\left\{ y_t^{\frac{|I|-|J|}{2}(2k+1)} b_{\beta(J)}^{J,p}(0) z^{\beta(J)} \right\}_{0 \in J \subset I}.
\]
We now define $W_{t,0}^c \subset W_t \cap U^0(t)$ as the subset where
\[
(2k-3)\log y_t < \sigma_i < \frac{1}{l+1} y_t, \quad 1 \leq i \leq l,
\]
and similarly define $W_{t,i}^c$ for other $i$.

\begin{figure}[h]
    \caption{An illustration of the position of $W_{t,i}^c$ in $W_t$}
    \centering
    \begin{tikzpicture}
        \draw[connect] (0,0)--(9,0) (9,0)--(4.5,7.794) (4.5,7.794)--(0,0);
        \draw[connect] (3,0)--(6,5.196) (6,0)--(3,5.196) (1.5,2.598 )--(7.5,2.598);
        \node[label] at (4.5,5.196) {$W_{t,0}^c$};
        \node[label] at (2.25,1.299) {$W_{t,1}^c$};
        \node[label] at (6.75,1.299) {$W_{t,2}^c$};
    \end{tikzpicture}
\end{figure}
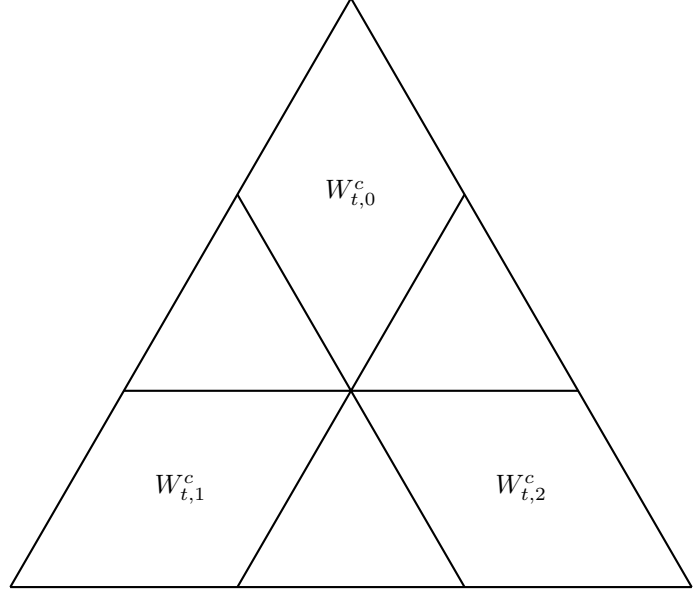
We can perform a change of variables on $W_{t,0}^c$ by setting $w_i = y_t^{\frac{1}{2}(2k+1)} z_i$ for $1 \leq i \leq l$. Then each $w_i$ ranges over
\[
A_t = \left\{ x \in \C \;\middle|\; y_t^{\frac{1}{2}(2k+1)} e^{-\frac{1}{2(l+1)} y_t} < |x| < y_t^2 \right\}.
\]
This yields a map $\phi_{p,t}: A_t^l \to \CP^{2^l-1}$ defined by
\[
(w_1, \ldots, w_l) \mapsto \left[ b_{\beta(J)}^{J,p}(0) w^{\beta(J)} \right]_{0 \in J \subset I},
\]
where the homogeneous coordinates are ordered according to the reverse order of the $J$'s as specified before Theorem~\ref{thm-almost-orthon}. It is clear that as $|t| \to 0$, the map $\phi_{p,t}$ converges to a Segre embedding $\Psi_p$ of type $\{ b_{\beta(J)}^{J,p}(0) \}_{0 \in J \subset I}$. The vertices of the image of $\phi_{p,t}$ are the points $\gamma_i$. Thus, we have the following proposition.

\begin{proposition}\label{prop-wc-0}
    As $u \to \infty$, the images $\Phi_{t_u,k+1}(W_{t_u,0}^c)$ converge to the parallelogram of type $\{ b_{\beta(J)}^{J,p}(0) \}_{0 \in J \subset I}$ with vertices $\{ \Phi_{J,k+1}(p) \}_{0 \in J \subset I}$.
\end{proposition}

Moreover, if we change the upper bound in the definition of $W_t$ to $c y_t^{-k+3/2}$ for any fixed $c > 0$, the conclusion remains unchanged. We can therefore state the following proposition.

\begin{proposition}
    As $u \to \infty$, the images $\Phi_{t_u,k+1}(W_t \cap U^0(t))$ converge to the parallelogram of type $\{ b_{\beta(J)}^{J,p}(0) \}_{0 \in J \subset I}$ with vertices $\{ \Phi_{J,k+1}(p) \}_{0 \in J \subset I}$.
\end{proposition}

\begin{proof}
    When $l = 1$, we have $W_{t,0}^c = W_t \cap U^0(t)$, so the result follows directly from Proposition~\ref{prop-wc-0}. For $l > 1$, we proceed by induction. Observe that
    \[
    (W_t \cap U^0(t)) \setminus W_{t,0}^c = \bigcup_{1 \leq i \leq l} G_i,
    \]
    where $G_i = \{ q \in W_t \cap U^0(t) \mid \sigma_i \geq \frac{1}{l+1} y_t \}$. On $G_i$, we have $|z_i|^2 < e^{-\frac{1}{l+1} y_t}$, so the terms $b_{\beta(J)}(0) z^{\beta(J)}$ for $i \notin J$ can be neglected. For example, if $i = 1$, the map $G_1 \to \CP^{2^{l-1}-1}$ is given by $z \mapsto [b_{\beta(J)}(0) z^{\beta(J)}]_{\{0,1\} \subset J \subset I}$. By the induction hypothesis, the images $\Phi_{t_u,k+1}(G_1)$ converge to the parallelogram of type $\{ b_{\beta(J)}^{J,p}(0) \}_{\{0,1\} \subset J \subset I}$ with vertices $\{ \Phi_{J,k+1}(p) \}_{\{0,1\} \subset J \subset I}$, which is a sub-parallelogram of the limit of $\Phi_{t_u,k+1}(W_{t_u,0}^c)$. The same argument applies for each $i$. This completes the proof.
    \end{proof}
Similar arguments apply to $W_t \cap U^0(t)$ with $i \neq 0$, where $J$ ranges over all $J$ satisfying $i \in J \subset I$. Thus, we have the following proposition.

\begin{proposition}\label{prop-empty-wt}
    As $u \to \infty$, the images $\Phi_{t_u,k+1}(W_{t_u})$ converge to the union of $l+1$ parallelograms $\bigcup_{0 \leq i \leq l} \mathcal{P}_i$, where each $\mathcal{P}_i$ is the parallelogram of type $\{b_{\beta(J)}^{J,p}(0)\}_{i \in J \subset I}$ with vertices $\{\Phi_{J,k+1}(p)\}_{i \in J \subset I}$. For $i \neq j$, the intersection $\mathcal{P}_i \cap \mathcal{P}_j$ is a codimension $1$ sub-parallelogram of $\mathcal{P}_i$.
\end{proposition}

\begin{figure}[h]
    \caption{An illustration of the limit of the image $\Phi_{t_u,k+1}(W_t)$}
    \centering
    \begin{tikzpicture}
        \draw[connect] (0,0)--(4,0) (4,0)--(6,3.464) (6,3.464)--(8,0) (6,3.464)--(2,3.464);
        \draw[connect] (4,0)--(6,-3.464) (8,0)--(6,-3.464) (6,-3.464)--(2,-3.464) (2,-3.464)--(0,0) (2,3.464)--(0,0);
        \node[label] at (3,1.732) {$\mathcal{P}_0$};
        \node[label] at (3,-1.732) {$\mathcal{P}_1$};
        \node[label] at (6,0) {$\mathcal{P}_2$};
    \end{tikzpicture}  
\end{figure}
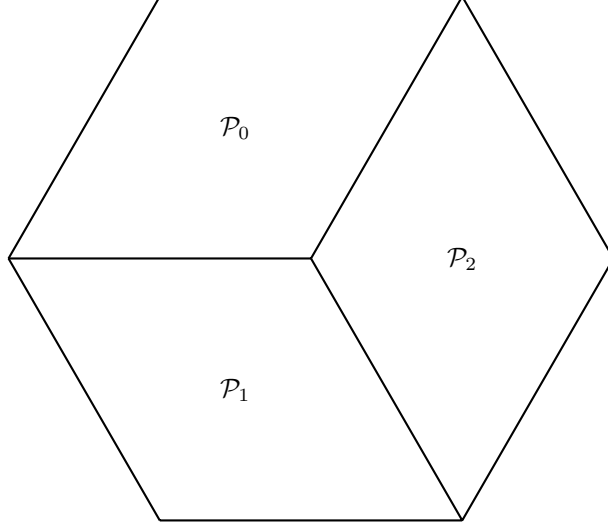

For each $i \in I$, let $A_i$ denote the normal bundle of $X_{0,I}$ in $X_{0,I \setminus \{i\}}$, which is simply the restriction of $[X_{0,i}]$ to $X_{0,I}$. Then, for each $i \in I$, by the construction in subsection~\ref{subsec-pro-bundle} and using the notations there, we have the $(\CP^1)^l$-bundle
\[
\pi_{I,i} : \mathcal{Z}_{I,i} \coloneqq \Theta_{\{A_0, \ldots, A_{i-1}, A_{i+1}, \ldots, A_l\}} \to X_{0,I}.
\]

For each proper subset $J \subset I$, the embedding $\Phi_{J,k+1}$ induced by $\{B_{J,0,1}\}$ embeds $X_{0,I}$ into $\CP^{n_J-1}$. Let $V_{J,I}$ denote the minimal linear subspace containing $\Phi_{J,k+1}(X_{0,I})$. Let $[z_1, \ldots, z_{n_J}]$ be the homogeneous coordinates of $\CP^{n_J-1}$. After a unitary transformation on $\CP^{n_J-1}$, we may assume that $V_{J,I}$ is defined by $z_j = 0$ for $j > \gamma$ for some $\gamma$. Then $\{z_j\}_{j \leq \gamma}$ pulls back to a basis $\mathcal{E}_{J,I} = \{s_j\}_{j \leq \gamma}$ of $H^0(X_{0,I}, (k+1)L - \sum_{i \notin J} [X_{0,i}])$. For each $p \in X_{0,I}$, we can apply a unitary $\gamma \times \gamma$ matrix to $\{s_j\}_{j < \gamma}$ so that the new basis $\{s'_j\}_{j < \gamma}$ satisfies $s'_i(p) = 0$ for $i > 1$.

By multiplying $s'_1$ by a scalar, we can arrange $s'_1(p) = \frac{s^{J,p}}{S_{\hat{J}}}(p)$ for all $p \in X_{0,I}$. By choosing local frames $e_i$ for $[X_{0,i}]$ for $0 \leq i \leq m$ and $e_L$ for $L$, we can write
\[
s'_0 = g_J e_L^{\otimes(k+1)} \otimes \bigotimes_{0 \leq i \leq m} e_i^*, \quad S_i = h_i e_i, \; i > l, \quad S_i = h_i z_i e_i, \; i \leq l,
\]
where $h_i(p) \neq 0$ for all $i$. Then
\[
b^{J,p}_{\beta(J)}(p) = g_J(p) \prod_{i \notin J} h_i(p) = c_p g_J(p) \left( \prod_{i \in I \setminus J} h_i(p) \right),
\]
where $c_p = \prod_{i \notin I} h_i(p)$.

By the decomposition in Proposition~\ref{prop-x-m}, the bases $\mathcal{E}_{J,I}$ for $0 \in J \subset I$ together define an embedding $\Psi : \mathcal{Z}_{I,0} \to \CP^{m_I}$ for some $m_I$. The image of the fiber $\pi_{I,0}^{-1}(p)$ is a parallelogram of type $\{g_J(p)\}_{0 \in J \subset I}$ with vertices $\{\Psi_J(p)\}_{0 \in J \subset I}$, where $\Psi_J$ is the embedding defined by $\mathcal{E}_{J,I}$. As seen before, by a change of variables on $(\CP^1)^l$, the image of the fiber $\pi_{I,0}^{-1}(p)$ is also a parallelogram of type $\{c_p g_J(p) \prod_{i \in I \setminus J} h_i(p)\}_{0 \in J \subset I}$ with vertices $\{\Psi_J(p)\}_{0 \in J \subset I}$.

Let $\wcal_{I,t} = \{ q \in X_t \mid d_I(q) < y_t^{-k+3/2} \}$. Then we have the following theorem.

\begin{theorem}
    As $u \to \infty$, the image $\Phi_{t_u,k+1}(\wcal_{I,t_u})$ converges to the union of $l+1$ varieties $\bigcup_{0 \leq i \leq l} Y_{I,i}$, where each $Y_{I,i}$ is the image of the projective bundle $\zcal_{I,i}$. Moreover, for $i \neq j$, the intersection $Y_{I,i} \cap Y_{I,j}$ is the image of the projective bundle $\Theta_{\{A_v \mid 0 \leq v \leq l,\, v \neq i, j\}}$.
\end{theorem}

\

\subsubsection{The case $D_I \neq \emptyset$}

Let $I \subset \{0, \ldots, m\}$ with $|I| \geq 2$, and for any $i \in I$, define $A_i$ to be the normal bundle of $X_{0,I}$ in $X_{0,I \setminus \{i\}}$, and set
\[
\zcal_{I,i} \coloneqq \Theta_{\{A_v \mid v \in I,\, v \neq i\}} \to X_{0,I}.
\]
For any $\delta > 0$ and any point $p \in X_{0,I}^\delta$, the same arguments as in the previous case apply, yielding the same conclusion as in Proposition~\ref{prop-empty-wt}. If we denote by $\wcal_{I,t}^{\delta}$ the set of points in $X_t$ whose distance to $X_{0,I}^\delta$ is less than $y_t^{-k+3/2}$, then we have:

\begin{proposition}
    As $u \to \infty$, the image $\Phi_{t_u,k+1}(\wcal_{I,t_u}^\delta)$ converges to the union of $|I|$ varieties $\bigcup_{i \in I} Y_{I,i}$, where each $Y_{I,i}$ is the image of the projective bundle $\zcal_{I,i}$.
\end{proposition}

\

\begin{definition}
    Let $I \subset \{0, \ldots, m\}$ with $|I| \geq 2$. Let $\sigma$ be a cell of maximal dimension such that the $I$-cell is on its boundary. We define the \emph{length} of $I$ to be $\dim(\sigma) - |I| + 1$.
\end{definition}

Suppose $U \in \ucal^I$ is centered at some $p \in X_{0,I}^o$. Let $J$ be a proper subset of $I$. Recall that in the expansion $f^{J,p} = \sum_\beta b^{J,p}_\beta z^\beta$, we have $\beta \geq \beta(J)$. As in subsection~\ref{subsubsec-6.1}, we have
\[
\sum_{\beta} \int_{|z_j| < R_j} |b^{J,p}_\beta|^2 c_\beta \nu_l < \frac{C_1}{k^{l+1} \epsilon_1 (2\pi)^{l+1}},
\]
and $|b^{J,p}_\beta(p)|^2 (2\pi)^{n-l} \prod_{j=l+1}^n R_j^2 \leq \int_{|z_i| < R_i} |b^{J,p}_\beta|^2 \vcal_l$. Therefore,
\[
\sum_{\beta > \beta(J)} |b^{J,p}_\beta(p)|^2 (2\pi)^{n-l} \prod_{j=l+1}^n R_j^2 c_\beta \leq \frac{C_1}{k^{l+1} \epsilon_1 (2\pi)^{l+1}}.
\]
Thus, for any $\epsilon > 0$, there exists $\epsilon' > 0$ such that on $\mathfrak{N}_p$, if $|z_j| < \epsilon'$ for every $j \in I$, then
\[
\sum_{\beta > \beta(J)} |b^{J,p}_\beta(p)|^2 |z^{\beta - \beta(J)}|^2 < \epsilon.
\]
This implies that for any $\epsilon > 0$, there exists $\epsilon' > 0$ such that if $|z_j| < \epsilon'$ for every $j \in I$, then
\[
\left| \frac{|f^{J,p}|^2}{|b^{J,p}_{\beta(J)}(p)|^2 |z^{\beta(J)}|^2} - 1 \right| < \epsilon.
\]

The same arguments show that for some smaller $\epsilon'$, on $\mathfrak{N}_p$, for any section $s$ in $\bcal_{J,k+1,2} \cup (\bcal_{J,k+1,1}' \setminus \{ s^{J,p}_{t,\mathrm{rn}} \})$, we have
\[
\frac{|f_s|^2}{|b^{J,p}_{\beta(J)}(p)|^2 |z^{\beta(J)}|^2} < \epsilon^2,
\]
where $f_s$ is the holomorphic function representing $s$.

The same arguments apply to the sections $s^{J,p'}_{t,\mathrm{rn}}$ for any point $p' \in X_{0,I} \cap U$. By induction, we obtain the following lemma.

\begin{lemma}\label{lem-epsilon-delta}
    For any $\epsilon > 0$, there exist $\delta > 0$ and $u_0$ such that for all $u > u_0$, for each $I$, each $U \in \ucal^I$ with coordinates $(z_0, \ldots, z_n)$, and each $p \in U \cap X_{0,I}$, we have
    \[
    \left| \frac{ \sum_{J \subset I} y_{t_u}^{-(|J|-1)(2k+1)} |b^{J,p}_{\beta(J)}|^2 |z^{\beta(J)}|^2 }{ \sum_{s \in \bigcup_I \bcal_{t,I}} |f_s|^2 } - 1 \right| < \epsilon
    \]
    on $\mathfrak{N}_p \cap (X_{t_u} \setminus X_{t_u}^\delta)$, where $f_s$ is the local holomorphic function representing $s$.
\end{lemma}
\begin{theorem}\label{thm-y-j}
    For $|I| \geq 2$, the image $\Phi_{t_u,k+1}(\wcal_{I,t_u})$, as $u \to \infty$, converges to the union of a collection of varieties
    \[
    \bigcup_{J,\, I \subset J} Y_J,
    \]
    where each $Y_J$ is the union of $|J|$ varieties $\bigcup_{i \in J} Y_{J,i}$, and each $Y_{J,i}$ is the image of the projective bundle $\zcal_{J,i}$.
\end{theorem}
\begin{proof}
    We proceed by induction on the length of $I$. The base case, where the length of $I$ is $0$ (i.e., $D_I$ is empty), has already been established. Assume now that the length of $I$ is at least $1$. By the induction hypothesis, for each $J$ with $I \subsetneq J$ and $X_{0,J} \neq \emptyset$, the image $\Phi_{t_u,k+1}(\wcal_{J,t_u})$ converges, as $u \to \infty$, to a subvariety as described in the theorem.

    Let
    \[
    \wcal_{I,D,t} = \left\{ q \in X_t \mid d(q, D_I) < y_t^{-k+3/2} \right\}.
    \]
    Since $\wcal_{I,D,t}$ is the union of the sets $\wcal_{J,t}$ for all $J$ with $I \subsetneq J$ and $|J| - |I| = 1$, it follows that the image $\Phi_{t_u,k+1}(\wcal_{I,D,t_u})$ converges, as $u \to \infty$, to the union
    \[
    \bigcup_{J,\, I \subsetneq J} Y_J,
    \]
    where each $Y_J$ is the union of $|J|$ varieties $\bigcup_{i \in J} Y_{J,i}$, with each $Y_{J,i}$ the image of the projective bundle $\zcal_{J,i}$.

    Now, fix any $\epsilon > 0$. There exists $\delta > 0$ (as in Lemma~\ref{lem-epsilon-delta}) such that $\delta < \delta_2$ and $-\log \delta > M$, where $\delta_2$ is the constant from Proposition~\ref{prop-delta-2}. For any $q \in X_{0,I}$ with $d(q, D_I) < \delta$, we have $q \in U$ for some $U \in \ucal$ centered at a point $p \in X_{0,J}^o$ for some $I \subsetneq J$. Without loss of generality, assume $I = \{0, \ldots, l\}$, $J = \{0, \ldots, l'\}$, and $\prod_{i=0}^{l'} z_i = t$ in local coordinates.

    Define
    \[
    \mathfrak{J}_{p,\delta,t} = \left\{ z \in \mathfrak{N}_p \mid \delta > d(z, D_I) \geq y_t^{-k+3/2} \right\}.
    \]
    For $q_0 \in X_{0,I} \cap \mathfrak{J}_{p,\delta,t}$, consider the set
    \[
    W_{q_0,t} = \left\{ q \in X_t \mid |z_i(q)| < y_t^{-k+3/2} \text{ for } i \leq l,\; z_j(q) = z_j(q_0) \text{ for } j > l \right\}.
    \]
    For each $I' \subset J$, as before, we have sections $s_{t,\mathrm{rn}}^{I',p}$ and corresponding terms \[y_t^{-\frac{|I'|-1}{2}(2k+1)} b_{\beta(I')}^{I',p} z^{\beta(I')},\] where $\beta(I')$ is the multi-index with $\beta_j = 0$ for $j \in I'$ and $\beta_j = 1$ for $j \in J \setminus I'$.

    The same arguments as in subsection~\ref{subsubsec-6.1} show that, to understand the image $\Phi_{t,k+1}(W_{q_0,t})$, we may ignore the $v_t$-part in $s_{t,\mathrm{rn}}$ for all $I'$, as well as the sections $s_{t,\mathrm{rn}}$ for those $I''$ with $I'' \nsubseteq J$. Moreover, $d(q, D_I) \geq 2 y_t^{-k+3/2}$ implies that for some $c > 0$, $|z_j| > c y_t^{-k+3/2}$ for $j \in J \setminus I$, so we can also ignore the sections $s_{t,\mathrm{rn}}^{I',p}$ with $I' \nsubseteq I$.

    Lemma~\ref{lem-epsilon-delta} further implies that, if we replace the local functions $y_t^{-\frac{|I'|-1}{2}(2k+1)} f^{I',p}$ by $y_t^{-\frac{|I'|-1}{2}(2k+1)} b_{\beta(I')}^{I',p} z^{\beta(I')}$ for $I' \subset I$ and ignore the other sections in $\bcal_{I',k+1,2} \cup (\bcal_{I',k+1,1}' \setminus \{ s^{I',p}_{t,\mathrm{rn}} \})$, then the new local map $\Phi'_t$ sends any $q \in W_{q_0,t}$ to a point whose distance to $\Phi_{t,k+1}(q)$ is less than $k^n \epsilon$.

    Now, it suffices to consider the map $\Phi''_t : W_{q_0,t} \to \CP^{2^{l+1}-2}$ defined by
    \[
    \left\{ y_t^{-\frac{|I'|-1}{2}(2k+1)} b_{\beta(I')}^{I',p} z^{\beta(I')} \right\}_{\emptyset \neq I' \subset I},
    \]
    which is equivalent to the map
    \[
    \left\{ y_t^{\frac{|I|-|I'|}{2}(2k+1)} b_{\beta(I')}^{I',p} z^{\beta(I') - \beta(I)} \right\}_{\emptyset \neq I' \subset I}.
    \]
    As before, for each $i \in I$, define
    \[
    W^i_{q_0,t} = \left\{ q \in W_{q_0,t} \mid \sigma_i > \frac{1}{l+1} \left( y_t - \sum_{j=l+1}^{l'} \sigma_j \right) \right\}.
    \]
    Since $\sigma_j < (k - \frac{3}{2}) \log y_t - \log c$ for each $j > l$, the extra term $\sum_{j=l+1}^{l'} \sigma_j$ does not affect the argument. Set $y_t' = \frac{1}{l+1} \left( y_t - \sum_{j=l+1}^{l'} \sigma_j \right)$. Then, on
    \[
    W^{0,c}_{q_0,t} = \left\{ q \in W^0_{q_0,t} \mid \sigma_j < y_t',\; 1 \leq j \leq l \right\},
    \]
    we can ignore the terms $y_t^{-\frac{|I'|-1}{2}(2k+1)} b_{\beta(I')}^{I',p} z^{\beta(I')}$ for $0 \notin I'$. By the substitution $w_i = y_t^{\frac{1}{2}(2k+1)} z_i$, we obtain a holomorphic map $(A_t)^l \to \CP^{2^l-1}$ defined by
    \[
    \left\{ b_{\beta(I')}^{I',p} w^{\beta(I') - \beta(I)} \right\}_{0 \in I' \subset I},
    \]
    where $A_t = \left\{ w \in \C \mid y_t^{\frac{1}{2}(2k+1)} e^{-y_t'} < |w| < y_t^2 \right\}$.

    It is then clear that the image $\Phi'_t(W^{0,c}_{q_0,t})$, as $u \to \infty$, converges to the parallelogram $\pcal_{I,p,0}$ of type $\{ b_{\beta(I')}^{I',p}(0) \}_{0 \in I' \subset I}$ with vertices $\{ \Phi_{I',k+1}(p) \}_{0 \in I' \subset I}$. Moreover, the convergence is uniform for $q_0 \in X_{0,I} \cap \mathfrak{J}_{p,\delta,t}$ in the sense that
    \[
    \sup_{q_0 \in X_{0,I} \cap \mathfrak{J}_{p,\delta,t_u}} d\left( \Phi'_{t_u}(W^{0,c}_{q_0,t_u}), \pcal_{I,p,0} \right) \to 0,
    \]
    where $d(\cdot, \cdot)$ denotes the distance in the target projective space with respect to the Fubini-Study metric.

    By induction on $l$, it follows that the image $\Phi'_{t_u}(W_{q_0,t_u})$, as $u \to \infty$, converges to the union of parallelograms $\bigcup_{i \in I} \pcal_{I,p,i}$, where each $\pcal_{I,p,i}$ is defined similarly to $\pcal_{I,p,0}$, uniformly for $q_0 \in X_{0,I} \cap \mathfrak{J}_\delta$. Therefore, for $u$ sufficiently large,
    \[
    d\left( \Phi_{t_u,k+1}(W_{q_0,t_u}), \bigcup_{i \in I} \pcal_{I,p,i} \right) < 2k^n \epsilon.
    \]

    Furthermore, replacing $p$ by any $p_0 \in X_{0,J} \cap U$ and $s_{t,\mathrm{rn}}^{I',p}$ by $s_{t,\mathrm{rn}}^{I',p_0}$, we can define
    \[
    \mathfrak{J}_{p_0,\delta,t} = \left\{ z \in \mathfrak{N}_{p_0} \mid \delta > d(z, D_I) \geq y_t^{-k+3/2} \right\}.
    \]
    For $q_0 \in X_{0,I} \cap \mathfrak{J}_{p_0,\delta,t}$, we can similarly define $W_{q_0,t}$ and show that
    \[
    d\left( \Phi_{t_u,k+1}(W_{q_0,t_u}), \bigcup_{i \in I} \pcal_{I,p_0,i} \right) < 2k^n \epsilon.
    \]
    By the convergence of $\wcal_{I,t_u}^\delta$, we conclude that for any $\epsilon > 0$, there exists $u_0 > 0$ such that for all $u > u_0$,
    \[
    d\left( \Phi_{t_u,k+1}(\wcal_{I,t_u}), \bigcup_{J,\, I \subset J} Y_J \right) < \epsilon.
    \]
    Therefore, $\Phi_{t_u,k+1}(\wcal_{I,t_u})$ converges, as $u \to \infty$, to $\bigcup_{J,\, I \subset J} Y_J$.
\end{proof}

Let $d_D$ denote the distance function to $\Sing(X_0)$, and set $\wcal_t = \{ q \in X_t \mid d_D(q) < y_t^{-k+3/2} \}$. Then, as a direct corollary of the preceding theorem, we have:

\begin{corollary}\label{cor-sing-x0}
    The image $\Phi_{t_u,k+1}(\wcal_{t_u})$, as $u \to \infty$, converges to the union of a collection of varieties $\bigcup_{I,\, X_{0,I} \subset \Sing(X_0)} Y_I$, where each $Y_I$ is the union of $|I|$ varieties $\bigcup_{i \in I} Y_{I,i}$, and each $Y_{I,i}$ is the image of the projective bundle $\zcal_{I,i}$.
\end{corollary}

For $|t|$ sufficiently small, $X_t \setminus \wcal_t$ has $m+1$ connected components $\bigcup_{0 \leq i \leq m} X_{t,i}^W$, where $X_{t,i}^W$ consists of points whose distance to $X_{0,i}$ is less than their distance to any other $X_{0,j}$. For $\delta > 0$, we also have $X_t^\delta = \bigcup_{0 \leq i \leq m} X_{t,i}^\delta$.

\begin{theorem}
    There exists $c > 0$ such that for each $i$ and for each $s \in \bigcup_{i \notin I} \bcal_{I,t}$, we have
    \[
    \sup_{x \in X_t \setminus \wcal_t} \frac{\|s\|^2}{\sum_{s' \in \bcal_{i,t}} \|s'\|^2} < c y_t^{-4}.
    \]
\end{theorem}

\begin{proof}
    Without loss of generality, assume $i = 0$. When $\delta > 0$ is sufficiently small, any point $q \in X_t$ with $d_D(q) < \delta$ lies in some $U \in \ucal$ centered at a point $p \in X_{0,I}$ for some $I$.

    Let $s^{0,p} \in \hcal_{\{0\},t,1}$ be the chosen peak section at $p$ of degree $\{0\}$, and let $f^{0,p} = \sum_\beta b^{0,p}_\beta z^\beta$ be its local representation. Denote by $\beta(0)$ the multi-index with $\beta_0 = 0$ and $\beta_j = 1$ for $j \in I \setminus \{0\}$.

    We have shown that for any $\epsilon > 0$, there exists $\delta > 0$ such that on $U \cap (X_{t,0}^W \setminus X_{t,0}^\delta)$, for $J \subset I$,
    \[
    \left| \frac{|b^{J,p}_{\beta(J)} z^{\beta(J)}|^2}{\sum_{s \in \bcal_{J,t}} |f_s|^2} - 1 \right| < \epsilon,
    \]
    where $f_s$ denotes the local holomorphic function representing $s$.

    On the other hand, for any $J$ with $J \nsubseteq I$ and $s \in \bcal_{J,t}$, we have shown that for sufficiently small $\delta$, $|f_s|^2 = O(y_t^{-2k+1}) d_{p,t}$ on $\mathfrak{N}_p \cap (X_t \setminus X_t^\delta)$. If $0 \notin J$, then since $\sigma_0 > y_t - (2k-3)(|I|-1) \log y_t$ on $X_{t,0}^W$, i.e., $|z_0|^2 < e^{-y_t} y_t^{(2k-3)(|I|-1)}$, we have $|f_s|^2 = O(e^{-\frac{1}{2} y_t}) d_{p,t}$ on $\mathfrak{N}_p \cap (X_{t,0}^W \setminus X_{t,0}^\delta)$.

    This leaves the terms $|b^{J,p}_{\beta(J)} z^{\beta(J)}|^2$ for those $J$ with $0 \in J \subset I$. If $|J| \geq 2$, then
    \[
    |b^{J,p}_{\beta(J)} z^{\beta(J)}|^2 = O\left( y_t^{-(|J|-1)(2k+1) - (|I|-|J|)(2k-3)} \right)
    \]
    on $\mathfrak{N}_p \cap (X_{t,0}^W \setminus X_{t,0}^\delta)$.

    Clearly, on $U \cap (X_{t,0}^W \setminus X_{t,0}^\delta)$,
    \[
    |b^{0,p}_{\beta(0)} z^{\beta(0)}|^2 > c' y_t^{(-2k+3)(|I|-1)}
    \]
    for some $c' > 0$. Therefore,
    \[
    \sup_{x \in \mathfrak{N}_p \cap (X_{t,0}^W \setminus X_{t,0}^\delta)} \frac{\|s\|^2}{\sum_{s' \in \bcal_{0,t}} \|s'\|^2} < c y_t^{-4}
    \]
    for each $s \in \bigcup_{0 \notin I} \bcal_{I,t}$.

    The same argument applies when $p$ is replaced by any $p_0 \in \frac{1}{2} U \cap X_{0,I}$, so
    \[
    \sup_{x \in X_{t,0}^W \setminus X_{t,0}^\delta} \frac{\|s\|^2}{\sum_{s' \in \bcal_{0,t}} \|s'\|^2} < c y_t^{-4}
    \]
    for each $s \in \bigcup_{0 \notin I} \bcal_{I,t}$. For $i \neq 0$, similar estimates hold:
    \[
    \sup_{x \in X_{t,i}^W \setminus X_{t,i}^\delta} \frac{\|s\|^2}{\sum_{s' \in \bcal_{i,t}} \|s'\|^2} < c y_t^{-4}
    \]
    for each $s \in \bigcup_{i \notin I} \bcal_{I,t}$.

    It remains to estimate on $X_{t,0}^\delta$.

When $\delta$ is sufficiently small, the term $b^{i,p} z_i$ dominates with relative error $\epsilon$, so on the complement of $\wcal_{t_u}$, $|z_i|$ is large compared to $y_t^{-\frac{1}{2}(2k+1)}$. Moreover, for small enough $\delta$, the Bergman kernel on $X_i^{\delta}$ is small, and the volume of $X_i \setminus X_i^{\delta}$ is also small. Therefore, the orthonormal basis of $\hcal_{0,i,k+1}$ remains almost orthonormal when restricted to $X_i^{\delta}$; that is, the Bergman kernel of $(X_i^{\delta}, (k+1)L)$ is close to that of $(X_i, (k+1)L)$. By Theorem~\ref{thm-ruan}, the Bergman kernel of $(X_{t,i}^\delta, (k+1)L)$ is similarly close to that of $(X_i^{\delta}, (k+1)L)$. Since the $L^2$-norms of sections $s \in \bigcup_{i \notin I} \bcal_{I,t}$ are $O(y_t^{-k-1/2})$, it follows that on $X_{t,i}^\delta$, the pointwise norms satisfy
\[
\frac{\|s\|^2}{\sum_{s' \in \bcal_{i,t}} \|s'\|^2} = O(y_t^{-2k+1}).
\]

This completes the proof of the theorem.
\end{proof}

This theorem implies that, to study the convergence of $\Phi_{t_u,k+1}(X_{t_u,i}^W)$, we may effectively ignore the contributions from sections in $\bigcup_{i \notin I} \bcal_{I,t}$. Fix an exhaustion of compact subsets $F_\beta \subset\subset X_0 \setminus \Sing(X_0)$ and diffeomorphisms $\phi_{\beta,t}: F_\beta \to X_t$ as in Theorem~\ref{thm-ruan}. It then follows that $\Phi_{t_u,k+1}(\phi_{\beta,t}(F_\beta))$ converges to $\Phi_{0,k+1}(F_\beta)$ as $u \to \infty$. In other words, for any compact subset $\mathfrak{C}$ in the interior of $X_{t_u,i}^W$, the images $\Phi_{t_u,k+1}(\mathfrak{C})$ converge to $\Phi_{0,k+1}(\mathfrak{C})$.

Furthermore, for any $\epsilon > 0$, there exists $\delta > 0$ such that
\[
d\left(\Phi_{t_u,k+1}(X_{t_u,i}^W \setminus X_{t_u,i}^\delta),\, \Phi_{0,i,k+1,1}(X_{t_u,i}^W \setminus X_{t_u,i}^\delta)\right) < \epsilon,
\]
where $\Phi_{0,i,k+1,1}$ is the embedding defined by $\bcal_{\{i\},t,1}$, treating all other sections as zero. It is then straightforward to see that $\Phi_{0,i,k+1,1}(X_{t_u,i}^W \setminus X_{t_u,i}^\delta)$ converges to $\Phi_{0,k+1}(D_{\{i\}})$.

Therefore, we have shown that $\Phi_{t_u,k+1}(X_{t_u,i}^W)$ converges to $\Phi_{0,k+1}(X_{0,i})$ as $u \to \infty$. Combining this with Corollary~\ref{cor-sing-x0}, we have completed the proof of Theorem~\ref{thm-main}.

\

\subsection*{Proof of theorem \ref{thm-2}}
\begin{theorem}[Theorem \ref{thm-2}]
    Let $\rho_{t,k}$ denote the Bergman kernel of $\hcal_{t,k}$. Define
    \[
    \lambda_u(t,k) = \max_{x \in X_t} \rho_{t,k}(x), \qquad
    \lambda_l(t,k) = \min_{x \in X_t} \rho_{t,k}(x).
    \]
    Then for $k$ sufficiently large, there exist constants $c_k > 0$, $c_k' > 0$, and $c_k'' > 0$ such that
    \[
    c_k < \frac{\lambda_u(t,k)}{|\log |t||^{\kappa-1}} < c_k', \qquad
    \lambda_l(t,k) < c_k'' \left(|\log |t||^{-2k+1} a_t^{2k}\right)^{\kappa-1},
    \]
    where $a_t = \log |\log |t||$.
\end{theorem}

Let $\{t_u\}$ be a good sequence. In the following, we continue to use $t$ to denote an element of this sequence.

Let $\rho_{i,k+1}$ denote the Bergman kernel of $\hcal_{0,i,k+1}$, and set
\[
\mathfrak{S} = \max_i \sup_{X_{0,i}} \rho_{i,k+1}.
\]

Fix $\epsilon < \frac{1}{2}$. Then, by Lemma~\ref{lem-epsilon-delta}, there exists $\delta > 0$ such that its conclusion holds. On $X_t^{\delta}$, we have $\rho_{t,k+1} < 2\mathfrak{S}$ for $|t|$ sufficiently small.

For any $I$ with $|I| \geq 2$, we may assume $I = \{0, \ldots, l\}$. For each $U \in \ucal^I$, the set $^\delta U = \{ z \in U \mid d(z, \Sing(X_0)) < \delta \}$ contains a polydisc $\dcal = \{ z \in U \mid |z_i| < \delta_1,\, i \leq l \}$ for some $\delta_1 > 0$. For a smaller $\delta' > 0$, we have $^{\delta'} U \subset \dcal$. Thus, we may assume $\log \delta_1 < -(2k+2)$ and $^\delta U = \dcal$.

Set
\[
\mathfrak{I} = \min_i \inf_{X_{0,i}^{\delta}} \rho_{i,k+1}.
\]
The pointwise norm of the restriction $s_t$ of a section $\tilde{s} = f (dz_0 \wedge \cdots \wedge dz_n)^{\otimes (k+1)}$ on $U^0(t)$ is given by $|f|^2 e^{(k+1)\phi_t} \prod_{i=1}^l \sigma_i^{2k+2}$.

For each $p \in U \cap X_{0,I}$, on $\mathfrak{N}_p^\delta = \mathfrak{N}_p \cap (X_{t_u} \setminus X_{t_u}^{\delta})$, we can use the functions $\{ y_t^{-\frac{|I'|-1}{2}(2k+1)} b_{\beta(I')}^{I',p} z^{\beta(I')} \}_{\emptyset \neq I' \subset I}$ to approximate $\rho_{t,k+1}$. Denote
\[
\mathfrak{N}_p^{\delta,W} = \left\{ z \in \mathfrak{N}_p^\delta \mid |z_i|^2 \geq y_t^{-2k+3} \text{ for some } i \leq l \right\}.
\]

On $U^0(t)\cap \mathfrak{N}_p^{\delta,W}$, we consider the functions
\[
\{F_{I'}\}_{\emptyset\neq I'\subsetneq I},
\]
where
\[
F_{I'} = y_t^{-(|I'|-1)(2k+1)} |b_{\beta(I')}^{I',p} z^{\beta(I')}|^2 \prod_{i=1}^{l} (\sigma_i)^{2k+2}.
\]
For each $I'$ in the index set,
\[
F_{I'} = y_t^{-(|I'|-1)(2k+1)} |b_{\beta(I')}^{I',p}|^2 \prod_{i\in I'} (\sigma_i)^{2k+2} \prod_{i=1,\,i\notin I'}^{l} |z_i|^2 (\sigma_i)^{2k+2}.
\]
Since the function $e^{-x}x^{2k+2}$ on $[0,+\infty)$ attains its unique maximum $e^{-2k-2}(2k+2)^{2k+2}$ at $x=2k+2$, we have
\begin{align*}
    F_{I'} &< y_t^{(|I'|-1)(2k+2-(2k+1))} |b_{\beta(I')}^{I',p}|^{2} [e^{-2k-2}(2k+2)^{2k+2}]^{|I'|-1} \\
    &= y_t^{|I'|-1} |b_{\beta(I')}^{I',p}|^{2} [e^{-2k-2}(2k+2)^{2k+2}]^{|I'|-1}.
\end{align*}

Now, focus on the region $\mathfrak{N}_{p,W} \cap U^0(t)$, where $|z_i|^2 < y_t^{-2k+3}$ for $0\leq i\leq l$. For $I'\subsetneq I$, the same estimate holds:
\[
F_{I'} < y_t^{|I'|-1} |b_{\beta(I')}^{I',p}|^{2} [e^{-2k-2}(2k+2)^{2k+2}]^{|I'|-1}.
\]
For $I'=I$, we have
\[
F_I \leq |b_{\beta(I)}^{I,p}|^{2} y_t^{|I|-1},
\]
and at the point $q_c$ with $\sigma_i(q_c) = \frac{1}{l+1} y_t$ for $1\leq i\leq l$,
\[
F_I(q_c) = |b_{\beta(I)}^{I,p}|^{2} \left[\frac{1}{l+1} y_t\right]^{|I|-1}.
\]
Therefore,
\[
c_1(k)^{-1} < \frac{\sup_{\mathfrak{N}_p^{\delta}} \rho_{t,k+1}}{|b_{\beta(I)}^{I,p}|^{2} y_t^{|I|-1}} < c_1(k)
\]
for some constant $c_1(k)$ independent of $t$ and $U$. Thus,
\[
c_2(k)^{-1} < \frac{\sup_{X_t} \rho_{t,k+1}}{y_t^{\kappa-1}} < c_2(k)
\]
for some constant $c_2(k)$ independent of $t$. This proves the upper and lower bounds for the maximum of the Bergman kernel in Theorem~\ref{thm-2} for a good sequence $\{t_u\}$.

Suppose, for contradiction, that there exists a sequence $\{t_i\}$ such that
\[\frac{\sup_{X_{t_i}} \rho_{t_i,k+1}}{y_{t_i}^{\kappa-1}} \to \infty\]  as $i\to\infty$. Then, by extracting a subsequence which is a good sequence, we obtain a contradiction. Therefore, there exists $c_k$ such that $\frac{\sup_{X_t} \rho_{t,k+1}}{y_t^{\kappa-1}} < c_k$ for $|t|$ sufficiently small. The lower bound follows by the same argument. This establishes the first part of Theorem~\ref{thm-2}.

Next, we estimate the minimum of $\rho_{t,k+1}$ on $\mathfrak{N}_{p,W}$. Analogous to inequality~\ref{e-dpt}, we have
\begin{equation}
    \sum_{0\in J\subset I} y_t^{-(|J|-1)(2k+1)} |b^{J,p}_{\beta(J)}(0)|^2 |z_{\hat{J}}|^2 \leq \max_J |b^{J,p}_{\beta(J)}(0)|^2 \prod_{j=1}^{l} (y_t^{-(2k+1)} + |z_j|^2).
\end{equation}
Thus, we need to estimate the minimum of
\[
f(q) = \prod_{j=1}^{l} (y_t^{-(2k+1)} + |z_j|^2) \sigma_j^{2k+2}.
\]
Let $q_0 \in \mathfrak{N}_{p,W}$ be the point where $|z_j|^2 = y_t^{-(2k+1)}$ for $1\leq j\leq l$. Then
\[
f(q_0) = \left[2 y_t^{-(2k+1)} ((2k+1)\log y_t)^{2k+2}\right]^l.
\]
This implies
\[
\inf_{\mathfrak{N}_p^{\delta}} \rho_{t,k+1} < c_3(k) \max_J |b^{J,p}_{\beta(J)}(0)|^2 \left[y_t^{-(2k+1)} (\log y_t)^{2k+2}\right]^l
\]
for some constant $c_3(k)$ depending only on $k$. Therefore,
\[
\inf_{X_t} \rho_{t,k+1} < c_4(k) \left[y_t^{-(2k+1)} (\log y_t)^{2k+2}\right]^{\kappa-1}
\]
for some constant $c_4(k)$ independent of $t$. This again holds for $t$ in a good sequence $\{t_u\}$, and a similar contradiction argument yields the general result. This completes the proof of Theorem~\ref{thm-2}.
\subsubsection*{Proof that $\pi_*(kL)$ is locally free}

Since each $X_{0,i}$ is smooth, any section \[s \in H^0_0(X_{0,i}, kL)\] extends to a section on $\xcal$. Define $S^1(X_0) = \Sing(X_0)$ and inductively $S^{a+1}(X_0) = \Sing(S^a(X_0))$. Then we have the following decompositions:
\[
H^0(S^1(X_0), kL) = H^0(S^2(X_0), kL) \oplus \bigoplus_{|I|=2} H^0_0(X_{0,I}, kL),
\]
and more generally,
\[
H^0(S^{a}(X_0), kL) = H^0(S^{a+1}(X_0), kL) \oplus \bigoplus_{|I|=a} H^0_0(X_{0,I}, kL).
\]
Since $S^{n}(X_0)$, if nonempty, consists of finitely many points, any section in \[H^0(S^n(X_0), kL)\] can be extended to $\xcal$. By backward induction on $a$, it follows that every section of $H^0(X_0, kL)$ extends to $\xcal$. Therefore, $\pi_*(kL)$ is locally free for sufficiently large $k$.

\bibliographystyle{plain}
\bibliography{references}
\end{document}